\documentclass[letter]{amsart}

\usepackage{amssymb}
\usepackage{amsfonts}
\usepackage{amsmath}
\usepackage{amsthm}
\usepackage{graphicx}
\usepackage{mathrsfs}
\usepackage{dsfont}
\usepackage{amscd}
\usepackage{multirow}
\usepackage{mathpazo}
\usepackage[all]{xy}
\usepackage[scaled]{helvet} 
\usepackage{courier} 
\normalfont
\usepackage[T1]{fontenc}
\usepackage{calligra}
\usepackage{verbatim}
\usepackage[usenames,dvipsnames]{color}
\usepackage[colorlinks=true,citecolor=Violet,linkcolor=blue]{hyperref}

\newcommand{\N}{{\mathds{N}}}
\newcommand{\Z}{{\mathds{Z}}}

\newcommand{\R}{{\mathds{R}}}
\newcommand{\C}{{\mathds{C}}}

\newcommand{\D}{{\mathfrak{D}}}
\newcommand{\A}{{\mathfrak{A}}}
\newcommand{\B}{{\mathfrak{B}}}
\newcommand{\M}{{\mathfrak{M}}}

\newcommand{\Lip}{{\mathsf{L}}}
\newcommand{\Hilbert}{{\mathscr{H}}}

\newcommand{\propinquity}[1]{{\mathsf{\Lambda}^{\#}_{#1}}}
\newcommand{\cBall}[4]{{{#1}\left[#3,#4\right]_{#2}}}
\newcommand{\Kantorovich}[1]{{\mathsf{mk}_{#1}}}
\newcommand{\sigmaKantorovich}[1]{{\mathsf{mk}^\sigma_{#1}}}

\newcommand{\Haus}[1]{{\mathsf{Haus}_{#1}}}

\newcommand{\StateSpace}{{\mathscr{S}}}
\newcommand{\unital}[1]{{\mathfrak{u}{#1}}}
\newcommand{\mongekant}{{Mon\-ge-Kan\-to\-ro\-vich metric}}
\newcommand{\qms}{quantum locally compact metric space}

\newcommand{\pqms}{proper quantum metric space}
\newcommand{\pqpms}{poin\-ted pro\-per quan\-tum me\-tric space}

\newcommand{\Lqcms}{{\JLL} quantum compact metric space}
\newcommand{\qcms}{quantum compact metric space}
\newcommand{\lcqms}{quantum locally compact metric space}

\newcommand{\unit}{1}

\newcommand{\sa}[1]{{\mathfrak{sa}\left({#1}\right)}}

\newcommand{\indicator}[1]{{p_{#1}}}
\newcommand{\compacts}[1]{{\mathcal{K}\left({#1}\right)}}

\newcommand{\PQMS}{{\mathbf{PQMS}}}
\newcommand{\PQMST}{{\mathbf{PQMST}}}

\newcommand{\corner}[2]{{\left[{#2}\middle\vert{#1}\right]}}
\newcommand{\lsa}[3]{{ \mathfrak{sa}\left[{#1}\middle|{#3}\right]_{#2}  }}

\newcommand{\Loc}[3]{{\mathfrak{Loc}\left[ {#1} \middle\vert {#3} \right]_{#2}}}

\newcommand{\Adm}{{\mathrm{Adm}}}

\newcommand{\JLL}{Lei\-bniz}

\newcommand{\dom}[1]{{\operatorname*{dom}\left({#1}\right)}}
\newcommand{\codom}[1]{{\operatorname*{codom}\left({#1}\right)}}
\newcommand{\diam}[2]{{\mathrm{diam}\left({#1},{#2}\right)}}
\newcommand{\norm}[2]{{\left\|{#1}\right\|_{#2}}}
\newcommand{\tunnelset}[4]{{\text{\calligra Tunnels}\,\left[{#3}\stackrel{{#1}}{\rightarrow}{#4}\middle\vert #2 \right]}}

\newcommand{\Jordan}[2]{{{#1}\circ{#2}}} 
\newcommand{\Lie}[2]{{\left\{{#1},{#2}\right\}}} 

\newcommand{\targetsettunnel}[3]{{\mathfrak{t}_{#1}\left({#2}\middle\vert{#3}\right)}}

\newcommand{\liftsettunnel}[3]{{\mathfrak{l}_{#1}\left({#2}\middle\vert{#3}\right)}}

\newcommand{\LipLoc}[3]{{\mathfrak{LipLoc}\left[ {#1} \middle\vert {#3} \right]_{#2}}}

\newcommand{\tunnelextent}[2]{{\chi\left({#1}\middle\vert{#2}\right)}}

\newcommand{\alg}[1]{{\mathfrak{#1}}}
\newcommand{\almostsubseteq}[2]{\;{\subseteq^{#1}_{#2}}\;}
\newcommand{\Set}[2]{{\left\{\begin{aligned} #1 \end{aligned} \middle\vert\begin{aligned} #2 \end{aligned}\right\}}}
\newcommand{\set}[2]{{\left\{{#1}\middle\vert{#2}\right\}}}

\newcommand{\closure}[1]{{\mathrm{cl}\left[{#1}\right]}}

\newcommand{\vertiii}[1]{{\left\vert\kern-0.25ex\left\vert\kern-0.25ex\left\vert #1\right\vert\kern-0.25ex\right\vert\kern-0.25ex\right\vert}}

\theoremstyle{plain}
\newtheorem{theorem}{Theorem}[subsection]

\newtheorem{corollary}[theorem]{Corollary}

\newtheorem{lemma}[theorem]{Lemma}
\newtheorem{proposition}[theorem]{Proposition}

\newtheorem{theorem-definition}[theorem]{Theorem-Definition}
\newtheorem{proposition-definition}[theorem]{Proposition-Definition}

\theoremstyle{definition}
\newtheorem{definition}[theorem]{Definition}

\newtheorem{notation}[theorem]{Notation}

\newtheorem{hypothesis}[theorem]{Hypthesis}

\theoremstyle{remark}

\newtheorem{example}[theorem]{Example}

\newtheorem{remark}[theorem]{Remark}

\renewcommand{\geq}{\geqslant}
\renewcommand{\leq}{\leqslant}

\numberwithin{equation}{subsection}

\begin{document}

\title[Topographic Gromov-Hausdorff quantum Hypertopology]{Topographic Gromov-Hausdorff quantum Hypertopology for Quantum Proper Metric Spaces}
\author{Fr\'{e}d\'{e}ric Latr\'{e}moli\`{e}re}
\email{frederic@math.du.edu}
\urladdr{http://www.math.du.edu/\symbol{126}frederic}
\address{Department of Mathematics \\ University of Denver \\ Denver CO 80208}

\date{\today}

\begin{abstract}
We construct a topology on the class of {\pqpms s} which generalizes the topology of the Gromov-Hausdorff distance on proper metric spaces, and the topology of the dual propinquity on {\Lqcms s}. A {\pqpms} is a special type of {\lcqms} whose topography is proper, and with properties modeled on {\Lqcms s}, though they are usually not compact and include all the classical proper metric spaces. Our topology is obtained from an infra-metric which is our analogue of the Gromov-Hausdorff distance, and which is null only between iso\-metri\-cally iso\-morphic {\pqpms s}. Thus, we propose a new framework which extends noncommutative metric geometry, and in particular noncommutative Gro\-mov-Haus\-dorff topology, to the realm of quantum locally compact metric spaces.
\end{abstract}

\maketitle
\tableofcontents


\section{Introduction}

Gromov's distance between pointed locally compact metric spaces \cite{Gromov81} is a powerful tool of metric geometry and has found many remarkable applications \cite{Gromov}. Motivated by the apparent potential of importing this versatile instrument in the realm of quantum spaces, and inspired by the pioneering work of Connes \cite{Connes97, Connes} in noncommutative metric geometry, Rieffel introduced the quantum Gromov-Hausdorff distance \cite{Rieffel00}, a generalization of Gromov's distance to the class of \emph{compact} quantum metric spaces \cite{Rieffel98a,Rieffel99}. Yet, Gromov's distance was originally introduced for locally compact metric spaces and found many applications within this general setting, and a further generalization of his distance to non-compact quantum metric spaces proved a very challenging task. This is the task we propose to undertake in this paper. We hence shall define a topology on a class of {\lcqms s}, induced by an inframetric which we call the topographic Gromov-Hausdorff propinquity (or topographic propinquity in short). Following the tradition in topology, we call this topology on a class of spaces a hypertopology.

We thus propose a new framework which encompass, at once, the notions of quantum Gromov-Hausdorff convergence for quantum compact metric spaces, and the notion of Gromov-Hausdorff convergence between locally compact metric spaces, by constructing a generalization of the dual Gromov-Hausdorff propinquity \cite{Latremoliere13b} and the associated hypertopology on a large subclass of {\lcqms s} \cite{Latremoliere12b}, which we call the class of {\pqpms s}. Thus, all examples of convergence for the dual propinquity and for Gromov's original distance are also examples of convergence for our new metric. The study of further examples and the properties of our new hypertopology constitute the future of our project, while this paper focuses on the construction of our new hypertopology. We hope that our present work will find useful applications, in particular, wherever metric considerations occur in mathematical physics where the underlying space is not-compact, such as physics on the Moyal plane.

Noncommutative metric geometry proposes to study certain classes of noncommutative topological algebras as generalizations of algebras of Lipschitz functions over metric spaces. Thus, our work is based on the extension to the noncommutative world of the following classical picture. Let $(X,\mathsf{m})$ be a metric space. For any function $f : X\rightarrow\C$,  we define the Lipschitz constant of $f$ as:
\begin{equation}\label{Lipschitz-seminorm-eq}
\mathsf{Lip}_{\mathsf{m}} (f) = \left\{\frac{|f(x)-f(y)|}{\mathsf{m}(x,y)} : x,y \in X, x\not=y\right\}\text{,}
\end{equation}
allowing for $\mathsf{Lip}_{\mathsf{m}}$ to be infinite. The function $\mathsf{Lip}_{\mathsf{m}}$ is a seminorm on the space of Lipschitz functions $\mathfrak{L} = \Set{f : X\rightarrow\C}{\mathsf{Lip}_{\mathsf{m}}(f) < \infty }$. In particular, if $(X,\mathsf{m})$ is locally compact, then $\mathfrak{L}_0 = \Set{f \in C_0(X)}{\mathsf{Lip}_{\mathsf{m}}(f) < \infty }$ is a dense C*-subalgebra of the C*-algebra $C_0(X)$ of $\C$-valued continuous functions which vanish at infinity. Now, unless stated otherwise, we will only work with the restriction of $\mathsf{Lip}_{\mathsf{m}}$ to $\sa{\mathfrak{L}_0} = \sa{\C_0(X)}\cap\mathfrak{L}_0$, which is a dense subspace of the self-adjoint part $\sa{C_0(X)}$ of $C_0(X)$: the reason for this particular restriction will become apparent shortly.

The dual of the Lipschitz seminorm $\mathsf{Lip}_{\mathsf{m}}$ induces an extended metric $\Kantorovich{\mathsf{Lip}_{\mathsf{m}}}$ on the space $\StateSpace(C_0(X))$ of Radon probability measures over $X$, named the {\mongekant} \cite{Kantorovich40,Kantorovich58}. Kantorovich observed in \cite{Kantorovich40} that if $(X,\mathsf{m})$ is compact, then $\Kantorovich{\mathsf{Lip}_{\mathsf{m}}}$ is actually a metric whose topology is the restriction of the weak* topology of $\StateSpace(C(X))$, with $C(X)$ the C*-algebra of continuous $\C$-valued functions over $X$. Furthermore in this case, the embedding of $X$ in $\StateSpace(C(X))$ which maps to every point the Dirac probability measure at this point, becomes an isometry from $(X,\mathsf{m})$ into $(\StateSpace(C(X)),\Kantorovich{\mathsf{Lip}_{\mathsf{m}}})$. Thus, when $(X,\mathsf{m})$ is compact then the Lipschitz seminorm encodes all the metric information at the level of the C*-algebra of continuous functions.

This central observation prompted Rieffel to define in \cite{Rieffel98a, Rieffel99,Rieffel05} a compact quantum metric space as a pair $(\A,\Lip)$ of an order-unit space $\A$ and a densely-defined seminorm $\Lip$ on $\A$ such that the extended metric:
\begin{equation*}
\Kantorovich{\Lip}: \varphi,\psi \in \StateSpace(\A) \longmapsto \sup\left\{ |\varphi(a) - \psi(a) | : a\in \A, \Lip(a) \leq 1 \right\}
\end{equation*}
is in fact a metric for the weak* topology of the state space $\StateSpace(\A)$ of $\A$. When $(\A,\Lip)$ is a compact quantum metric space, the seminorm $\Lip$ is called a Lip-norm. We observe that Rieffel originally worked with order-unit spaces, thus no multiplicative structure was involved in this definition. Examples of compact quantum metric spaces include quantum tori \cite{Rieffel98a, Rieffel02},  finite dimensional C*-algebras \cite{Rieffel01}, group C*-algebras for hyperbolic groups endowed with a word metric \cite{Ozawa05}, among others. In particular, the examples coming from hyperbolic groups and their word metrics were introduced by Connes \cite{Connes89}, where noncommutative metric geometry begun. In general, a spectral triple does allow to define an extended distance on the state space of the associated C*-algebras, though no general condition has been established yet to determine whether this metric induces the weak* topology on the state space.

Rieffel then constructed a first extension of the Gromov-Hausdorff distance to the class of compact quantum metric spaces, aimed at providing a framework for some approximations results in mathematical physics (e.g. \cite{Connes97}) and at introducing tools from metric geometry to noncommutative geometry. The Gromov-Hausdorff distance between two compact metric spaces $(X,\mathsf{m}_X)$ and $(Y,\mathsf{m}_Y)$ is defined as the infimum of the Hausdorff distance between $\iota_X(X)$ and $\iota_Y(Y)$ in any compact metric space $(Z,\mathsf{m})$ for any two isometries $\iota_X$ and $\iota_Y$ from $(X,\mathsf{m}_X)$ and $(Y,\mathsf{m}_Y)$ into $(Z,\mathsf{m}_Z)$, respectively \cite{Gromov, Edwards75}. In fact, one may even restrict one's attention to admissible metrics on the disjoint union $X\coprod Y$, i.e. metrics for which the canonical injections of $X$ and $Y$ into $X\coprod Y$ are isometries.

In this context, we note that if $(Z,\mathsf{m}_Z)$ is a metric and $\iota_X: X\hookrightarrow Z$ is an isometry, and if $f : X\rightarrow\R$ is a Lipschitz function, then McShane's theorem \cite{McShane34} provides a function $g : Z \rightarrow\R$ whose Lipschitz constant equals the Lipschitz constant of $f$. Therefore, we see that $\mathsf{Lip}_{\mathsf{m}_X}$ is the quotient seminorm of $\mathsf{Lip}_{\mathsf{m}_Z}$ for the surjection $\iota_X^\ast: f \in C(Z) \twoheadrightarrow f\circ\iota_X$. Conversely, if $\mathsf{Lip}_{\mathsf{m}_X}$ is the quotient of $\mathsf{Lip}_{\mathsf{m}_Z}$ for $\iota_X^\ast$, then the map $\iota_X$ is easily seen to be an isometry. Thus, the dual notion of isometry is the notion of quotient of Lipschitz seminorms. We note that this relation relies completely on our choice to only work with real-valued Lipschitz functions. Indeed, as seen for instance in \cite{Rieffel10}, if $\iota_X : X\hookrightarrow Z$ is an isometry, then any Lipschitz $\C$-valued function $f$ on $(X,\mathsf{m}_X)$ with Lipschitz constant $L$ extends to a Lipschitz function on $(Z,\mathsf{m}_Z)$, though in general the infimum of the Lipschitz constant of all possible such extensions is strictly greater than the Lipschitz constant of $f$ on $X$ (the best one may hope in general is $\frac{4}{\pi}L$). This remark is the reason why we generally work with seminorms on dense subspaces of the self-adjoint part of C*-algebras rather than seminorms on dense subspaces of C*-algebras.

The construction of the quantum Gromov-Hausdorff distance proceeds naturally by duality from these observations. For any two compact quantum metric spaces $(\A,\Lip_\A)$ and $(\B,\Lip_\B)$, a Lip-norm on $\A\oplus\B$ is admissible when its quotient for the canonical surjections on $\A$ and $\B$ are, respectively, $\Lip_\A$ and $\Lip_\B$. The infimum of the Hausdorff distances between $\StateSpace(\A)$ and $\StateSpace(\B)$, identified with their isometric copies in $(\StateSpace(\A\oplus\B),\Kantorovich{\Lip})$, over all possible admissible Lip-norms $\Lip$, is the quantum Gromov-Hausdorff distance $\mathsf{dist}_q((\A,\Lip_\A),(\B,\Lip_\B))$ between $(\A,\Lip_\A)$ and $(\B,\Lip_\B)$.

The quantum Gromov-Hausdorff distance then allows the construct finite dimensional approximations of quantum tori by matrix algebras \cite{Latremoliere05}, and finite dimensional approximations of the C*-algebras of continuous functions on co-adjoint orbits of compact Lie groups \cite{Rieffel01,Rieffel09,Rieffel10c}, as well as to establish various continuity of families of compact quantum metric spaces. Recent research focuses on the continuity of certain C*-algebraic related structures such as projective modules \cite{Rieffel06,Rieffel09,Rieffel10,Rieffel10c}, which in turn raises the question of how to encode more of the C*-algebraic structure, a point to which we shall shortly return.

The matter of extending the full construction of Gromov, which was based on locally compact metric spaces, to the noncommutative realm, or in other terms, to extend Rieffel's construction to the non-compact setting, is a serious challenge, due to a series of important difficulties which we now address. The first issue is that the {\mongekant} is not quite as well behaved for noncompact metric spaces. 

Let $(X,\mathsf{d}_X)$ be a locally compact metric space and $\mathsf{Lip}_X$ be the associated Lipschitz constant, defined by Expression (\ref{Lipschitz-seminorm-eq}), and restricted to real-valued functions. The natural C*-algebraic dual object of the locally compact space $X$ is the C*-algebra $C_0(X)$ of $\C$-valued continuous functions vanishing at infinity on $X$, which is not a unital C*-algebra. In particular, its state space $\StateSpace(C_0(X))$ is not a weak* compact subset of the dual of $C_0(X)$, nor is it weak* locally compact. The {\mongekant} $\Kantorovich{\mathsf{Lip}_X}$ defined as the extended metric induced on $\StateSpace(C_0(X))$ by the dual of $\mathsf{Lip}_X$ is typically not a metric, as it may take the value $\infty$, and it does not generally metrizes the restriction of the weak* topology of $\StateSpace(C_0(X))$. Thus, the classical model on which the theory of compact quantum metric spaces is based does not directly extend to the noncompact setting \cite{Latremoliere12b}.

We proposed a first attempt at addressing these issues by replacing the {\mongekant} which a noncommutative analogue of the bounded-Lipschitz distance in \cite{Latremoliere05b}. We were able to obtain a very natural characterization of the densely-defined seminorms on separable C*-algebras whose associated bounded-Lipschitz metric metrize the weak* topology of the state space, in the spirit of \cite{Rieffel98a}, and this matter already included the introduction of a new topology on C*-algebras which is typically weaker than the strict topology, yet is the correct topology to consider for our metric considerations. We also note that our theory included Rieffel's compact quantum metric spaces as special cases, when the underlying C*-algebras were unital. 

Our first attempt introduced some very valuable concepts which we shall retrieve in our current work, yet we felt that it would prove interesting to study the geometry of the {\mongekant} in the noncommutative setting as well. Our second work on this subject in \cite{Latremoliere12b} proposed a solution for this problem. It is based upon a key observation by Dobrushin \cite{Dobrushin70} regarding the {\mongekant} on general metric spaces. While the {\mongekant} does not metrize the weak* topology on arbitrary subsets of Radon probability measures, it does so on sets which possess a strong form of the uniform tightness property, related to the underlying metric rather than the underlying topology. We call this property Dobrushin tightness. It is a major result regarding the topology of the {\mongekant}, and it first appears in relation with noncommutative metric geometry in \cite{Bellissard10}. However, using this result in the noncommutative setting is itself quite challenging, as it is not at all clear how to generalize the notion of Dobrushin tight sets from Radon probability measures to more general states of noncommutative C*-algebras.

The solution which we offer in \cite{Latremoliere12b} relies on the introduction of a new component to the signature of a quantum metric space: to a pair $(\A,\Lip)$ of a C*-algebra and a densely defined seminorm $\Lip$ on the self-adjoint part of $\A$, we add an Abelian C*-subalgebra $\M$ of $\A$ which contains an approximate unit for $\A$, which we call the topography of $(\A,\Lip,\M)$. The interpretation of the topography is that it provides a set of observables allowing for both local notions and notions of large scale geometry, such as Dobrushin tightness (or more specially, what we call tameness, which is a closely related concept) to be studied. Based on the notion of topography and the associated analogue of Dobrushin tightness, we proposed and characterized a notion of {\lcqms s} in a manner which extends Rieffel's notion of {\qcms s}.

Equipped with a suitable notion of {\lcqms s}, we undertake in this paper the task of generalizing Gromov's distance to an appropriate class of {\lcqms s}. Gromov's distance is a metric up to isometry when restricted to proper metric spaces, which are locally compact metric spaces whose closed balls are all compact. This notion fits quite well in our approach to {\lcqms s}, and we shall thus introduce a notion of {\pqms s}. The notion of a pointed metric space is replaced by a notion of a {\lcqms s} with a choice of a state which restricts to a Dirac probability measure on the topography. Thus, defining our class of appropriate {\pqpms s} is a natural process from the work we did in \cite{Latremoliere12b}.

The challenge of defining a hypertopology on the class of {\pqpms s} is caused by two issues which interact to make things difficult. The first issue has been a matter of much research \cite{Kerr02,Li03,Li05} including of our own \cite{Latremoliere13,Latremoliere13b,Latremoliere13c}: Rieffel's definition of the quantum Gromov-Hausdorff distance does not capture the C*-algebraic structure, and thus in particular, distance zero does not imply that underlying C*-algebras are *-isomorphic. This is a source of concern for two reasons. First of all, recent research in noncommutative metric geometry focuses on the study of continuity property for C*-algebra-related structures such as projective modules, and thus would likely prefer to work with a notion of metric convergence which behave well with respect to the C*-algebraic structure. Ensuring that distance zero implies *-isomorphism is a mean to test that the distance encodes enough information to recover the C*-algebraic structure, though in fact more is desirable. What one wants is a metric defined only using C*-algebras and Lip-norms which are in some sense well-behaved with respect to the multiplication. In particular, the Lipschitz seminorms satisfy the Leibniz inequality, and many examples of {\lcqms s} come with this property as well. 

To devise a generalized Gromov-Hausdorff distance which satisfies the above requirements proved elusive for a while. Rieffel introduced the quantum proximity \cite{Rieffel10c} when seeking for such a distance, though the proximity is not known to be a metric. A core difficulty is that the quotient of Lip-norms with the Leibniz property may fail to be Leibniz \cite{Blackadar91}. Earlier attempts at making *-isomorphism a necessary condition for distance zero all relied on replacing the state spaces with their {\mongekant} by more complex structures \cite{Kerr02,Li03}, and did not exploit any connection between Lip-norms and C*-algebraic structure. The first metric to exploit such a connection, in the form of the Leibniz property for Lip-norms, was our own quantum Gromov-Hausdorff propinquity \cite{Latremoliere13}. We then proposed a metric which only involves the {\mongekant} on state spaces, and exploit the Leibniz property, and satisfied all the desired constraints above: the dual Gromov-Hausdorff propinquity \cite{Latremoliere13b}, which is also a complete metric. We devised our new metrics in part due to the desire to have a well-behaved metric with respect to C*-algebras to work with when dealing with {\lcqms s}, whose very definition invoke the C*-algebraic structure, and in prevision that the future of noncommutative metric geometry will likely require such desirable properties in both the compact and locally compact setup. Our recent work on this subject \cite{Latremoliere14} strengthens the argument that the dual propinquity is a good candidate for a noncommutative C*-algebraic extension of the Gromov-Hausdorff distance in the compact case.

This first difficulty meets another matter specific to the noncompact setting in which we now work. The key condition of admissibility for Lip-norms, first introduced by Rieffel \cite{Rieffel00} as a mean to encode the notion of isometric embedding, is a noncommutative analogue of the extension theorem of McShane \cite{McShane34} for Lipschitz functions, as we discussed earlier in this introduction. However, when working in the noncompact setting, some care must be taken in applying McShane's result. While a Lipschitz function $f : X\rightarrow \R$ can still be extended to a Lipschitz function $g : Z \rightarrow\R$ if $X$ is a subspace of $Z$, in such a way that the Lipschitz constant of $g$ (and the sup-norm of $g$) is the same as the Lipschitz constant of $f$ (and the sup-norm of $f$), it is not always true that if $f$ vanishes at infinity, so does its extension $g$. Moreover, our work on {\lcqms s} in \cite{Latremoliere12b} led us to a characterization which involves compact subsets of the topographies. Thus, we want to work with compactly supported Lipschitz functions, and we wish to extend them to larger spaces while keeping some control on the size of the support. This is not a completely straightforward issue. While easy enough to understand in the classical picture, this problem does impose a new meaning for the notion of admissibility in the noncompact case. We believe that this feature of the locally compact setting is a surprising departure from the methods employed when defining noncommutative analogues of the Gromov-Hausdorff distance for {\qcms s}: a main contribution of this paper is indeed to find an appropriate notion to replace Rieffel's admissibility, i.e. the use of isometric embeddings, which turns out to lead us in a somewhat surprising direction.

These two main sources of problems interact to make any construction of a Gromov-Hausdorff topology on {\pqpms s} somewhat obscure: for instance, one must avoid the recourse to quotients of Leibniz Lip-norms, and invoke instead methods such as the ones found in \cite{Latremoliere13,Latremoliere13b,Latremoliere13c,Latremoliere14} if one wishes for some form of triangle inequality to hold when introducing a metric between {\pqpms s}. Yet we also must ensure that these methods may be applied, which constrained the notion of admissibility. In general, the outcome of our efforts as summarized in this paper is the construction of an inframetric, rather than a metric, i.e. a quantity which satisfies all the axioms of a metric except for a modified version of the triangle inequality, which involves a factor $2$ in our case. This is however enough to define a reasonable topology on the class of {\pqpms}, in the sense that the following desirable properties hold:
\begin{enumerate}
\item our hypertopology is characterized fully by its convergent sequences, and it is defined by an inframetric which we call the topographic Gromov-Hausdorff propinquity,
\item our hypertopology is ``Hausdorff modulo isometric isomorphism'': two {\pqpms s} which are limits of the same sequence for our topology must have their underlying C*-algebras *-iso\-morphic and their metric structures isometric,
\item the restriction of our hypertopology to the subclass of {\qcms s} is the same as the topology given by the dual Gro\-mov-Haus\-dorff propinquity; 
\item the restriction of our hypertopology to the subclass of classical pointed proper metric spaces is a weaker topology than the Gro\-mov-Haus\-dorff distance topology --- as with the quantum Gromov-Hausdorff distance or the dual propinquity, our use of the state space, rather than the pure state space, in our definition of convergence, weakens the inframetric we build compared with the original Gromov construction (we essentially define a form of Gromov-Kantorovich-Hausdorff distance).
\end{enumerate}

Thus, we meet the basic requirements for an extension of Gromov's topology to the noncommutative setting, and incorporate all known examples of convergences for both {\Lqcms s} and for classical proper metric spaces. Moreover, we strive for a minimal theory. All our choices of definition in this paper are designed to provide a specific and needed tool for the main argument of the paper, which is the separation property of our new hypertopology; we avoid making any more demand than what is the bare necessity imposed to us by our approach to {\lcqms s}. We propose a framework which is quite adaptable, following the model we initiated in \cite{Latremoliere13,Latremoliere13b}: we define, in fact, a family of hypertopologies which are induced by classes of what we call tunnels, under appropriate conditions. This flexibility is important as it will allow to strengthen the topology if desired, depending on future developments of the field.

Our paper begins with an exposition of the notion of Gromov-Hausdorff convergence for pointed, proper metric spaces suitable for our generalization to the noncommutative setting. The material in this section is presented as we wish to make very clear the connection between our work and the distance of Gromov in \cite{Gromov81}. The second section of this paper then presents a summary of the foundations of the theory of {\lcqms s}, as we proposed in our earlier research \cite{Latremoliere12b}, in preparation for the present paper. This is the section where we introduce the notion of {\pqpms s}, and present an alternative characterization of {\lcqms s}. These two first sections thus lay the framework within which we will work.

We then introduce the topographic Gromov-Hausdorff propinquity, our generalization of Gromov's distance, in the next two sections. To this end, we first define a form of correspondence between {\pqpms s}, which we call passages, and then define a measurement of the ``distortion'' of a passage for a given radius. This leads us to introduce tunnels and their target sets, from which the local propinquity is computed. The topographic propinquity is an inframetric obtained from the local propinquity, and the basis for the definition of our hypertopology.

The fourth part of this paper proves the main theorem, which is that our hypertopology is separated modulo pointed isometric isomorphism of {\pqpms s}. We conclude with the comparison of our new topology and the Gromov-Hausdorff topology for classical proper metric spaces, as well as with the topology of the dual Gromov-Hausdorff propinquity on {\Lqcms s}.

\section{Gromov-Hausdorff Convergence for Pointed Proper Metric Spaces}

We begin our paper with the classical picture which we wish to generalize. Our purpose is to give a presentation of the Gromov-Hausdorff distance suitable for its noncommutative generalization; in particular we will make an observation regarding the extension of Lipschitz functions with compact support which will play an important role in our construction of the propinquity.

We begin with two notations which we will employ throughout this paper. First, we introduce a notation for closed balls in a metric space:

\begin{notation}
Let $(X,\mathsf{d})$ be a metric space, $x_0\in X$ and $r\geq 0$. The closed ball:
\begin{equation*}
\Set{x\in X }{\mathsf{d}(x,x_0)\leq r}
\end{equation*}
is denoted by $\cBall{X}{\mathsf{d}}{x_0}{r}$. When the context is clear, we simply write $\cBall{X}{}{x_0}{r}$ for $\cBall{X}{\mathsf{d}}{x_0}{r}$.
\end{notation}

Second, when working with Gromov-Hausdorff distance, we will often use the following notion of approximate inclusion:

\begin{notation}
Let $A,B \subseteq Z$ be subsets of a metric space $(Z,\mathsf{d})$. We write $B\almostsubseteq{Z,\mathsf{d}}{\varepsilon} A$ when:
\begin{equation*}
B\subseteq \cBall{Z}{\mathsf{d}}{A}{\varepsilon} = \bigcup_{a\in A}\cBall{Z}{\mathsf{d}}{a}{\varepsilon}\text{.}
\end{equation*}
When the context is clear, we may simply write $B\subseteq_\varepsilon A$ for $B\almostsubseteq{Z,\mathsf{d}}{\varepsilon} A$.
\end{notation}

\subsection{Local Hausdorff Convergence}

Gromov introduced in \cite{Gromov81} the Gro\-mov-Haus\-dorff distance, a far-reaching generalization of the Hausdorff distance \cite{Hausdorff} between compact subsets of a metric space to a pseudo-metric on the class of pointed locally compact metric spaces. This pseudo-metric becomes a metric up to isometry on the class of pointed proper metric spaces, and was the foundation for much research in metric geometry \cite{Gromov,Gromov81}. An earlier work by Edwards \cite{Edwards75} introduced the notion of distance between arbitrary compact metric spaces, but a great leap was taken by Gromov, both by working in the more challenging setting of locally compact spaces, and by the profound use of this new metric to solve difficult problems in group theory and group geometry. Many examples of applications for metric convergence can be found in \cite{Gromov}. It is our purpose in this paper to extend Gromov's notion of convergence for pointed proper metric spaces to the noncommutative setting. This subsection and the next surveys the construction due to Gromov.

The Hausdorff distance between two closed subsets of a metric space, introduced by Hausdorff in \cite[p. 293]{Hausdorff}, can be described as follows:

\begin{definition}\label{Haus-def}
Let $(Z,\mathsf{d})$ be a metric space. The Hausdorff distance between two closed subsets $A,B \subseteq Z$ is defined as:
\begin{equation*}
\begin{split}
\Haus{Z,\mathsf{d}}(A,B) &= \inf \Set{\varepsilon > 0}{A\almostsubseteq{(Z,\mathsf{d})}{\varepsilon} B \text{ and } B\almostsubseteq{(Z,\mathsf{d})}{\varepsilon} A} \\
&= \inf \Set{\varepsilon > 0}{A\subseteq \cBall{Z}{\mathsf{d}}{B}{\varepsilon} \text{ and }B\subseteq\cBall{Z}{\mathsf{d}}{A}{\varepsilon} }\text{.}
\end{split}
\end{equation*}
When the context is clear, we may write $\Haus{\mathsf{d}}$ instead of $\Haus{Z,\mathsf{d}}$.
\end{definition}

The Hausdorff distance between two arbitrary closed subsets of a metric space may be infinite, so Definition (\ref{Haus-def}) introduces an extended metric, properly speaking. Of course, the Hausdorff distance is always finite between bounded closed sets, and is particularly well-behaved when the working on the class of closed subsets of a compact set, which it endows with a compact topology \cite{burago01}. 

When working with the class of all closed subsets of a locally compact metric space, however, a weaker topology than the Hausdorff topology may be desirable. For instance, if we consider a family of circles in the plane, all tangent to some fixed line $L$ at a fixed point $x$, and whose radii converge to infinity, and if we were to stand as an observer at the point $x$, then we would like to see this family of circles converging to the line $L$. This notion does rely on the choice of a base point, and is local in nature. Note that the Hausdorff distance between the line $L$ and any circle is always infinite.

We are thus led to consider a local form of Hausdorff convergence between pointed subspaces of a metric space, based upon the behavior of closed balls centered around a reference point. The theory of the Hausdorff distance suggests that we will obtain a well-behaved theory for a local form of Hausdorff convergence if we assume that closed balls are always compact, and we are thus led to introduce:

\begin{definition}
A \emph{proper metric space} $(Z,\mathsf{d})$ is a metric space whose closed balls are compact. A \emph{pointed proper metric space} $(Z,\mathsf{d},z)$ is a proper metric space $(Z,\mathsf{d})$ and a point $z\in Z$, which we shall refer to as the \emph{base point} of $(Z,\mathsf{d},z)$.
\end{definition}

We note that a proper metric space is always complete and is either compact or of infinite diameter, and that a closed subspace of a proper metric space is proper. Sometimes, the term ``boundedly compact metric space'' is used as a synonym for proper metric space. Moreover, our notion of {\lcqms s} \cite{Latremoliere12b} was designed to accommodate a generalization of proper metric spaces for the purpose of this paper, as we shall see in the next section. 

A first natural attempt at defining a weakened Hausdorff topology on the class of pointed closed subspaces of a locally compact metric space would be to simply take the distance between the closed balls around the point of reference for each possible radius: if $(X,x_0)$ and $(Y,y_0)$ were two pointed metric subspaces of a proper metric space $(Z,\mathsf{d})$, then we could define for all $r > 0$:
\begin{equation}\label{wrong-def-eq}
\inf\Set{\varepsilon > 0}{
\mathsf{d}(x_0,y_0)&\leq\varepsilon\text{ and }\\
\cBall{X}{}{x_0}{r}&\almostsubseteq{Z,\mathsf{d}}{\varepsilon} \cBall{Y}{}{y_0}{r} \text{,}\\
\cBall{Y}{}{y_0}{r}&\almostsubseteq{Z,\mathsf{d}}{\varepsilon} \cBall{X}{}{x_0}{r}}
\end{equation}
and then consider the topology on the class of all pointed subspaces of $(Z,\mathsf{d})$ generated by the family of pseudo-metric given by Expression (\ref{wrong-def-eq}) for all $r > 0$.

However, the following example shows that this idea fails to capture our intuition and is not even compatible with Hausdorff convergence for compact spaces:

\begin{example}\label{simple-compact-cv-example}
For all $n\in\N$ we set $I_n = \left\{0,\frac{1}{n+1}+2\right\}$ and $I=\left\{0,2\right\}$. Note that $(I_n)_{n\in\N}$ is a sequence of compacts in $\R$ which converge to $I$ for the Hausdorff distance. We choose $0$ as our base point for $I$ and all $I_n$, $n\in\N$. Then, the sequence consisting, for all $n\in\N$, of the closed ball $\cBall{I_n}{}{0}{2} = \left\{0\right\}$, does not converge in the Hausdorff distance to $\cBall{I}{}{0}{2} = I$ --- in fact the Hausdorff distance is always $2$.
\end{example}

The previous example suggests that we may need to relax our notion of convergence somewhat still, by recognizing that given two compact subspaces $X$ and $Y$ of a metric space $(Z,\mathsf{d})$, and given $x_0 \in X$, $y_0\in Y$, and $r > 0$, the Hausdorff distance between $X$ and $Y$ in $Z$ does not bound above the Hausdorff distance between the balls $\cBall{X}{}{x_0}{r}$ and $\cBall{Y}{}{y_0}{r}$, but instead informs us how near $\cBall{X}{}{x_0}{r}$ is from $Y$, and similarly for $\cBall{Y}{}{y_0}{r}$ and $X$. We thus modify the pseudo-metrics in Equation (\ref{wrong-def-eq}) to obtain the following quantities:

\begin{definition}\label{delta-r-def}
Let $X,Y\subseteq Z$ be two subsets of a metric space $(Z,\mathsf{d})$ and let $x_0\in X$, $y_0\in Y$. For any $r > 0$, we define:
\begin{equation*}
\delta_r^{(Z,\mathsf{d})}( (X,x_0) ,(Y,y_0) ) = \inf\Set{\varepsilon > 0}{
\mathrm{d}(x_0,y_0) &\leq \varepsilon\text{ and }\\
\cBall{X}{}{x_0}{r} &\almostsubseteq{Z,\mathsf{d}}{\varepsilon} Y\text{,}\\
\cBall{Y}{}{y_0}{r}&\almostsubseteq{Z,\mathsf{d}}{\varepsilon} X \text{.}\\
}
\end{equation*}
\end{definition}

We wish to endow the class of closed subspaces of a proper metric space with the topology generated by the ``open balls'' for all $\delta_r^{(Z,\mathsf{d})}$ $(r\geq 0)$ --- note however that $\delta_r^{(Z,\mathsf{d})}$ is not a pseudo-metric in general, though we shall see how to construct an inframetric from these quantities later on in this paper. To understand this form of convergence, we prove the following theorem which provides several characterizations of the quantities introduced in Definition (\ref{delta-r-def}). We will use the following notation throughout this paper:

\begin{notation}
The set of compact subsets of a topological space $Z$ is denoted by $\compacts{Z}$.
\end{notation}

\begin{theorem}\label{GH-equivalences-thm}
Let $(Z,\mathsf{d})$ be a proper metric space, and let $X,Y \subseteq Z$. Let $x_0\in X$ and $y_0\in Y$. Let $r>0$. The following are equivalent:
\begin{enumerate}
\item $\delta_r^{(Z,\mathsf{d})}((X,x_0),(Y,y_0)) \leq \varepsilon$,
\item there exist $K\in\compacts{Y}$ and $Q\in\compacts{X}$ such that:
\begin{enumerate}
\item $\cBall{X}{}{x_0}{r-2\varepsilon}\subseteq Q\subseteq \cBall{X}{}{x_0}{r+2\varepsilon}$, 
\item $\cBall{Y}{}{y_0}{r-2\varepsilon}\subseteq K\subseteq \cBall{Y}{}{y_0}{r+2\varepsilon}$, 
\item $\Haus{\mathds{d}}(\cBall{X}{}{x_0}{r},K)\leq\varepsilon$ and $\Haus{\mathds{d}}(Q,\cBall{Y}{}{y_0}{r})\leq\varepsilon$,
\item $\mathsf{d}(x_0,y_0)\leq \varepsilon$,
\end{enumerate}
\item there exist $K\in\compacts{Y}$ and $Q\in\compacts{X}$ such that:
\begin{enumerate}
\item $Q\subseteq \cBall{X}{}{x_0}{r+2\varepsilon}$, 
\item $K\subseteq \cBall{Y}{}{y_0}{r+2\varepsilon}$, 
\item $\Haus{\mathds{d}}(\cBall{X}{}{x_0}{r},K)\leq\varepsilon$ and $\Haus{\mathds{d}}(Q,\cBall{Y}{}{y_0}{r})\leq\varepsilon$,
\item $\mathsf{d}(x_0,y_0)\leq \varepsilon$,
\end{enumerate}
\item there exist $K\in\compacts{Y}$ and $Q\in\compacts{X}$ such that:
\begin{enumerate}
\item $\Haus{\mathds{d}}(\cBall{X}{}{x_0}{r},K)\leq\varepsilon$ and $\Haus{\mathds{d}}(Q,\cBall{Y}{}{y_0}{r})\leq\varepsilon$,
\item $\mathsf{d}(x_0,y_0)\leq \varepsilon$.
\end{enumerate}
\end{enumerate}

In particular:
\begin{equation}\label{alt-exp-eq0}
\delta_r^{(Z,\mathsf{d})}((X,x_0),(Y,y_0)) = \min\Set{\varepsilon > 0}{
\mathsf{d}(x_0,y_0)&\leq\varepsilon \text{ and }\\
\cBall{X}{}{x_0}{r}&\subseteq_\varepsilon \cBall{Y}{}{y_0}{r+2\varepsilon},
\\ \cBall{Y}{}{y_0}{r}&\subseteq_\varepsilon \cBall{X}{}{x_0}{r+2\varepsilon}}
\end{equation}
\end{theorem}

\begin{proof}
Assume (1). Define:
\begin{equation*}
K = \left\{ y \in Y : \mathsf{d}(y,\cBall{X}{}{x_0}{r})\leq\varepsilon \right\}\text{.}
\end{equation*}
Let $y\in K$. Since $\cBall{X}{}{x_0}{r}$ is compact as $X$ is proper, there exists $x\in \cBall{X}{}{x_0}{r}$ such that $\mathsf{d}(x,y)\leq\varepsilon$. Therefore $K\subseteq_\varepsilon \cBall{X}{}{x_0}{r}$. Moreover, we then have:
\begin{equation*}
\mathsf{d}(y_0,y)\leq\mathsf{d}(y_0,x_0)+\mathsf{d}(x_0,x)+\mathsf{d}(x,y)\leq \varepsilon + r + \varepsilon\leq r+2\varepsilon\text{.}
\end{equation*}
Thus $K\subseteq \cBall{Y}{}{y_0}{r+2\varepsilon}$. In particular, $K\in\compacts{Y}$ as a closed bounded subset of a proper metric space. Note that by symmetry, $\cBall{Y}{}{y_0}{r}\subseteq_\varepsilon \cBall{X}{}{x_0}{r+2\varepsilon}$ as well.

Now, since $\cBall{X}{}{x_0}{r}\subseteq_\varepsilon Y$ by (1), for all $x\in \cBall{X}{}{x_0}{r}$ there exists $y \in Y$ with $\mathsf{d}(x,y)\leq\varepsilon$, and thus $\cBall{X}{}{x_0}{r}\subseteq_\varepsilon K$. Therefore, $\Haus{\mathsf{d}}(\cBall{X}{}{x_0}{r},K) \leq \varepsilon$.

Let now $y \in \cBall{Y}{}{y_0}{r-2\varepsilon}$. Since $\cBall{Y}{}{y_0}{r-2\varepsilon} \subseteq \cBall{Y}{}{y_0}{r} \subseteq_\varepsilon \cBall{X}{}{x_0}{r+2\varepsilon}$, there exists $x \in \cBall{X}{}{x_0}{r+2\varepsilon}$ such that $\mathsf{d}(x,y)\leq \varepsilon$. Now:
\begin{equation*}
\begin{split}
\mathsf{d}(x,x_0) &\leq \mathsf{d}(x_0,y_0)+\mathsf{d}(y_0,y)+\mathsf{d}(y,x)\\
&\leq \varepsilon + r - 2\varepsilon + \varepsilon\\
&\leq 2\varepsilon + r - 2\varepsilon \leq r\text{.}
\end{split}
\end{equation*}
Hence $x\in \cBall{X}{}{x_0}{r}$, so $y \in K$ by definition. Hence $\cBall{Y}{}{y_0}{r-2\varepsilon} \subseteq K$.

The rest of our proof of (2) is done by symmetry.

Of course, (2) implies (3), (3) implies (4).

Assume (4). Then $\cBall{X}{}{x_0}{r} \subseteq_\varepsilon K \subseteq Y$ and $\cBall{Y}{}{y_0}{r} \subseteq_\varepsilon Q \subseteq X$, while $\mathsf{d}(x_0,y_0) \leq\varepsilon$, so (1) holds. This completes the proof of the equivalence between our four assertions.

Since (1) implies (3), we conclude that:
\begin{equation*}
\delta_r^{(Z,\mathsf{d})}((X,x_0),(Y,y_0)) \geq \inf\Set{\varepsilon>0}{
\mathsf{d}(x_0,y_0) &\leq \varepsilon\text{ and }\\
\cBall{X}{}{x_0}{r} &\almostsubseteq{}{\varepsilon} \cBall{Y}{}{y_0}{r+2\varepsilon},\\
\cBall{Y}{}{y_0}{r} &\almostsubseteq{}{\varepsilon} \cBall{X}{}{x_0}{r+2\varepsilon}
}\text{.}
\end{equation*}
The converse inequality is trivial. It is thus enough to prove that the set:
\begin{equation*}
E = \Set{\varepsilon > 0}{
\mathrm{d}(x_0,y_0) &\leq \varepsilon\text{ and }\\
\cBall{X}{}{x_0}{r} &\almostsubseteq{Z,\mathsf{d}}\varepsilon \cBall{Y}{}{y_0}{r+2\varepsilon}\text{,}\\
\cBall{Y}{}{y_0}{r} &\almostsubseteq{Z,\mathsf{d}}{\varepsilon} \cBall{X}{}{x_0}{r+2\varepsilon} \text{ and }\\
}
\end{equation*}
is closed to conclude that Expression (\ref{alt-exp-eq0}) holds.

Let $(\varepsilon_n)_{n\in\N} \in E$ converging to $\varepsilon$. Let $x\in \cBall{X}{}{x_0}{r}$. For each $n\in\N$, since $\varepsilon_n \in E$, there exists $y^n \in \cBall{Y}{}{y_0}{r+2\varepsilon_n}$ such that $\mathsf{d}(x,y^n)\leq \varepsilon_n$ since $\cBall{Y}{}{y_0}{r+2\varepsilon_n}$ is compact, as a closed ball of a proper metric space.

Since $(\varepsilon_n)_{n\in\N}$ converges, it is bounded. Thus there exists $M \geq 0$ such that, for all $n\in\N$, we have $\varepsilon_n \leq M$. 

Consequently, the sequence $(y^n)_{n\in\N}$ lies in the closed ball $\cBall{Y}{}{y_0}{r+2M}$ by construction, and the latter ball is compact since $Y$ is closed in $Z$ and $Z$ is proper. Thus, there exists a convergent subsequence $(y^{n_k})_{k\in\N}$ of $(y^n)_{n\in\N}$, with limit denoted by $y$. By construction, $\mathsf{d}(y_0,y)\leq r + 2\varepsilon$ and $\mathsf{d}(x,y)\leq \varepsilon$. Hence, $\cBall{X}{}{x_0}{r} \subseteq_\varepsilon \cBall{Y}{}{y_0}{r+2\varepsilon}$. 

The same reasoning applies with $X$ and $Y$ roles reversed. Note that of course, $\mathsf{d}(x_0,y_0)\leq\varepsilon_n$ for all $n\in\N$ implies that $\mathsf{d}(x_0,y_0)\leq\varepsilon$. Thus, the set $E$ is closed. It is nonempty and bounded below by $0$, so we conclude that $\inf E = \min E$, i.e. the infimum is in fact the smallest element of $E$. This concludes the first part of our theorem.
\end{proof}

In particular, we note that Expression (\ref{alt-exp-eq0}) is very close to Expression (\ref{wrong-def-eq}), and thus Definition (\ref{delta-r-def}) captures our intuition well. The reader is invited to check that, in any Euclidean space $E$, for a net of spheres $(S_n)_{j\in J}$ of radii converge to infinity, all tangent at some fixed $P$ to some hyperplane $H$, and for any $r>0$, we have $\lim_{j\in J} \delta_r^E((S_j,P),(H,P)) = 0$ --- i.e. we captured the intuition which started our quest for this topology on proper pointed metric spaces.

We can now easily checked that Hausdorff convergence of compacts is equivalent to the convergence in our new sense --- which includes the trivial case of Example (\ref{simple-compact-cv-example}):

\begin{proposition}\label{cv-convergence-prop}
Let $(Z,\mathsf{d})$ be a compact metric space, $(X_n)_{n\in\N}$ a sequence of closed subsets of $Z$ and $X$ a closed subset of $Z$. The following assertions are equivalent:
\begin{enumerate}
\item the sequence $(X_n)_{n\in\N}$ of compact subspaces of $Z$ converges for the Hausdorff distance $\Haus{\mathsf{d}}$ to $X$,
\item for all $x\in X$ there exists a sequence $(x_n)_{n\in\N}$ in $Z$ converging to $x$ and such that $x_n \in X_n$ for all $n\in\N$ while:
\begin{equation*}
\lim_{n\rightarrow\infty} \sup\Set{\delta_r^{(Z,\mathsf{d})}((X_n,x_n),(X,x))}{r > 0} = 0\text{,}
\end{equation*}
\item for some sequence $(x_n)_{n\in\N}$ with $x_n\in X_n$ for all $n\in\N$, and some $x\in X$, we have for all $r>0$:
\begin{equation*}
\lim_{n\rightarrow\infty} \delta_r^{(Z,\mathsf{d})}((X_n,x_n),(X,x)) = 0\text{.}
\end{equation*}
\end{enumerate}
Moreover, we note that any of Assertions (1), (2) or (3) implies that for all convergent sequences $(x_n)_{n\in\N}$ with $x_n\in X_n$ for all $n\in\N$ and whose limit is denoted by $x$, we have $x\in X$ and for all $r>0$:
\begin{equation*}
\lim_{n\rightarrow\infty} \sup\Set{\delta_r^{(Z,\mathsf{d})}((X_n,x_n),(X,x))}{r>0} = 0\text{.}
\end{equation*}
\end{proposition}

\begin{proof}
Assume (1). Let $x\in X$. Let $(n_k)_{k\in\N}$ be a strictly increasing sequence in $\N$ such that for all $k\in\N$, if $p \geq n_k$ then $\Haus{\mathsf{d}}(X_p,X)\leq 2^{-k}$. For each $n\in\N$, if $n \in \{n_{k},\ldots, n_{k+1}-1\}$ for some $k\in\N$ then we have $\Haus{\mathsf{d}}(X_n,X)\leq 2^{-k}$ and we choose $x_n \in X_n$ such that $\mathsf{d}(x_n,x) \leq 2^{-k}$. For $n\in \{0,\ldots,n_0-1\}$ we choose $x_n\in X_n$ arbitrarily. Note that since $(n_k)_{k\in\N}$ is strictly increasing, $\Set{ \{n_k,\ldots,n_{k+1}-1\} } { k \in \N}$ is a partition of $\N$, so our construction is well-defined.

Thus by construction, $\lim_{n\rightarrow\infty} x_n = x$.
Now, if $n\in\{n_{k},\ldots,n_{k+1}-1\}$ for some $k\in\N$ then, for all $r>0$ we have:
\begin{equation*}
X_n \subseteq_{2^{-k}} X \text{ so } \cBall{X_n}{}{x_n}{r} \subseteq_{2^{-k}} X
\end{equation*}
and similarly $\cBall{X}{}{x}{r} \subseteq_{2^{-k}} X_n$ so:
\begin{equation*}
\delta_r^{\left(Z,\mathsf{d}\right)}((X_n,x_n),(X,x)) \leq 2^{-k}
\end{equation*}
and thus $\lim_{n\rightarrow\infty} \sup\set{\delta_r^{\left(Z,\mathsf{d}\right)}((X_n,x_n),(X,x))}{r>0} = 0$. This proves (2). 

Trivially, (2) implies (3).

Assume (3). For all $n\in\N$, since $X_n$ is compact, its diameter $D_n = \diam{X_n}{\mathsf{d}}$ is finite. Similarly $D = \diam{X}{\mathsf{d}} \in \R$.

Let $\varepsilon > 0$ and $R > 0$. By (3), there exists $N\in\N$ such that for all $n\geq N$, we have $\cBall{X_n}{}{x_n}{R} \subseteq_{\frac{\varepsilon}{2}} X$. If $y,z\in \cBall{X_n}{}{x_n}{R}$, then there exists $x_y,x_z\in X$ such that $\mathsf{d}(y,x_y)\leq\frac{1}{2}\varepsilon$ and $\mathsf{d}(z,x_z)\leq\frac{1}{2}\varepsilon$, so:
\begin{equation*}
\mathsf{d}(y,z)\leq\mathsf{d}(y,x_y) + \mathsf{d}(x_y,x_z) + \mathsf{d}(x_z,z) \leq \varepsilon + D\text{.}
\end{equation*}
Thus, we have:
\begin{equation}\label{compact-cv-eq0}
\forall\varepsilon > 0\; \exists N\in\N \;\forall n \geq N \quad \diam{\cBall{X_n}{}{x_n}{R}}{\mathsf{d}} \leq D+\varepsilon\text{.}
\end{equation}

Suppose that for some $R > D$ and for all $N\in\N$, there exists $n\geq N$ such that $D_n \geq R$. Let $\varepsilon = \frac{1}{2}(R-D) > 0$. By Expression (\ref{compact-cv-eq0}), there exists $N \in \N$ such that $D_n \leq D + \varepsilon = R-\varepsilon$ for all $n\geq N$. Now, there exists $n\geq N$ such that $D_n\geq R$, which is a contradiction.

Thus we deduce that $\limsup_{n\rightarrow\infty} D_n \leq D$. In particular, the sequence $(D_n)_{n\in\N}$ is bounded. Let $B \in \R$ be chosen so that $B\geq \max\set{D,D_n}{n\in\N}$. Again by (3), for all $\varepsilon > 0$ there exists $N\in\N$ such that, if $n\geq N$, we have:
\begin{equation*}
\delta_B^{(Z,\mathsf{d})}((X_n,x_n),(X,x)) < \varepsilon\text{,}
\end{equation*}
so:
\begin{equation*}
X_n = \cBall{X_n}{}{x_n}{B} \subseteq_\varepsilon X\text{ and }X = \cBall{X}{}{x}{B} \subseteq_\varepsilon X_n
\end{equation*}
so $\Haus{\mathsf{d}}(X_n,X) \leq\varepsilon$, as desired. This proves (1) and concludes the equivalences between (1), (2) and (3).

Last, assume (1) again. Let $(x_n)_{n\in\N}$ be a sequence in $Z$, converging to some $x\in Z$, and such that $x_n\in X_n$ for all $n\in\N$. 

Now, let $(X_{n_k})_{k\in\N}$ be a subsequence of $(X_n)_{n\in\N}$ such that:
\begin{equation*}
\Haus{\mathsf{d}}(X_{n_k},X)\leq \frac{1}{1+k}
\end{equation*}
for all $k\in \N$. Hence, for all $k\in\N$, there exists $y_k\in X$ such that:
\begin{equation*}
\mathsf{d}(y_k,x_{n_k})\leq \frac{1}{1+k}\text{.}
\end{equation*}

Since $(y_k)_{k\in\N}$ is a sequence in the compact $X$, there exists a subsequence $(y_{k_j})_{j\in\N}$ of $(y_k)_{k\in\N}$ converging to some $y\in X$. On the other hand, for all $j\in\N$:
\begin{equation*}
0\leq\mathsf{d}(y,x_{n_{k_j}})\leq\mathsf{d}(y,y_{k_j}) + \mathsf{d}(y_{k_j},x_{n_{k_j}}) \text{,}
\end{equation*}
so, by taking the limit as $j\rightarrow\infty$, we have $\mathsf{d}(x,y) = 0$, and thus $x \in X$ as desired.

Now, we return to the original sequence $(X_n,x_n)_{n\in\N}$. Let $\varepsilon > 0$ and $r > 0$ be given. There exists $N\in\N$ such that for all $n\geq N$, we have  $\Haus{\mathsf{d}}(X_n,X)\leq\varepsilon$, i.e. $X_n\subseteq_\varepsilon X$ --- which implies $\cBall{X_n}{}{x_n}{r}\subseteq_\varepsilon X$ for all $r > 0$ --- and $X\subseteq_\varepsilon X_n$ --- which implies $\cBall{X}{}{x}{r}\subseteq_\varepsilon X_n$ for all $r > 0$.

Let $M\in\N$ such that $\mathsf{d}(x_n,x) \leq \varepsilon$. Then for all $n\geq \max\{N,M\}$ we have $\sup\set{\delta_r^{(Z,\mathsf{d})}((X_n,x_n),(X,x))}{r>0}\leq\varepsilon$, as claimed. This proves the last statement of our proposition.
\end{proof}

\begin{remark}
Note that if we assume (3) in Proposition (\ref{cv-convergence-prop}), since:
\begin{equation*}
\lim_{n\rightarrow\infty} \delta_1((X_n,x_n),(X,x)) = 0\text{,}
\end{equation*}
by Definition (\ref{delta-r-def}), we have $(x_n)_{n\in\N}$ converges to $x$. Thus the convergence of $(x_n)_{n\in\N}$ to $x$ is built-in. Yet it actually plays no role in the proof that (3) implies (1). This is simply because, balls or large enough radii contain the whole compact set regardless of their center. In fact, Proposition (\ref{cv-convergence-prop}) implicitly refers to pointed Hausdorff convergence, and could be phrased by stating that $(X_n,x_n)_{n\in\N}$ converges for the \emph{pointed} Hausdorff distance to $(X,x)$ if and only if:
\begin{equation*}
(\delta_r^{(Z,\mathsf{d})}((X_n,x_n),(X,x)))_{n\in\N}
\end{equation*} converges to $0$ for all $r > 0$, using the notations this Proposition. Our point, however, in proving Proposition (\ref{cv-convergence-prop}), is that our new notion of convergence agrees with the usual notion of Hausdorff convergence for compact sets, hence our choice of formulation.
\end{remark}

\subsection{The Gromov-Hausdorff Distance}

We wish to work with proper metric spaces which are not given as subspaces of a fixed proper metric space, and to this end, we employ the following technique familiar when dealing with the Gromov-Hausdorff distance:

\begin{definition}\label{Delta-r-def}
Let $\mathds{X} = (X,\mathsf{d}_X,x_0)$ and $\mathds{Y} = (Y,\mathsf{d}_Y,y_0)$ be two pointed proper metric spaces and $r > 0$. We define $\Delta_r(\mathds{X},\mathds{Y})$ as:
\begin{equation}\label{Delta-r-def-eq0}
\inf \Set{ \delta_r^{(Z,\mathsf{d})}((\iota_X(X)&,\iota_X(x_0)),\\
&(\iota_Y(Y),\iota_Y(y_0)))}{ (Z,\mathsf{d}) &\text{ metric space,}\\
&\iota_X:X\hookrightarrow Z, \iota_Y:Y\hookrightarrow Z,\\
&\iota_X \text{ isometry from $(X,\mathsf{d}_X)$ into $(Z,\mathsf{d}_Z)$,}\\
&\iota_Y \text{ isometry from $(Y,\mathsf{d}_Y)$ into $(Z,\mathsf{d}_Z)$}
}\text{.}
\end{equation}
\end{definition}

We thus may define convergence of pointed proper metric spaces as:

\begin{definition}
A net $(X_j,\mathsf{d}_j,x_j)_{j\in J}$ of pointed proper metric spaces converges in the sense of Gromov-Hausdorff to a pointed proper metric space $(X,\mathsf{d},x)$ when:
\begin{equation*}
\forall r > 0 \quad \lim_{j\in J} \Delta_r((X_j,\mathsf{d}_j,x_j),(X,\mathsf{d},x)) = 0\text{.}
\end{equation*}
\end{definition}

We pause for a reflection on an important point. Expression (\ref{Delta-r-def-eq0}) in Definition (\ref{Delta-r-def}) involves embedding two pointed proper metric spaces $(X,\mathsf{d}_X,x_0)$ and $(Y,\mathsf{d}_Y,y_0)$ isometrically into a proper metric space $(Z,\mathsf{d})$; yet Theorem (\ref{GH-equivalences-thm}) suggests that we may not need isometric embeddings of the entire space. We will show in this paper, with our construction of the Gromov-Hausdorff propinquity, that indeed such relaxation of Definition (\ref{Delta-r-def}) is possible. However, the following example illustrates that some care must be taken in this matter:

\begin{example}\label{funny-example}
Let $Z = \R^2$ endowed with the metric induced by the norm $\|\cdot\|_\infty : (x,y)\in Z\mapsto \max\{|x|,|y|\}$. Let $X = \R$ and:
\begin{equation*}
Y = \Set{(x,0), \left(x,1\right)}{x\in\R}\subseteq \Z \text{.}
\end{equation*}
We endow $X$ with the usual metric and $Y$ with its metric induced by $\|\cdot\|_\infty$, so that $\iota_Y : y\in Y \mapsto y \in Z$ is a trivial isometry. We choose $0$ as the base point in $X$ and $(0,0) \in Y$ as the base point in $Y$.

Let $R \geq 1$ and $\varepsilon \in (0,1)$. Let $I=(-R-1,\infty)$, $J=(-R-2+2\varepsilon,-R-1]$ and $K = (-\infty, -R-2+2\varepsilon ]$. Define:
\begin{equation*}
\iota_X: x\in \R \longmapsto \begin{cases}
(x,\varepsilon) \text{ if $x \in I$,}\\
(-R-1, - x - R - 1 + \varepsilon ) \text{ if $x\in J$,}\\
( -2R - 3 +2\varepsilon - x  , 1 - \varepsilon ) \text{ if $x \in K$.}
\end{cases}
\end{equation*}

Thus $\iota_X(X)$ is the union of three lines: a horizontal half-line starting at $(-R-1,\varepsilon)$ and going toward the right, a vertical segment between $(-R-1,\varepsilon)$ and $(-R-1,1-\varepsilon)$, and last a horizontal half line starting at $(-R-1,1-\varepsilon)$ and going again to the right. It is straightforward to check that $\iota_X$ is continuous.

The point of the construction of $\iota_X$ is that, on the one hand, $\iota_X$ restricts to an isometry on the ball of radius $R$ centered at $0$ in $X$. On the other hand, $\iota_X$ is not an isometry from $X$ to $Z$. 

We now make two observations. First, the the ball of center $(0,\varepsilon) = \iota_X(0)$ and radius $R$ of $\iota_X(X)$ is actually the union of two horizontal segments in $Z$:
\begin{equation*}
\left[\left(-R,\varepsilon\right),\left(R,\varepsilon\right)\right]\cup\left[\left(-R,1-\varepsilon\right),\left(R,1-\varepsilon\right)\right]
\end{equation*}
and is actually within distance $\varepsilon$ from the two horizontal segments in $Y$ which constitute the ball of radius $R$ centered at $(0,0)$:
\begin{equation*}
\left[\left(-R,0\right),\left(R,0\right)\right]\cup\left[\left(-R,1\right),\left(R,1\right)\right]\text{.}
\end{equation*}

Thus, if we defined $\Delta_R$ by requiring $\iota_X$ and $\iota_Y$ to be isometric only on the balls of radius $R$, respectively centered at $x_0$ in $X$ and $y_0$ in $Y$, then we see that in the current example, $\Delta_R((X,x_0),(Y,y_0))\leq\varepsilon$ for any $\varepsilon > 0$: so $\Delta_R((X,x_0),(Y,y_0)) = 0$!

If $R < 1$, then it is easy to simply define $\iota_Y :x\in Y\mapsto (x,\varepsilon)$ for an arbitrary $\varepsilon > 0$ to get $\Delta_R(X,Y)\leq\varepsilon$, and thus again $\Delta_R((X,x_0),(Y,y_0)) = 0$.

A second observation, based upon the same computation as above, shows that both closed balls in $X$ and $Y$ of radius $R$ centered at their respective base points are mapped by $\iota_X$ and $\iota_Y$ to sets which are within $\epsilon > 0$ of both $X$ and $Y$. Thus, if we modified $\Delta_R$ so that we took the distance between the images of these balls (rather than the balls in the embedding space), we would still have $\Delta_R((X,x_0),(Y,y_0)) = 0$.

Consequently, requiring isometry only on the balls of radius $R$ and centered at the base points in $X$ and $Y$ in the definition of $\Delta_r$ would lead to $X$ and $Y$ being at Gromov-Hausdorff $0$, even though the two spaces are certainly not isometric (or even homeomorphic, for that matter!). This illustrates that we must use isometry over the whole space in Definition (\ref{Delta-r-def}).
\end{example}

Inspite of Example (\ref{funny-example}), our construction of the Gromov-Hausdorff propinquity will not involve isometric embeddings, which is quite surprising. The reason for this unexpected fact is that we will exchange the isometric embeddings of Definition (\ref{Delta-r-def}) with a form of Lipschitz extension property in the noncommutative setting, which will provide us with the needed rigidity to avoid the sort of phenomenon illustrated with Example (\ref{funny-example}) in our noncommutative setup. 

We also remark that neither $\delta_r^{(Z,\mathsf{d})}$ nor $\Delta_r$ are not pseudo-metrics, yet for any three pointed proper metric spaces $(X,\mathsf{d}_X,x_0)$, $(Y,\mathsf{d}_Y,y_0)$ and $(Z,\mathsf{d}_Z,z_0)$, we have:
\begin{multline*}
\Delta_r( (X,\mathsf{d}_X,x_0),(Y,\mathsf{d}_Y,y_0)) \leq \\\Delta_R((X,\mathsf{d}_X,x_0),(Z,\mathsf{d}_Z,z_0)) + \Delta_R((Z,\mathsf{d}_Z,z_0),(Y,\mathsf{d}_Y,y_0))
\end{multline*}
for any $R > r + \max\{\Delta_r(X,Z),\Delta_r(Z,Y)$ (and similarly with $\delta_r$, from which in fact we derive the inequality for $\Delta_r$); moreover, if $R > R'$ then:
\begin{equation*}
\Delta_{R'}((X,\mathsf{d}_X,x_0),(Y,\mathsf{d}_Y,y_0))\leq \Delta_R((X,\mathsf{d}_X,x_0),(Y,\mathsf{d}_Y,y_0))\text{,}
\end{equation*}
as can easily be checked from Theorem (\ref{GH-equivalences-thm}). These properties allow one to prove that, if we let:
\begin{equation*}
\mathsf{GH}((X,x_0,Y,y_0)) = \max\left\{ \inf \Set{r>0}{\Delta_{\frac{1}{r}}((X,x_0),(Y,y_0)) < r} , \frac{1}{2}\right\} \text{,}
\end{equation*}
for any two pointed proper metric spaces $(X,x_0)$ and $(Y,y_0)$, then we define an inframetric which is (up to the truncation, needed to ensure the triangle inequality and finiteness of the metric) the metric proposed by Gromov in \cite{Gromov81}. We thus have arrived to the desired construction. The actual proof that $\mathsf{GH}$ is an inframetric can be obtained in a similar manner as our Theorem (\ref{propinquity-inframetric-thm}) later on, which will concern our noncommutative analogue of $\mathsf{GH}$. Of course, the construction and analysis of the metric $\mathsf{GH}$ is all due to Gromov.

With this in mind, we note that:
\begin{proposition}
A net $(X_j,\mathsf{d}_j,x_j)_{j\in J}$ of pointed proper metric spaces converges to a pointed proper metric space $(X,\mathsf{d},x)$ in the Gromov-Hausdorff sense if and only if:
\begin{equation*}
\lim_{j\in J}\mathsf{GH}((X_j,\mathsf{d},x_j),(X,\mathsf{d},x)) = 0\text{.}
\end{equation*}
\end{proposition}

\begin{proof}
If $(X_j,\mathsf{d}_j,x_j)_{j\in J}$ converges then for all $\varepsilon > 0$, there exists $j_\varepsilon \in J$ such that $\Delta_{\varepsilon^{-1}}((X_j,\mathsf{d}_j,x_j),(X,\mathsf{d},x))<\varepsilon$ for all $j\succ j_\varepsilon$ (where $\succ$ is the order of the directed set $J$). Hence $\mathsf{GH}((X_j,\mathsf{d}_j,x_j),(X,\mathsf{d},x))\leq\varepsilon$ for $j\succ j_\varepsilon$.

Conversely, let $\varepsilon > 0$ and $R>0$. Let $\epsilon = \min\{\varepsilon, R^{-1}\}$. There exists $j_\varepsilon \in J$ such that for all $j\succ j_\varepsilon$ we have $\mathsf{GH}((X_j,\mathsf{d}_j,x_j),(X,\mathsf{d},x))<\epsilon$, so by definition:
\begin{equation*}
\Delta_{\epsilon^{-1}}((X_j,\mathsf{d}_j,x_j),(X,\mathsf{d},x))\leq\epsilon\text{.}
\end{equation*}

Thus for all $j\succ j_\varepsilon$ we have:
\begin{equation*}
\Delta_R((X_j,\mathsf{d}_j,x_j),(X,\mathsf{d},x))\leq\Delta_{\epsilon^{-1}}((X_j,\mathsf{d}_j,x_j),(X,\mathsf{d},x)) \leq \epsilon\leq\varepsilon\text{.}
\end{equation*}
This concludes our proof.
\end{proof}

We conclude this subsection with a characterization of Gromov-Hausdorff convergence often found in the literature \cite{burago01}, expressed in terms of convergence of nets, without explicit reference to a metric. This formulation is often useful, so we prove that it is indeed equivalent to our own approach. We also point out that when all the proper metric spaces of a convergent net are length spaces, then the original naive idea to define convergence directly in term of closed balls centered around the base point is equivalent to Gromov-Hausdorff convergence --- though we do not have a notion of quantum length space yet, so for us this approach will not be useful in this paper (and is not equivalent to Gromov-Hausdorff convergence in general, as seen in Example (\ref{simple-compact-cv-example})).

\begin{theorem}\label{net-version-GH-thm}
Let $(X_j,\mathsf{d}_j,x_j)_{j\in J}$ be a net of pointed proper metric spaces and $(X,\mathsf{d},x)$ be a pointed proper metric space. The following assertions are equivalent:
\begin{enumerate}
\item $(X_j,\mathsf{d}_j,x_j)_{j\in J}$ converges, in the sense of Gromov-Hausdorff, to $(X,\mathsf{d},x)$,
\item for all $r > 0$ and for all $\varepsilon > 0$, there exists $j_0 \in J$ such that, for all $j\succ j_0$, there exists a compact $K_j \in \compacts{X}$ and some isometric embeddings $\iota_j : X_j \hookrightarrow Z$ and $\iota:X\hookrightarrow Z$ in a proper metric space $(Z,\mathsf{d}_Z)$ such that:
\begin{enumerate}
\item The Hausdorff distance between $\iota_j(\cBall{X}{\mathsf{d}}{x}{r})$ and $\iota(K)$ is less than $\varepsilon$ in $(Z,\mathsf{d}_Z)$,
\item $\mathsf{d}_Z(x_j,x)\leq\varepsilon$,
\item $\cBall{X_j}{\mathsf{d}_j}{x_j}{r-2\varepsilon}\subseteq K$.
\end{enumerate}
\end{enumerate}
\end{theorem}

\begin{proof}
Assume $(X_j,\mathsf{d}_j,x_j)_{j\in J}$ converges, in the sense of Gromov-Hausdorff convergence, to $(X,\mathsf{d},x)$. Let $\varepsilon > 0$ and $r > 0$. There exists $j_\varepsilon \in J$ such that if $j\succ j_\varepsilon$ then:
\begin{equation*}
\Delta_r((X_j,\mathsf{d}_j,x_j),(X,\mathsf{d},x)) < \varepsilon\text{.}
\end{equation*}
Let $j\succ j_\varepsilon$. There exists a proper metric space $(Z,\mathsf{d}_Z)$, an isometry $\iota_j : X_j \hookrightarrow Z$ and an isometry $\iota : X\hookrightarrow Z$ such that:
\begin{equation*}
\delta_r^{(\Z,\mathsf{d}_Z)}((\iota_j(X_j),\iota_j(x_j)),(\iota(X),\iota(x))) < \varepsilon\text{.}
\end{equation*}
By Theorem (\ref{GH-equivalences-thm}), there exists $K \in \compacts{X_j}$ such that $\Haus{\mathsf{d}_Z}(\cBall{X}{\mathsf{d}}{x}{r},K)\leq\varepsilon$ and $\cBall{X_j}{\mathsf{d}_j}{x_j}{r-2\varepsilon}\subseteq K$, while $d(x_j,x)\leq\varepsilon$, as desired. This proves that (1) implies (2).

Conversely, assume (2). Let $r > 0$ and $\varepsilon > 0$, and let $j_0\in J$ be given by (2). Let $j \succ j_0$. Thus, there exists a proper metric space $(Z,\mathsf{d}_Z)$, two isometries $\iota_j : X_j\hookrightarrow Z$ and $\iota : X\hookrightarrow Z$, and $C\in\compacts{X_j}$ such that:
\begin{enumerate}
\item $\Haus{\mathsf{d}_Z}(\iota(\cBall{X}{}{x_0}{r + 2\varepsilon}),\iota_j(C))\leq\varepsilon$,
\item $\cBall{X_j}{\mathsf{d}_j}{x_j}{r}\subseteq C$,
\item $\mathsf{d}_Z(x_j,x)\leq\varepsilon$.
\end{enumerate}

Now, let:
\begin{equation*}
K = \left\{ y \in X_j : \mathsf{d}_Z(\iota_j(y),\iota(\cBall{X}{\mathsf{d}}{x}{r}))\leq \varepsilon \right\}
\end{equation*}
and
\begin{equation*}
Q = \left\{ y \in X : \mathsf{d}_Z(\iota(y),\iota_j(\cBall{X_j}{\mathsf{d}_j}{x_j}{r}))\leq \varepsilon \right\}\text{.}
\end{equation*}
The set $Q$ is compact in $X$ while the set $K$ is compact in $X_j$. By construction, $\iota(Q) \subseteq_\varepsilon \iota_j(\cBall{X_j}{\mathsf{d}_j}{x_j}{r})$ and $\iota_j(K) \subseteq_\varepsilon \iota(\cBall{X}{\mathsf{d}}{x}{r})$.

Let $y\in \cBall{X_j}{}{x_j}{r}$. Since $\cBall{X_j}{}{x_j}{r}\subseteq C$, there exists $y'\in \cBall{X}{}{x}{r+2\varepsilon}$ such that $\mathsf{d}_Z(\iota_j(y),\iota(y'))\leq\varepsilon$. By definition, $y'\in Q$. Thus $\iota_j(\cBall{X_j}{}{x_j}{r})\subseteq_\varepsilon \iota(Q)$. Therefore $\Haus{\mathsf{d}_Z}(\iota_j(\cBall{X_j}{}{x_j}{r}),\iota(Q))\leq\varepsilon$.

Let now $y\in \cBall{X}{}{x}{r}$. There exists $y' \in C$ such that $d(y,y')\leq\varepsilon$. By definition, $y\in K$. Thus $\Haus{\mathsf{d}_Z}(\iota(\cBall{X}{}{x_0}{r}),\iota_j(K))\leq\varepsilon$.

Consequently, $\Delta_r((X_j,\mathsf{d}_j,x_j),(X,\mathsf{d},x))\leq\varepsilon$ by Theorem (\ref{GH-equivalences-thm}). This concludes our proof.
\end{proof}

Theorem (\ref{net-version-GH-thm}) is sometimes expressed in terms of correspondences or in terms of $\varepsilon$-isometries \cite[Definition 8.1.1]{burago01} (and with a factor of $2$ removed, which is a detail caused by our flexibility in not mapping base points to base points, which will be useful in our latter work). As $\varepsilon$-isometries are not continuous functions in general, it is delicate to find a noncommutative equivalent notion; however $\varepsilon$-isometries can be used to define correspondences, which in turn can be understood as multi-valued functions. To a given correspondence, one may associate a metric distortion, and our construction in this paper will bare some analogies with this view, which will help our intuition but no so much our formal development. 

Before entering the noncommutative world, we must add to this classical prelude an important observation regarding extension of Lipschitz functions which will fit our context. As seen in our introduction, there is an important relationship between the Lipschitz extension theorem of McShane \cite{McShane34} and the notion of isometric embedding in the compact case. Thus one would expect a similar relationship, of similar importance, for our work in this paper. The situation turns out to be more subtle.

\subsection{Lifting Lipschitz Functions}

The {\mongekant} $\Kantorovich{\Lip}$ associated with the Lipschitz seminorm $\Lip$ of a locally compact metric space $(X,\mathsf{d})$ is not a metric in general on the state space $\StateSpace(C_0(X))$, and the topology induced on $\StateSpace(C_0(X))$ is not the weak* topology in general --- unless $(X,\mathsf{d})$ is in fact compact. In \cite{Latremoliere12b}, we developed a theory of {\lcqms s} which proposes a possible way to handle these difficulties, and we shall summarize the relevant conclusions of \cite{Latremoliere12b} in the next section. 

One difficulty, however, which is not addressed in \cite{Latremoliere12b}, is that in general, the restriction of the {\mongekant} to the pure states of $C_0(X)$ is not isometric to the original distance $\mathsf{d}$. Indeed, if $I = (0,1)$, endowed with its natural metric, and if $f \in C_0(I)$ is $1$-Lipschitz, then $\|f\|_{C_0(I)} \leq \frac{1}{2}$ (since $f(0)=f(1)=0$). Consequently, the diameter of $\StateSpace(C_0(I))$ for the {\mongekant} is at most (and easily checked to be exactly) $\frac{1}{2}$, rather than $1$ for the original metric. The relationship between the {\mongekant} and the original metric is given by the simple observation:

\begin{proposition}\label{trivial-bound-prop}
Let $(X,\mathsf{d})$ be a locally compact metric space. If $\Kantorovich{\Lip}$ be the {\mongekant} associated with the Lipschitz seminorm  $\Lip$ for the metric $\mathsf{d}$, then for any $x,y \in X$ we have:
\begin{equation*}
\Kantorovich{\Lip}(x,y) \leq \mathsf{d}(x,y)\text{.}
\end{equation*}
\end{proposition}

\begin{proof}
If $f \in C_0(X)$ is $1$-Lipschitz then by definition, $|f(x)-f(y)|\leq\mathsf{d}(x,y)$. This completes our proof.
\end{proof}

This problem is intimately related to the problem of extending Lipschitz functions. Indeed, if we consider the subspace $F = \{0.1, 0.9\}\subseteq I$ and the function $f: x \in F\mapsto x$ then $f$ is trivially $1$-Lipschitz on $F$, yet it has no extension as a $1$-Lipschitz function in $C_0(I)$. 

Of course, McShane's theorem \cite{McShane34} proves that any $1$-Lipschitz function on a subspace $X$ of a metric space $(Y,\mathsf{d})$ can be extended to a $1$-Lipschitz function on $Y$; however this function will not usually vanish at infinity, even if $f$ does vanish ``at infinity'' on $X$.
\begin{example}\label{counter-example}
Let $I = (-1,1)\times\left(-\frac{1}{2},\frac{1}{2}\right)$ endowed with the metric given by the norm $\|\cdot\|_\infty$ on $\R^2$. Let $f:x\in (-1,1)\rightarrow 1-|x|\in\R$. Then $f$ is $1$-Lipschitz and $f\in C_0(-1,1)$. Yet, for all $\varepsilon \in (0,2)$, there is no $\varepsilon$-Liphitz function $g:I\rightarrow \R$ whose restriction to $(-1,1)\times\{0\}$ is $f$, while $g\in C_0(I)$.
\end{example}

For proper metric spaces, however, such pathologies do not occur: the {\mongekant} extends the original metric from which it is built to the entire state space, and we can provide a good extension result for Lipschitz functions in our setting. Both the notion of Gromov-Hausdorff convergence and the notion of {\lcqms s} we introduced \cite{Latremoliere12b} rely heavily on compact subsets and thus we will focus our attention to extending compactly supported Lipschitz functions to compactly supported Lipschitz functions. In particular, we are interested in some form of control on the support of the extension of a Lipschitz function $f$ in term of the Lipschitz constant of $f$, the radius of the support of $f$ and the norm of $f$, as well as the Hausdorff distance between the support of $f$ and the entire space. All these quantities are naturally related to the problem at hand, and in fact, the following series of results play a central role in the construction of the Gromov-Hausdorff propinquity. Of course, we could use similar methods to show that any Lipschitz function which vanish at infinity on a closed subspace of a proper metric space has an extension on the whole space which vanishes at infinity and has the same Lipschitz seminorm and norm, but this will not play a role in this paper.

We thus begin with:

\begin{proposition}\label{proper-extension-prop}
Let $(Z,\mathsf{d})$ be a proper metric space and $X\subseteq Z$ be a closed subset of $Z$. If $f : X \rightarrow\R$ be a compactly supported Lipschitz function, then there exists a compactly supported function $g : Z\rightarrow \R$ whose restriction to $X$ agrees with $f$, and such that $g$ has the same Lipschitz seminorm and the same norm on $Z$ as $f$ does on $X$. Moreover, if the support of $f$ is contained in some $\cBall{X}{}{x_0}{R}$ for some $x_0\in X$ and $R > 0$, then we can choose the support of $g$ to be contained in $\cBall{Z}{}{x_0}{R+\|f\|_{C_0(X)}}$.
\end{proposition}

\begin{proof}
Without loss of generality, we assume that $f$ is $1$-Lipschitz. Since the support of $f$ is compact, it is bounded, so there exist $x_0 \in X$ and $R>0$ such that $f$ is supported on the closed ball $\cBall{X}{}{x_0}{R}$. 

By McShane's Theorem \cite{McShane34}, there exists a $1$-Lipschitz function $f_1 : Z \rightarrow \R$ such that the restriction of $f_1$ to $X$ is $f$. Define:
\begin{equation*}
f_2 = \max\{ \min\{f_1,\|f\|_{C_0(X)}\}, -\|f\|_{C_0(X)} \}\text{.}
\end{equation*}
The function $f_2$ is still $1$-Lipschitz, and it moreover satisfies $\|f_2\|_{C_0(Z)}=\|f\|_{C_0(X)}$, while the restriction of $f_2$ to $X$ is still $f$.

Let $M = \|f\|_{C_0(X)}$.

If $C = Z\setminus \cBall{Z}{}{x_0}{r+M}$ is empty, then $f_2$ answers our proposition. Let us now assume $C\not=\emptyset$, and let:
\begin{equation*}
t_1 : z \in Z\longmapsto M \cdot \frac{\mathsf{d}(z,Z\setminus \cBall{Z}{}{x_0}{R+M})}{\mathsf{d}(z,Z\setminus \cBall{Z}{}{x_0}{R+M}) + \mathsf{d}(z,\cBall{Z}{}{x_0}{R})} \text{.}
\end{equation*}

Assume there exists $x\in Z$ such that:
\begin{equation*}
\mathsf{d}(x,Z\setminus \cBall{Z}{}{x_0}{R+M}) + \mathsf{d}(x,\cBall{Z}{}{x_0}{R}) < M \text{.}
\end{equation*}
Let $\varepsilon > 0$ such that:
\begin{equation*}
\mathsf{d}(x,Z\setminus \cBall{Z}{}{x_0}{R+M}) + \mathsf{d}(x,\cBall{Z}{}{x_0}{R}) + \varepsilon < M \text{.}
\end{equation*}
Since $C$ is not empty, we can choose $z\in C$ such that:
\begin{equation*}
M > \mathsf{d}(x,Z\setminus \cBall{Z}{}{x_0}{R+M}) + \varepsilon \geq \mathsf{d}(z,x)\text{.}
\end{equation*}
Since $\cBall{Z}{}{x_0}{R}$ is compact, there exists $y\in\cBall{Z}{}{x_0}{R}$ such that:
\begin{equation*}
\mathsf{d}(x,\cBall{Z}{}{x_0}{R}) = \mathsf{d}(x,y)\text{.}
\end{equation*}
Thus:
\begin{equation*}
\begin{split}
\mathsf{d}(z,x_0) &\leq \mathsf{d}(x_0,y) + \mathsf{d}(y,z)\\
&\leq R + \mathsf{d}(x,y) + \mathsf{d}(x,z) \leq R + M\text{.}
\end{split}
\end{equation*}
Thus $z\not\in C$ and we have reached the contradiction. Thus, for all $x\in Z$ we have:
\begin{equation*}
\mathsf{d}(x,Z\setminus \cBall{Z}{}{x_0}{R+M}) + \mathsf{d}(x,\cBall{Z}{}{x_0}{R}) \geq M \text{.}
\end{equation*}
Since:
\begin{equation*}
y\in Z \mapsto\mathsf{d}(y,C)
\end{equation*}
is $1$-Lipschitz, we conclude that $t_1$ is $1$-Lipschitz.

By construction, $t_1$ is supported on the closed ball $\cBall{Z}{}{x_0}{R+M}$ which is compact since $(Z,\mathsf{d})$ is proper.

Let:
\begin{equation*}
g = \max\{ \min\{ f_2, t_1 \}, -t_1 \} \text{.}
\end{equation*}

Again, $g$ is $1$-Lipschitz and its restriction to $X$ agrees with $f$, since $t_1(x) = \|f\|_{C_0(X)}=\|f_2\|_{C_0(Z)}$ for all $x\in \cBall{X}{}{x_0}{r}$ and $t_1(x) = 0 \implies f_2(x)=f(x)=0$ for all $x\in X$. This completes the construction of $g$, though the support of $g$ has larger radius than the support of $f$ a priori. It should be noted that the support of $g$ is contained in $\cBall{Z}{\mathsf{d}}{x_0}{R+M}$ which is compact.
\end{proof}

Consequently, we can easily check:
\begin{proposition}
If $(X,\mathsf{d})$ is a proper metric space, then the {\mongekant} $\Kantorovich{\Lip}$ on $\StateSpace(C_0(X))$ associated with the Lipschitz seminorm $\Lip$ for $\mathsf{d}$, restricts to $\mathsf{d}$ on $X$, where we identify points and Dirac measures over $X$.
\end{proposition}

\begin{proof}
Let $x,y \in X$ and define $f : t \in \cBall{X}{\mathsf{d}}{x}{\mathsf{d}(x,y)} \mapsto \mathsf{d}(x,t)$. Then $f$ is $1$-Lipschitz on $\cBall{X}{\mathsf{d}}{x}{\mathsf{d}(x,y)}$ and thus, since $X$ is proper, there exists $g : X\rightarrow \R$ compactly supported on $X$ and $1$-Lipschitz by Proposition (\ref{proper-extension-prop}). Thus $\Kantorovich{\Lip}(x,y) \geq \mathsf{d}(x,y)$. The converse inequality is always true by Proposition (\ref{trivial-bound-prop}).
\end{proof}

Within the framework of the Gromov-Hausdorff distance, we will work with two closed pointed subsets $(X,x)$ and $(Y,y)$ of some proper metric space $(Z,\mathsf{d}_Z)$ and some $r > 0$, $\varepsilon > 0$, such that:
\begin{equation}\label{text-eq1}
\cBall{X}{}{x}{r}\subseteq_\varepsilon\cBall{Y}{}{y}{r+2\varepsilon}\text{,}
\end{equation}
as seen with Theorem (\ref{GH-equivalences-thm}). This additional datum will prove helpful in controlling the support of an extension of a compactly supported $1$-Lipschitz function, thus improving Proposition (\ref{proper-extension-prop}). Our motivation for such an improvement is that Proposition (\ref{proper-extension-prop}) provides us with a replica of a $1$-Lipschitz function $f$ compactly supported on $X$ as the restriction of an extension of $f$ to $Y$, but the situation is not very symmetric: the support of this restriction can have quite a bit larger diameter than the diameter of the support of $f$. This would prove a problem when trying to define a composition for tunnels, later in our paper, upon which our construction relies. Thus, we propose the following series of refinements for Lipschitz extensions.

\begin{theorem}\label{Lipschitz-classical-lift-thm}
Let $(Z,\mathsf{d})$ be a proper metric space, $X,Y \subseteq Z$, $x_0\in X$, $y_0\in Y$ and $r > 0$. Let:
\begin{equation*}
R = \inf \set{ \mathsf{d}_X(y,x_0) }{y \in X\setminus\cBall{X}{\mathsf{d}_X}{x_0}{r}} \text{.}
\end{equation*}

Let $\varepsilon \in \left(0,\frac{R}{2}\right)$ be chosen so that:
\begin{equation*}
\cBall{Y}{}{y_0}{2r+R}\subseteq_\varepsilon X\text{ and }\mathsf{d}(x_0,y_0)\leq\varepsilon\text{,}
\end{equation*}
and $\mathsf{d}(x_0,y_0) \leq \varepsilon$.

Let $f : X\rightarrow\R$ be a $1$-Lipschitz function such that $f(x) = 0$ if $\mathsf{d}(x_0,x) \geq r$.

Then $\|f\|_{C_0(X)}\leq R + r$ and there exists two $1$-Lipschitz functions $g : Z\rightarrow\R$ and $h : Y\rightarrow \R$ such that:
\begin{enumerate}
\item the restriction of $g$ to $X$ is $f$,
\item if $\mathsf{d}(x_0,y)\geq r+\|f\|_{C_0(X)}$, then $g(y) = 0$,
\item $\|g\|_{C_0(Z)}\leq \|f\|_{C_0(X)}$ and $\|h\|_{C_0(Y)}\leq\|g\|_{C_0(Y)}$,
\item $\|g-h\|_{C_0(Y)}\leq\varepsilon$ and $h(y) = 0$ whenever $\mathsf{d}(y_0,y)\geq r+2\varepsilon$.
\end{enumerate}
\end{theorem}

\begin{proof}
We first note that:
\begin{multline*}
R + r\geq \\ \sup\Set{\|h\|_X}{h:X\rightarrow\R \text{ is $1$-Lipschitz and }(\mathsf{d}(x_0,x)\geq r \implies h(x) = 0)}\text{.}
\end{multline*}
The result is trivial if $R = \infty$. Let us assume $R < \infty$. Let $h:X\rightarrow\R$ is $1$-Lipschitz and supported on $\cBall{X}{}{x_0}{r}$. For all $\varepsilon > 0$, there exists $x\in X\setminus\cBall{X}{}{x_0}{r}$ with $\mathsf{d}_X(x,x_0) < R + \varepsilon$.  Now, $h(x) = 0$, so for any $y \in \cBall{X}{}{x_0}{r}$:
\begin{equation*}
|h(y)|=|h(y)-h(x)|\leq \mathsf{d}_X(y,x) \leq \mathsf{d}_X(y,x_0)+\mathsf{d}_X(x_0,x)\leq r + R + \varepsilon\text{.}
\end{equation*}
The result follows since $\varepsilon > 0$ is arbitrary.

We let $M=\|f\|_{C_0(X)}$ to ease notations.

The function $g$ of our theorem is obtained by applying Proposition (\ref{proper-extension-prop}) to $f$. It should be noted that the support of $g$ is contained in:
\begin{equation*}
\cBall{Z}{\mathsf{d}}{x_0}{r+M}\subseteq\cBall{Z}{\mathsf{d}}{x_0}{2r+R}\text{.}
\end{equation*}

Let $B = \set{y \in Y}{\mathsf{d}(y,y_0)\in[r+2\varepsilon,2r+R]}$. We have two distinct cases: if $B = \emptyset$, then the restriction of $g$ to $Y$ is in fact supported on $\cBall{Y}{}{y_0}{r+2\varepsilon}$ and thus we can set $h$ to be the restriction of $g$ to $Y$, and our proof is complete. 

Assume instead that $B$ is not empty. The first key observation is as follows: if $y\in B$, then there exists $x\in X$ such that $\mathsf{d}(x,y)\leq\varepsilon$ since, by assumption, $B\subseteq \cBall{Y}{}{y_0}{2r+R}\subseteq_\varepsilon X$. Now:
\begin{equation*}
\mathsf{d}(y,x_0)\geq \mathsf{d}(y,y_0)-\mathsf{d}(y_0,x_0)\geq r + 2\varepsilon - \varepsilon = r + \varepsilon \text{.}
\end{equation*}
Thus:
\begin{equation*}
\mathsf{d}(x_0,x)\geq \mathsf{d}(y,x_0)-\mathsf{d}(x,y)\geq r + \varepsilon - \varepsilon = r \text{.}
\end{equation*}
Thus $g(x) = f(x) = 0$ by assumption. On the other hand, $g$ is $1$-Lipschitz, so:
\begin{equation}\label{lip-extension-eq0}
|g(y)|\leq |g(0)|+\mathsf{d}(y,x) \leq \varepsilon\text{.}
\end{equation}

Let $t_2 : y \in Y \mapsto \mathsf{d}(y,B)$. Note that $t_2$ is $1$-Lipschitz on $Y$. We now define:
\begin{equation*}
h : y \in Y \longmapsto \max\left\{ \min\left\{ g(y), t_2(y) \right\}, -t_2(y) \right\} \text{.}
\end{equation*}

By construction, $h$ is supported on $\cBall{Y}{}{y_0}{r+2\varepsilon}$. Indeed, if $y\in Y\setminus \cBall{Y}{}{y_0}{r+2\varepsilon}$ then either $y\in B$ and then $t_2(y)=0$ or $y\not\in B$ and then $g(y) = 0$, so either way $h(y) = 0$. 
We propose to show that $h$, while not an extension of $f$ in general, satisfies $\|g-h\|_{C_0(Y)} \leq\varepsilon$.

First, note that if $y \in Y\setminus \left(B\cup \cBall{Y}{}{y_0}{r+2\varepsilon}\right)$ then $g(y) = 0 = h(y)$. If $y\in B$ then $t_2(y) = 0$, so $h(y) = 0$, yet by Equation (\ref{lip-extension-eq0}), we have $|g(y)|\leq\varepsilon$. Hence $|g(y)-h(y)|\leq \varepsilon$ for all $y\in Y\setminus \cBall{Y}{}{y_0}{r+2\varepsilon}$.

Let now $y \in \cBall{Y}{}{y_0}{r+2\varepsilon}$. Let $x \in B$ be chosen so that $\mathsf{d}(x,y) = t_2(y)$, which exists since $B$ is compact (as $Y$ is proper) and nonempty. By Inequality (\ref{lip-extension-eq0}), we have $|g(x)|\leq\varepsilon$. Since $g$ is $1$-Lipschitz, we have:

\begin{equation}\label{lip-extension-eq1}
|g(y)-g(x)|\leq \mathsf{d}(x,y) = t_2(y)\text{.}
\end{equation}

If $g(y)\geq t_2(y)\geq 0$ then by Equation (\ref{lip-extension-eq1}), we have:
\begin{equation*}
g(y)\leq t_2(y) + g(x) \leq t_2(y) + \varepsilon
\end{equation*}
and $h(y) = t_2(y)$, so $|g(y)-h(y)| = g(y)-t_2(y) \leq\varepsilon$. If instead $g(y)\leq -t_2(y)\leq 0$ then, again by Equation (\ref{lip-extension-eq1}), we have:
\begin{equation*}
g(y)\geq -\mathsf{d}(x,y) + g(x) \geq -t_2(y) - \varepsilon\text{.}
\end{equation*}
On the other hand, $h(y) = -t_2(y)$, and thus $|g(y)-h(y)| = h(y)-g(y) \leq \varepsilon$.

Last if $g(y)\in[-t_2(y),t_2(y)]$ then $h(y)=g(y)$ so $|g(y)-h(y)| = 0\leq\varepsilon$.

In conclusion $\|g-h\|_{C_0(Y)} \leq\varepsilon$ as desired.
\end{proof}

Theorem (\ref{Lipschitz-classical-lift-thm}) provides us with an approximate replica: using the notations of the theorem, we build, from $f$ on $X$, a $1$-Lipschitz function $h$ on $Y$ which is within $\varepsilon$ of some $1$-Lipschitz extension of $f$ to the whole space; moreover our replica has a support of similar diameter as the support of $f$ as long as $\varepsilon$ is small. This would do fine, but we can actually improve our result as follows.

\begin{corollary}\label{Lisphitz-classical-lift-1-corollary}
Let $(Z,\mathsf{d}_Z)$ be a proper metric space, $X,Y \subseteq Z$, $x_0\in X$, $y_0\in Y$ and $r > 0$. Let:
\begin{equation*}
R = \inf \set{ \mathsf{d}_X(y,x_0) }{y \in X\setminus\cBall{X}{\mathsf{d}_X}{x_0}{r}} \text{.}
\end{equation*}

Let $\varepsilon \in \left(0,\frac{R}{2}\right)$ be chosen so that:
\begin{equation*}
\cBall{Y}{}{y_0}{2r+R}\subseteq_\varepsilon X\text{ and }\mathsf{d}(x_0,y_0)\leq\varepsilon\text{.}
\end{equation*}

There exists a proper metric space $(W,\mathsf{w})$ and two isometric embeddings $\omega_X : (X,\mathsf{d}_X)\hookrightarrow (W,\mathsf{w})$ and $\omega_Y : (Y,\mathsf{d}_Y)\hookrightarrow (W,\mathsf{w})$ such that, for all $1$-Lipschitz function $f : X\rightarrow\R$ such that $f(x) = 0$ if $\mathsf{d}(x_0,x) \geq r$, there exists a $1$-Lipschitz function $j : W\rightarrow \R$ such that $j\circ\omega_X = f$ on $X$ and $j\circ\omega_Y(x) = 0$ for all $x\in Y$ with $\mathsf{d}_Y(y,y_0)\geq r + 2\varepsilon$.
\end{corollary}

\begin{proof}
Let $W = X\coprod Z \coprod Y$ be the disjoint union of $X$, $Z$ and $Y$. To keep our notations clear, we do not identify $X$ with $\iota_X(X)$ nor $Y$ with $\iota_Y(Y)$. 

We define the following function on $W\times W$:
\begin{equation*}
\mathsf{w} : (x,y) \in W \mapsto \begin{cases}
\mathsf{d}_X(x,y) & \text{if $x,y \in X$,}\\
\mathsf{d}_Y(x,y) & \text{if $x,y \in Y$,}\\
\mathsf{d}_Z(x,y) & \text{if $x,y \in Z$,}\\
\mathsf{d}_Z(\iota_X(x),y) + \varepsilon &\text{if $x\in X, y\in Z$,}\\
\mathsf{d}_Z(\iota_Y(x),y) + \varepsilon &\text{if $x\in Y, y\in Z$,}\\
\mathsf{d}_Z(x,\iota_X(y)) + \varepsilon &\text{if $x\in Z, y\in X$,}\\
\mathsf{d}_Z(x,\iota_Y(y)) + \varepsilon &\text{if $x\in Z, y\in Y$,}\\
\mathsf{d}_Z(\iota_X(x),\iota_Y(y)) + 2\varepsilon &\text{if $x\in X, y\in Y$,}\\
\mathsf{d}_Z(\iota_Y(x),\iota_X(y)) + 2\varepsilon &\text{if $x\in Y, y\in Z$.}\\
\end{cases}
\end{equation*}

It is easy to check that $\mathsf{w}$ thus defined is a distance on $W$ since $\mathsf{d}(\iota_X(x),\iota_X(y)) = \mathsf{d}_X(x,y)$ for $x,y \in X$ and similarly $\mathsf{d}(\iota_X(x),\iota_X(y)) = \mathsf{d}_Y(x,y)$ for $x,y\in Y$. Let $\omega_X : x\in X \mapsto x \in X \subseteq W$ and $\omega_Y : x\in Y \mapsto y \in Y\subseteq W$: by construction, $\omega_X$ and $\omega_Y$ are isometries from $(X,\mathsf{d}_X)$ and $(Y,\mathsf{d}_Y)$, respectively, into $(W,\mathsf{w})$. 

Now, let $f : X\rightarrow \R$ be a $1$ Lipschitz function supported on $\cBall{X}{}{x}{r}$ and let $g: Z\rightarrow \R$ and $h:Y\rightarrow\R$ be given by Theorem (\ref{Lipschitz-classical-lift-thm}). We define:
\begin{equation*}
j : x\in W \mapsto \begin{cases}
f(x) & \text{if $x\in X$,}\\
g(x) & \text{if $x\in Z$,}\\
h(x) & \text{if $x\in Y$.}
\end{cases}
\end{equation*}
It is obvious by construction that $j$ restricted to $X$, $Z$ and $Y$ is $1$-Lipschitz. Let $y \in Y$ and $x\in Z$ be given. Then:
\begin{equation*}
\begin{split}
|j(x)-j(y)| &= |g(x) - h(y)| \leq |g(x) - g(\iota_Y(y))| + |g(\iota_Y(y)) - h(y)|\\
&\leq \mathsf{d}(x,\iota_Y(y)) + \varepsilon = \mathsf{w}(x,y) \text{.}
\end{split}
\end{equation*}
Similarly, if $x\in X, y \in W$ then:
\begin{equation*}
\begin{split}
|j(x) - j(y)| &= |f(x) - g(y)| \\
&= |f(x) - g(\iota_X(x))| + |g(\iota_X(x))-g(y)| \\
&\leq 0 + \mathsf{d}(\iota_X(x),y) \leq \mathsf{w}(x,y) \text{.}
\end{split}
\end{equation*}
Last if $x\in X, y \in Y$, then:
\begin{equation*}
\begin{split}
|j(x)-j(y)| &= |f(x)-h(y)| \\
&\leq |f(x)-g(\iota_X(x))|+|g(\iota_X(x))-g(\iota_Y(y))|+|g(\iota_Y(y))-h(y)| \\
&\leq 0 + \mathsf{d}(\iota_X(x),\iota_Y(y)) + \varepsilon \leq \mathsf{w}(x,y) \text{.}
\end{split}
\end{equation*}
By construction, $j$ is therefore $1$-Lipschitz on $W$. Moreover, $j\circ\omega_Y = h$ is supported in $\cBall{Y}{}{y}{r+2\varepsilon}$ while $j\circ\omega_X = f$ on $X$. This proves our corollary.
\end{proof}

\begin{corollary}\label{Lipschitz-lift-main-corollary}
Let $(Z,\mathsf{d})$ be a proper metric space, $X,Y \subseteq Z$, $x_0\in X$, $y_0\in Y$ and $r > 0$. Let:
\begin{equation*}
R = \inf \set{ \mathsf{d}_X(y,x_0) }{y \in X\setminus\cBall{X}{\mathsf{d}_X}{x_0}{r}} \text{.}
\end{equation*}

Let $\varepsilon \in \left(0,\frac{R}{2}\right)$ be chosen so that:
\begin{equation*}
\cBall{Y}{}{y_0}{2r+R}\subseteq_\varepsilon X\text{ and }\mathsf{d}(x_0,y_0)\leq\varepsilon\text{.}
\end{equation*}

There exists a metric $\mathsf{w}$ on the disjoint union $X\coprod Y$ whose restrictions to $X$ and $Y$ are given respectively by $\mathsf{d}_X$ and $\mathsf{d}_Y$ such that $(X\coprod Y,\mathsf{w})$ is a proper metric space, such that, for all $1$-Lipschitz function $f : X\rightarrow\R$ such that $f(x) = 0$ if $\mathsf{d}(x_0,x) \geq r$, there exists a $1$-Lipschitz function $j : X\coprod Y \rightarrow \R$ such that the restriction of $j$ to $X$ is $f$, while $j(x) = 0$ for all $x\in Y$ with $\mathsf{d}_Y(x,y_0)\geq r + 2\varepsilon$.
\end{corollary}

\begin{proof}
By Corollary (\ref{Lisphitz-classical-lift-1-corollary}), we may build am proper metric $\mathsf{w}$ on $W = X\coprod Z \coprod Y$ and a function $j : W\rightarrow\R$ which is $1$-Lipschitz on $W$, restricts to $f$ on $X$ and is whose restriction to $Y$ is supported on $\cBall{Y}{}{y}{r+2\varepsilon}$. Endow $X\coprod Y$ with the restriction of $\mathsf{w}$ from $W$, and note that the restriction of $j$ to $X\coprod Y$ satisfies our requirements.
\end{proof}

We now proceed to introduce the class of {\pqpms}, which we then propose to metrize.

\section{Pointed Proper Quantum Metric Spaces}

The first step in generalizing metric geometry to noncommutative geometry is of course to define the category of quantum metric spaces. The notion of quantum compact metric space was introduced by Rieffel in \cite{Rieffel98a}, refined in \cite{Rieffel99}, and has served as a successful basis for the generalization of the Gromov-Hausdorff distance to the noncommutative setup.

The notion of a {\lcqms} is the result of our research, and presented several subtle challenges. We proposed a notion of {\lcqms s} designed precisely with Gromov-Hausdorff convergence in mind in \cite{Latremoliere12b}, which itself built on the first steps toward the theory of locally compact quantum metric space found in \cite{Latremoliere05b}.

In this section, we recall from \cite{Latremoliere12b} the foundation of the theory of {\lcqms s}.

The following notations will be used throughout this paper.

\begin{notation}
Let $\A$ be a C*-algebra. The set of self-adjoint elements of $\A$ is denoted by $\sa{\A}$. The state space of $\A$ is denoted by $\StateSpace(\A)$ and the norm on $\A$ is denoted by $\|\cdot\|_\A$. We denote by $\unital{\A}$ the smallest C*-algebra containing $\A$ and a unit, i.e. $\A$ itself if it contains a unit, and $\A\oplus\C$ otherwise. Either way, the unit of $\unital{\A}$ is denoted by $\unit_\A$.
\end{notation}

\subsection{Locally Compact Quantum Compact Metric Spaces}

Metric information is encoded, in the noncommutative setting, by a seminorm which shares some properties of Lipschitz seminorms in the classical setting. The minimal assumptions for noncommutative metric geometry are summarized in the following definition:

\begin{definition}
A \emph{Lipschitz pair} $(\A,\Lip)$ is given by a C*-algebra $\A$ and a seminorm $\Lip$ defined on a dense subspace $\dom{\Lip}$ of $\sa{\unital{\A}}$, such that:
\begin{equation*}
\Set{a\in\dom{\Lip}}{\Lip(a)=0} = \R\unit_\A\text{.}
\end{equation*}
\end{definition}

We follow the convention that, if $\Lip$ is a seminorm defined on a dense subset $\dom{\Lip}$ of $\sa{\A}$, and if $a\not\in\sa{\A}$, then $\Lip(a) = \infty$. With this convention, we have $\dom{\Lip} = \set{ a\in\sa{\A} }{ \Lip(a) < \infty }$. 

The focus of our attention will be the following extended metric:

\begin{definition}
The {\mongekant} associated with a Lipschitz pair $(\A,\Lip)$ is the extended metric defined on the state space $\StateSpace(\A)$ of $\A$ by:
\begin{equation*}
\Kantorovich{\Lip}:\varphi,\psi\in\StateSpace(\A)\longmapsto \sup\Set{|\varphi(a)-\psi(a)|}{a\in\sa{\A}\text{ and }\Lip(a)\leq 1}\text{.}
\end{equation*}
\end{definition}

Our approach to {\lcqms s} involves an additional structure, which allows one to define, at once, convergence at infinity and locality, in a manner which will be particularly suited to our goal of extending the Gromov-Hausdorff distance to the noncommutative context. This structure is given by:

\begin{definition}
A \emph{topographic quantum space} $(\A,\M)$ is a C*-algebra $\A$ and an Abelian C*-subalgebra $\M$ of $\A$ containing an approximate unit of $\A$. The C*-algebra $\M$ is then called the \emph{topography} of $(\A,\M)$, or a topography on $\A$.
\end{definition}

The natural notion of morphisms of topographic quantum spaces is:

\begin{definition}\label{topomorph-def}
Let $(\A,\M_\A)$ and $(\B,\M_\B)$ be two quantum topographic spaces. A *-epimorphism $\pi : \A \rightarrow \B$ is a \emph{topographic morphism} when:
\begin{enumerate}
\item $\pi(\M_\A)\subseteq \M_\B$,
\item $\pi$ maps some approximate unit for $\A$ in $\M_\A$ to an approximate unit for $\B$.
\end{enumerate}
\end{definition}

Let $\pi : \D\rightarrow\A$ be a *-morphism which maps some approximate unit of $\D$ to an approximate unit of $\A$. We first observe that $\pi$ maps all approximate units of $\D$ to approximate units of $\A$, and of course $\pi$ is unital if $\D$ is; which in turns enforces that $\A$ is unital. We also note that if $\varphi\in\StateSpace(\B)$ then $\varphi\circ\pi \in \StateSpace(\A)$. With this in mind, we will use the following notation:

\begin{notation}
Let $\pi:\A\rightarrow\B$ be a *-morphism which maps an approximate unit of $\A$ to an approximate unit of $\B$. The map $\pi^\ast : \StateSpace(\B) \rightarrow \StateSpace(\A)$ is defined by $\varphi \in \StateSpace(\B) \mapsto \varphi\in\pi$.
\end{notation}

Given a topographic quantum space, several natural constructions become possible:
\begin{notation}
Let $(\A,\M)$ be a topographic quantum space. For any $K\in\compacts{\M^\sigma}$, we denote by $\indicator{K}$ the indicator function of $K$, seen as a projection in $\M^{\ast\ast}$, and we denote $\indicator{K}a\indicator{K}$ by $\corner{K}{a}$.

The set $\sa{\A}\cap\indicator{K}\A\indicator{K}$ is denoted by $\lsa{\A}{\M}{K}$. Moreover, the convex subset:
\begin{equation*}
\set{ \varphi \in \StateSpace(\A) }{ \varphi(\indicator{K}) = 1 }
\end{equation*}
is denoted by $\StateSpace[\A|K]$.

The set $\bigcup_{K\in\compacts{\M}}\StateSpace[\A|K]$ is denoted by $\StateSpace(\A|\M)$.
\end{notation}

The signature of a {\lcqms} is given as:

\begin{definition}
A \emph{Lipschitz triple} $(\A,\Lip,\M)$ is given by a Liphitz pair $(\A,\Lip)$ and a topographic quantum space $(\A,\M)$.
\end{definition}

The {\mongekant} associated with the Lipschitz seminorm for a continuous metric on a compact space metrizes the weak* topology on the set of Radon probability measures over the space. This feature was the fundamental property used to define compact quantum metric spaces.

When working with locally compact metric spaces instead, the {\mongekant} does not possess this property any longer. However, Dobrushin \cite{Dobrushin70} proved that the {\mongekant} still metrizes the weak* topology of some subsets of Radon probability measures which enjoy a strong form of tightness related to the metric structure. The noncommutative analogue of this property \cite{Latremoliere12b} is given by:

\begin{definition}
Let $(\A,\Lip,\M)$ be a Lipschitz triple. A subset $\mathscr{S}$ of $\StateSpace(\A)$ is \emph{tame} when:
\begin{equation*}
\lim_{K\in\compacts{\M^\sigma}}\sup\Set{|\varphi(a-\corner{K}{a})|}{\varphi\in\mathscr{S}, a\in\sa{\A},\Lip(a)\leq 1} = 0\text{.}
\end{equation*}
\end{definition}

Thus, we can define the objects of interest in our paper:

\begin{definition}
A Lipschitz triple $(\A,\Lip,\M)$ is a \emph{\lcqms} when the restriction of the {\mongekant} associated with $(\A,\Lip)$ to any tame subset $\mathscr{S}$ of $\StateSpace(\A)$ metrizes the weak* topology restricted to $\mathscr{S}$.
\end{definition}

In \cite{Latremoliere12b}, we propose a characterization of {\lcqms s} in the spirit of Rieffel's characterization of {\qcms s}, itself a noncommutative analogue of Arz{\'e}la-Ascoli theorem. Our characterization reads as follows:

\begin{theorem}\label{lcqms-thm}
Let $(\A,\Lip,\M)$ be a Lipschitz triple. The following assertions are equivalent:
\begin{enumerate}
\item $(\A,\Lip,\M)$ is a {\lcqms},
\item For all $c,d \in \M$ with compact support, and for some $\mu\in\StateSpace(\A|\M)$, the set:
\begin{equation*}
\Set{cad}{a\in\sa{\unital{A}},\Lip(a)\leq 1, \text{ and } \mu(a) = 0}
\end{equation*}
is norm precompact,
\item for all $c,d \in \M$ with compact support, and for any $\mu\in\StateSpace(\A|K)$, the set:
\begin{equation*}
\Set{cad}{a\in\sa{\unital{A}},\Lip(a)\leq 1, \text{ and } \mu(a) = 0}
\end{equation*}
is norm precompact.
\end{enumerate}
\end{theorem}

We now wish to restrict our attention to a class of {\lcqms s} which generalize the notion of proper metric spaces to the noncommutative setting. To this end, we must understand the metric structure on the topography. In \cite{Latremoliere12b}, we developed the theory of {\lcqms s} at the highest level of generality we thought appropriate, and we noted that in this context, there was a natural distance on the Gel'fand spectrum of the topography of a {\lcqms}, as we recall in the next theorem. 

\begin{theorem}\label{old-metric-thm}
Let $(\A,\Lip,\M)$ be a {\lcqms}. For any $\mu\in\StateSpace(\M)$, let:
\begin{equation*}
\StateSpace(\A|\mu) = \Set{\varphi\in\StateSpace(\A)}{\forall m\in \M \quad \varphi(m) = \mu(m) }\text{.}
\end{equation*}
The function:
\begin{equation*}
\mathsf{d}{\Lip}:\mu,\nu\in\StateSpace(\M)\longmapsto \inf\Set{\Kantorovich{\Lip}(\varphi,\psi)}{\varphi\in\StateSpace(\A|\mu),\psi\in\StateSpace(\A|\nu)}
\end{equation*}
is an extended metric on $\StateSpace(\M)$ which metrizes the weak* topology of any tame subset of $\StateSpace(\M)$. In particular, it metrizes the Gel'fand spectrum $\M^\sigma$ of $\M$.
\end{theorem}

Unfortunately, this metric is not quite appropriate for our purpose. Indeed, our goal is to obtain an hypertopology on some class of {\lcqms s} such that two objects are separated by open sets whenever they are not isometrically isomorphic, which we propose should involve the topography, as we shall see in the next section. Now, the notion of {\lcqms} makes fairly weak requirements on the interaction between the Lip norm and the topography. The next section will introduce a natural stronger requirement which appears needed to obtain the desired separation property, and in turn, we obtain a new and more appropriate metric on the topography, as we shall see.

\subsection{Pointed Proper Quantum Metric Spaces}

The noncommutative analogue of a proper metric space is a {\lcqms} which inherits some of the basic properties we shall need from proper metric spaces to ensure that our new hypertopology has the desired separation property. The first consideration is to put a proper metric on the topography of a {\lcqms}. To this end, we propose the following natural approach. We begin with notations which will be used throughout this paper.

\begin{notation}
The Gel'fand spectrum of a commutative C*-algebra $\M$ is denoted by $\M^\sigma$. The set of all compact subsets of $\M^\sigma$ is denoted by $\compacts{\M^\sigma}$.
\end{notation}

\begin{notation}
We define:
\begin{equation*}
\Loc{\A}{\M_\A}{K} = \set{ a \in \lsa{\A}{\M_\A}{K} }{ \Lip(a) < \infty }\text{.}
\end{equation*}
We also set:
\begin{equation*}
\Loc{\A}{\M_\A}{\star} = \bigcup_{K \in \compacts{\M^\sigma}}\Loc{\A}{\M_\A}{K}\text{.}
\end{equation*}
\end{notation}

An element of $\lsa{\A}{\M_\A}{\star}$ may be called locally supported, and an element of $\Loc{\A}{\M_\A}{\star}$ is thus a Lipschitz locally supported element, hence our notation. We avoid the use of the term compactly supported as this notion now depends on a choice of topography within our framework.

Note that if $a\in\A\cap \indicator{K}\A\indicator{K}$, then there exists $b \in \indicator{K}\A\indicator{K}$ such that $a = \corner{K}{b}$. Now:
\begin{equation*}
a = \indicator{K}b\indicator{K} = \indicator{K}\indicator{K} b \indicator{K}\indicator{K} = \indicator{K}a\indicator{K} = \corner{K}{a}
\end{equation*}
so $\Loc{\A}{\M_\A}{K} = \set{ a \in \sa{\A} } { \corner{K}{a} = a\text{ and }\Lip(a)<\infty }$.

We now propose to relate the topography and the Lip norm of a {\lcqms} so that we may employ the restriction of the Lip-norm to the topography to metrize it. In what follows, we shall write the restriction of a Lip-norm to a C*-subalgebra with the same symbol as the original Lip-norm.

\begin{definition}
A {\lcqms} $(\A,\Lip,\M)$ is \emph{standard} when the set:
\begin{equation*}
\set{ a \in \sa{\M} } { \Lip(m) < \infty }
\end{equation*}
is dense in $\sa{\M}$, i.e. when $(\M,\Lip)$ is a Lipschitz pair (where we use the same notation for the restriction of $\Lip$ to $\sa{\M}$ and $\Lip$ itself).
\end{definition}

\begin{notation}
Let $(\A,\Lip,\M)$ be a standard {\lcqms}, so that $(\M,\Lip)$ is a Lipschitz pair. The {\mongekant} on $\StateSpace(\M)$ is denoted by $\sigmaKantorovich{\Lip}$.
\end{notation}

The naturality of our notion of a standard metric space, and an informal justification for choosing this approach to metrizing the topography of a {\lcqms} whenever possible, is established with the following theorem. The formal use of this notion will be seen later on in this paper.

\begin{theorem}
Let $(\A,\Lip,\M)$ be a {\lcqms}. Then $(\A,\Lip,\M)$ is standard if and only if $(\M,\Lip,\M)$ is a {\lcqms}.
\end{theorem}

\begin{proof}
If $(\M,\Lip,\M)$ is a {\lcqms} then $(\M,\Lip)$ is a Lipschitz pair, so $(\A,\Lip,\M)$ is standard.

Suppose now $(\A,\Lip,\M)$ is standard. Let $c,d \in \sa{\M_\A}$ be compactly supported and $\mu \in \StateSpace(\A|\M)$. By Theorem (\ref{lcqms-thm}), the set:
\begin{equation*}
\set{ c a d }{ a\in \sa{\unital{\A}}, \mu(a) = 0 \text{ and }\Lip(a) \leq 1 }
\end{equation*}
is totally bounded in $\sa{\A}$. Consequently, the set:
\begin{equation*}
\set{ c m d }{ m \in \sa{\unital{\M}}, \mu(m) = 0 \text{ and }\Lip(m) \leq 1 }
\end{equation*}
is totally bounded in $\sa{\M}$, so by Theorem (\ref{lcqms-thm}) again, we conclude that $(\M,\Lip,\M)$ is a {\lcqms}.
\end{proof}

We now have a metric structure on the topography of a standard {\lcqms}. As the reader may expect, we will require that for this metric, the spectrum of the topography of a {\pqms} should be a proper metric space.

We also need to relate the Lip-norm of a {\lcqms} with the multiplicative structure of the underlying C*-algebra. Once again, \cite{Latremoliere12b} was developed in the spirit of \cite{Rieffel99} to allow for a high degree of generality and flexibility. As recent research has proven, connecting the metric structure and the multiplicative structure of a quantum metric space proves a powerful asset in the development of our theory \cite{Rieffel05,Rieffel09,Rieffel10,Rieffel10c,Latremoliere13}. In fact, we developed \cite{Latremoliere13b,Latremoliere13c,Latremoliere14} a noncommutative analogue of the Gromov-Hausdorff distance for {\qcms s} whose Lip-norms are Leibniz, in an appropriate sense, and showed that this metric, called the dual Gromov-Hausdorff propinquity, answered several questions in the current research in metric noncommutative geometry. The current paper proposes to extend this metric to {\pqms}, and thus we naturally require a form of Leibniz identity to hold for our Lip-norms as well. We refer to \cite{Latremoliere13b} for a discussion of these matters in the compact case.

We are thus led to introduce:

\begin{definition}
A \emph{Leibniz Lipschitz pair} $(\A,\Lip)$ is a Lipschitz pair such that, for all $a,b \in \sa{\A}$:
\begin{align*}
\Lip\left(\Jordan{a}{b}\right)&\leq \|a\|_\A\Lip(b) + \Lip(a)\|b\|_\A\text{,}\\
\intertext{and}
\Lip\left(\Lie{a}{b}\right)&\leq \|a\|_\A\Lip(b) + \Lip(a)\|b\|_\A\text{.}
\end{align*}
\end{definition}

We now have the following basic ingredients: a metric on the topography of a standard {\lcqms} and a notion of Leibniz identity. In other words, we added a connection between Lip-norms and topography and a connection between Lip-norms and algebra. 

We now can propose our definition for {\pqms}.

\begin{definition}\label{pqms-def}
A {\lcqms} $(\A,\Lip,\M)$ is a \emph{\pqms} when:
\begin{enumerate}
\item $\A$ is separable,
\item $(\A,\Lip)$ is a Leibniz Lipschitz pair,
\item $\Lip$ is lower semi-continuous with respect to the norm topology on $\sa{\A}$,
\item $\Loc{\A}{\M}{\star}$ is norm dense in $\dom{\Lip}$,
\item the set $\Loc{\A}{\M}{\star}\cap\M_\A$ is dense in $\sa{\M_\A}$ (in particular, $(\A,\Lip,\M)$ is standard),
\item $(\M^\sigma,\sigmaKantorovich{\Lip})$ is a proper metric space.
\end{enumerate}
\end{definition}

We note that, for $(\A,\Lip,\M)$ a {\lcqms}, if $(e_n)_{n\in\N}$ is an approximate unit for $\A$ in $\M$, such that $\lim_{n\rightarrow 0}\Lip(e_n) = 0$ and $\|e_n\|_\A\leq 1$ for all $n\in\N$, then $(\A,\Lip,\M)$ is a {\pqms} provided that $\Lip$ is lower semi-continuous and Leibniz.

\begin{example}
If $(X,\mathsf{d})$ is a proper metric space then $(C_0(X),\Lip,C_0(X))$ is a {\pqms} for $\Lip$ the usual Lipschitz seminorm of $\mathsf{d}$.
\end{example}

\begin{example}
In the next section, we shall see that {\Lqcms} are compact examples of {\pqms s}.
\end{example}

\begin{definition}
A \emph{\pqpms} $(\A,\Lip,\M,\mu)$ is a {\pqms} $(\A,\Lip,\M)$ together with a pure state $\mu \in \M^\sigma$. The state $\mu$, identified with a point in $\M^\sigma$, is called the \emph{base point} of $(\A,\Lip,\M,\mu)$.
\end{definition}

We propose some lighter notations when dealing with {\pqms s} in the rest of this paper:
\begin{notation}
Let $(\A,\Lip,\M,\mu)$ be a {\pqpms} and $r > 0$. Since $(\M^\sigma,\sigmaKantorovich{\Lip})$ is proper, the closed ball $\cBall{\M^\sigma}{\sigmaKantorovich{\Lip}}{\mu}{r}$ is compact. We thus employ the following short notations:
\begin{itemize}
\item $ \Loc{\A}{\M_\A}{\mu,r} = \Loc{\A}{\M_\A}{\cBall{\M^\sigma}{\sigmaKantorovich{\Lip}}{\mu}{r}}$,
\item $\StateSpace[\A|\mu,r] = \StateSpace[\A|\cBall{\M^\sigma}{\sigmaKantorovich{\Lip}}{\mu}{r}]$.
\end{itemize}
\end{notation}

We also observe that for a {\pqpms} $(\A,\Lip,\M,\mu)$, we have the identity:
\begin{equation*}
\Loc{\A}{\M_\A}{\star} = \bigcup_{r > 0} \Loc{\A}{\M_\A}{\mu,r}\text{.}
\end{equation*}

We wish to complement our definition of {\pqpms s} with an observation. If $(X,\mathsf{d})$ is proper, then we can find an approximate unit $(e_n)_{n\in\N}$ in $C_0(X)$ with the properties that $\|e_n\|_{C_0(X)}\leq 1$ and $e_n$ is compactly supported for all $n\in\N$ and $\lim_{n\rightarrow \infty}\Lip(e_n) = 0$, where $\Lip$ is the Lipschitz seminorm of $\mathsf{d}$. To see this, simply consider, for some $x\in X$:
\begin{equation}\label{proper-local-eq0}
e_n : x \in X \mapsto \frac{ \mathsf{d}(x,X\setminus \cBall{X}{}{x}{n}) }{\mathsf{d}(x,X\setminus\cBall{X}{}{x}{2n}) + \mathsf{d}(x,\cBall{X}{}{x}{n})}
\end{equation}
for all $n\in\N$ and note that $\Lip(e_n) \leq\frac{1}{n}$ for all $n\in\N$. We will prove that the existence of such an approximate unit is in fact equivalent to $X$ being proper.

It would be natural to require such an approximate unit to exist in the noncommutative case. However, it appears that our Definition (\ref{pqms-def}) is strong enough to carry out the construction of our hypertopology. On the other hand, requiring the existence of a natural approximate unit inspired by the theory of classical proper metric spaces can be use to provide a small improvement in the separation properties of our hypertopology, so we now formalize this idea.

\begin{definition}\label{strongly-proper-def}
A triple $(\A,\Lip,\M)$ is \emph{strongly proper} quantum metric space when:
\begin{enumerate}
\item $\Lip$ is defined on a dense subset of $\A$,
\item $\A$ is separable, and $\Lip$ is lower semi-continuous,
\item $(\A,\Lip,\M)$ is a standard {\lcqms} (where we identify $\Lip$ with its restriction to $\sa{\A}$),
\item for all $a,b\in\A$ we have:
\begin{equation*}
\Lip(ab)\leq\|a\|_\A \Lip(b) + \|b\|_\A\Lip(a)\text{,}
\end{equation*}
\item there exists an approximate unit $(e_n)_{n\in\N}$ in $\sa{\M}$ for $\A$ such that for all $n\in\N$, we have $\|e_n\|_\A\leq 1$ and $e_n \in \Loc{\A}{\M_\A}{\star}$, $\lim_{n\rightarrow\infty} \Lip(e_n) = 0$.
\end{enumerate}
\end{definition}

We shall always identify the seminorm of a strongly proper {\lcqms} with its restriction to the self-adjoint part of its domain when appropriate. Strongly proper {\lcqms s} are {\pqms s}:

\begin{proposition}\label{strongly-proper-implies-proper-prop}
A strongly proper quantum metric space $(\A,\Lip,\M)$ is a {\pqpms}, and moreover for any $a\in\sa{\A}$ with $\Lip(a)<\infty$ there exists a sequence $(a_n)_{n\in\N}$ with $a_n\in\Loc{\A}{\M}{\star}$ for all $n\in\N$, converging to $a$ in norm and such that $\lim_{n\rightarrow\infty}\Lip(a_n) = \Lip(a)$. If moreover $a\in\M_\A$ then we can choose $a_n \in \M_\A$ for all $n\in\N$.
\end{proposition}
\begin{proof}
Let $(\A,\Lip,\M)$ be a strongly proper quantum metric space, and let $(e_n)_{n\in\N}$ be an approximate unit with the properties listed in Definition (\ref{strongly-proper-def}). Thus $(\A,\Lip)$ is a Leibniz Lipschitz pair \cite[Proposition 2.17]{Latremoliere13} and $\A$ is separable. Moreover, $\Lip$ is lower semi-continuous. For all $n\in\N$, let $K_n\in\compacts{\M^\sigma}$ such that $e_n \in \Loc{\A}{\M}{K_n}$.

Let $a\in\sa{\A}$ with $\Lip(a)<\infty$ and let $\varepsilon > 0$. There exists $N\in\N$ such that for all $n\geq N$, we have $\|e_n a e_n - a\|_{\A} < \varepsilon$ since $(e_n)_{n\in\N}$ is an approximate unit for $\A$. 

Now, there exists $M\in\N$ such that $\Lip(e_n)\leq \frac{\varepsilon}{2\|a\|_\A + 1}$ for all $n\geq M$. Thus, using the Leibniz property:
\begin{equation}
\Lip(e_n a e_n) \leq 2\Lip(e_n)\|a\|_\A + \Lip(a) \leq \varepsilon + \Lip(a)\text{.}
\end{equation}
Since $e_nae_n \in \Loc{\A}{\M}{K_n}$ for all $n\in\N$, we conclude that for any $n\geq\max\{N,M\}$, we have found $a_n=e_nae_n \in \Loc{\A}{\M}{\star}$ with $\|a-a_n\|_\A\leq\varepsilon$ and $\Lip(a_n)\leq\Lip(a)+\varepsilon$. Thus $\limsup_{n\rightarrow\infty}\Lip(a_n)\leq \Lip(a)$. Yet since $\Lip$ is lower semi-continuous, $\Lip(a)\leq\liminf_{n\rightarrow\infty}\Lip(a_n)$, so:
\begin{equation*}
\lim_{n\rightarrow\infty}\Lip(a_n) = \Lip(a).
\end{equation*}

Moreover, if $a\in\sa{\M}$ then $e_nae_n\in\Loc{\A}{\M}{\star}\cap\M$. 

It remains to prove that $\M^\sigma$ is proper. Let $r > 0$ and let $x \in \M^\sigma$. Let $B = \cBall{\M^\sigma}{\sigmaKantorovich{\Lip}}{x}{r}$. Since $(e_n)_{n\in\N}$ is an approximate unit in $\M$ uniformly bounded in norm by $1$, we conclude that $(e_n(x))_{n\in\N}$ converges to $1$ and $e_n(y)\leq 1$ for all $y\in\M^\sigma$. Let $N\in\N$ such that for all $n\geq N$, we have $e_n(x) > \frac{2}{3}$.

Now, let $M\in\N$ such that $\Lip(e_n)\leq \frac{1}{3r}$ for all $n\geq M$ and let $n\geq \max{N,M}$. By definition of $\sigmaKantorovich{\Lip}$, we thus have:
\begin{equation*}
|e_n(y)-e_n(x)|\leq \frac{1}{3r}\sigmaKantorovich{\Lip}(x,y)
\end{equation*}
for all $x,y\in \M^\sigma$; thus for all $y\in B$ we conclude that $e_n(y) > \frac{1}{3}$.

Therefore, $B \subseteq e_n^{-1}([\frac{1}{3},1])$ --- where we see $e_n$ as a function on $M^\sigma$ by Gel'fand-Naimark theorem. By assumption, $e_n$ has compact support, and is continuous, so $e_n^{-1}([\frac{1}{3},1])$ is compact. Thus $B$, as a closed subset of a compact, is compact as well. Thus $(\M^\sigma,\Kantorovich{\Lip})$ is proper. This completes our proof.

\end{proof}
We are not aware of a converse to Proposition (\ref{strongly-proper-implies-proper-prop}). A difficulty stems from the fact that the restriction of the Lip-norm $\Lip$ to the topography $\M$ of a {\lcqms} $(\A,\Lip,\M)$ needs not be the standard Lipschitz seminorm for the associated metric $\sigmaKantorovich{\Lip}$: it is generically greater. Thus, even though there exists an approximate unit with Lipschitz seminorm converging to $0$ in the topography of any {\pqms} by Expression (\ref{proper-local-eq0}), it is unclear whether the Lip-norms of any such sequence converges to $0$ as well. Note however that we have shown that a locally compact metric space $X$ is proper if and only if there exists an approximate unit in $C_0(X)$ with the properties listed in Assertion (5) of Definition (\ref{strongly-proper-def}).

The commutative case suggests, in fact, that we impose an additional constraint to the approximate units which appear in Assertion (5) of Definition (\ref{strongly-proper-def}). Such a restriction implies a natural consequence about ``compact quantum metric subspaces'', as we shall now prove. We include these observations here, though this class of {\pqms s} does not seem to play a role at the level of generality of our current theory.
\begin{definition}
A strongly proper {\lcqms} $(\A,\Lip,\M)$ is \emph{localizable} when there exists an approximate unit $(e_n)_{n\in\N}$ in $\M$ for $\A$ such that:
\begin{itemize}
\item for all $n\in\N$, the element $e_n$  is a compactly supported as a function of $\M^\sigma$ with $0\leq e_n(x)\leq 1$ for all $x\in \M^\sigma$,
\item $\lim_{n\rightarrow\infty} \Lip(e_n) = 0$,
\item for all $n\leq m$ we have $e_n e_m = e_n$.
\end{itemize}
\end{definition}

\begin{proposition}
Let $(\A,\Lip,\M)$ be a localizable {\lcqms}. Then for all $K\in\compacts{\M}$, the pair $(\corner{K}{\sa{\A}},\Lip_K)$, where $\Lip_K$ is the quotient seminorm:
\begin{equation*}
\Lip_K : b\in\corner{K}{\sa{\A}} \mapsto \inf\set{\Lip(a)}{\corner{K}{a} = b}
\end{equation*}
of $\Lip$, is a compact quantum metric space in the sense of \cite{Rieffel98a,Rieffel99}.
\end{proposition}

\begin{proof}
Let $\Lip_K$ be the quotient of $\Lip$ to $\corner{K}{\sa{\A}}$ (seen as an order unit subspace of $\A^{\ast\ast}$ with order-unit $\indicator{K}$) for some $K\in\compacts{\M^\sigma}$. Now, since $(e_n)_{n\in\N}$ is an approximate unit for $\A$, the sequence $(e_n)_{n\in\N}$ converges pointwise to $1$ on $K$. Since $e_n e_m = e_n$ if $n\leq m$, we note that in fact, $e_m(x) = 1$ for all $x$ in the support of $e_n$ when $n < m$. We also note that $e_n e_m = e_n$ and $0\leq e_n\leq 1$ implies that $(e_n)_{n\in\N}$ is increasing. Thus, by Dini's theorem, $(e_n)_{n\in\N}$ converges uniformly to $1$ on the compact $K$.

From this we deduce that for some $N\in\N$, we have $e_n(x) = 1$ for all $n\geq N$ and $x\in K$. Thus $\corner{K}{e_n} = \indicator{K}$. Yet $\lim_{n\rightarrow \infty}\Lip(e_n) = 0$. We conclude that $\Lip_K(\indicator{K}) = 0$.

Thus $(\corner{K}{\sa{\A}},\Lip_K)$ is a Lipschitz pair (the density of the domain of $\Lip_K$ is easy to check as well: if $b\in\corner{K}{\sa{\A}}$ then there exists, for all $\varepsilon>0$, some $a\in\sa{\A}$ with $\Lip(a)<\infty$ and some $c\in\sa{\A}$ with $\corner{K}{c} = b$ and $\|a-c\|_\A\leq\varepsilon$ so $\Lip_K(\corner{K}{c})<\infty$ and $\|b  - \corner{K}{c}\|_{\corner{K}{\sa{\A}}}\leq\varepsilon$). Theorem (\ref{lcqms-thm}) and \cite[Theorem 1.9]{Rieffel98a} allows us to conclude that $(\corner{K}{\sa{\A}},\Lip_K)$ is a compact quantum metric space.
\end{proof}

In general, we do not have an argument to prove that a {\pqms} is localizable. We also will not need this property in our current work.

We now address the question of isometries between {\pqpms s}. Ideally, we wish such an object to be a morphism between the underlying C*-algebras whose dual map is an isometry and maps the base points to each other. However, the notion of morphism between non-unital C*-algebras requires some care: typically, one would want a morphism from $\A$ to $\B$ to mean a *-morphism from $\A$ to the multiplier C*-algebra of $\B$. Regular *-morphisms from $\A$ to $\B$ are the noncommutative analogues of proper maps. 

However, if $f : X \rightarrow Y$ is an isometry from a \emph{proper} metric space $(X,\mathsf{d}_X)$ to a metric space $(Y,\mathsf{d}_Y)$, and if $K\in\compacts{Y}$ is compact, then $f^{-1}(K)$ is bounded since $K$ is bounded and $f$ is an isometry. Moreover, $f^{-1}(K)$ is closed since $K$ is closed as a compact subset of a Hausdorff space. Consequently, $f^{-1}(K)$ is compact in $(X,\mathsf{d}_X)$ since the latter is proper. Thus isometries from proper metric spaces are proper maps. We can thus expect an isometry between {\pqpms s} to be a regular *-morphism, and even ideally a topographic morphism.

Moreover, as discussed in \cite{Rieffel00}, the notion of isometry between {\qcms s} is best handled when the Lip-norms involved are lower semi-continuous. We may expect that an isometry, in our context, should mean a map which preserves the Lip-norms. However, it is actually enough to preserve Lip-norms on some dense subset of their domain, in order to obtain an isometry between the associated {\mongekant}. This allows us to propose the following strong notion of isometric isomorphism between {\pqpms s}:

\begin{definition}
Let $\mathds{A} = (\A,\Lip_\A,\M_\A)$ and $\mathds{B} = (\B,\Lip_\B,\M_\B)$ be two {\pqpms s}. An isometric isomorphism from $\mathds{A}$ to $\mathds{B}$ is a topographic morphism $\pi : (\A,\M_\A) \rightarrow (\B,\M_\B)$ such that:
\begin{enumerate}
\item $\pi$ is a *-isomorphism onto $\B$,
\item $\pi$ restrict to a Jordan-Lie isomorphism from $\Loc{\A}{\M_\A}{\star}$ to $\Loc{\B}{\M_\A}{\star}$,
\item $\pi(\M_\A) = \M_\B$,
\item for all $a\in\Loc{\A}{\M_\A}{\star}$, we have $\Lip_\B\circ\pi(a) = \Lip_\A(a)$,
\item $\mu_\B\circ\pi = \mu_\A$.
\end{enumerate}
\end{definition}

Our purpose in this paper is to define a hypertopology on the class of all {\pqpms s} which is Hausdorff modulo isometric isomorphism. We also wish to relate our new topology to the dual propinquity topology \cite{Latremoliere13b,Latremoliere14}. The next section will only address very briefly the connection between {\Lqcms s} and {\pqpms s}, in order to set up our later comparison.

\subsection{\Lqcms s}

A special case of {\pqpms} is given by {\Lqcms s}, which are compact quantum metric spaces of a special kind appropriate for the construction of the dual Gromov-Hausdorff propinquity, and including many important examples of {\qms s} such as the quantum tori, classical compact metric spaces, hyperbolic group C*-algebras, and any quantum compact metric space obtained, for instance, from a spectral triple. This class was introduced in \cite{Latremoliere13} as a general framework for the construction of the quantum propinquity, and then served as the framework for the dual propinquity as well. It generalizes Rieffel's notion of a compact C*-metric space.

\begin{definition}
A {\Lqcms} $(\A,\Lip)$ is a Leibniz Lipschitz pair where $\A$ is unital, $\Lip$ is lower semi-continuous with respect to $\|\cdot\|_\A$, and such that $\Kantorovich{\Lip}$ metrizes the weak* topology of $\StateSpace(\A)$.
\end{definition}

The key remark, which we will use implicitly when needed, is:
\begin{remark}
Let $(\A,\Lip)$ be a {\Lqcms}. Let $\M_\A = \C\unit_\A$ and $\mu : t\unit_\A\mapsto t$. Then $(\A,\Lip,\M_\A,\mu)$ is a {\pqpms}, which we call the \emph{canonical {\pqpms} associated with $(\A,\Lip_\A)$}. When no confusion may arise, we may identify $(\A,\Lip)$ with this {\pqpms}.
\end{remark}

We must point out that there is no reason why we must choose the trivial topography $\C\unit_\A$ in a {\Lqcms} $(\A,\Lip)$. Many other topographies can be chosen in general, and they do play a role in our work. Indeed, if one wishes {\Lqcms s} to approximate a {\pqpms} $\mathds{A}$ which is not compact, then the topographies for our {\Lqcms s} will have to converge appropriately the final topography of $\mathds{A}$. Thus, there are many {\pqpms s} associated to a single {\Lqcms}, though the one we identified here is of course special.

We now propose an important observation connecting {\pqpms s} and {\Lqcms s}:

\begin{proposition}\label{compact-liploc-prop}
If $(\D,\Lip_\D,\M_\D)$ is a {\lcqms} with $\Lip_\D$ norm lower semi-continuous, then $\D$ is unital if, and only if there exists a compact subset $K$ of $\M_\D^\sigma$ such that the set:
\begin{equation*}
\LipLoc{\D}{\M_\D}{K} = \set{ a \in \Loc{\D}{\M_\D}{K} }{ \Lip_\A(a)\leq 1 }
\end{equation*}
is not compact in $\sa{\D}$.
\end{proposition}

\begin{proof}
If $\D$ is unital then $\unit_\D \in \M_\D$ --- in particular, $\M_\D^\sigma$ is compact. Let $K=\M_\D^\sigma$. Then $\unit_\D \in \Loc{\D}{\M_\D}{K} = \sa{\D}$ and since $\Lip_\D(\unit_\D) = 0$, we conclude that $\R\unit_\D\subseteq \LipLoc{\D}{\M_\D}{K}$; hence $\LipLoc{\D}{\M_\D}{K}$ is unbounded in norm and a fortiori not compact.

If $\D$ is not unital, then let $K\in\compacts{\M_\D^\sigma}$ be chosen. Since $K\not=\M_\D^\sigma$ as $\D$ is not unital, there exists $x\in \M_\D^\sigma$ with $x\not\in K$. Let $\delta = \sigmaKantorovich{\Lip_\D}(x,K) $; as $K$ is compact, $\delta > 0$. Let $\varphi \in \StateSpace(\D)$ be a positive Hahn-Banach extension of $x$; since $\M_\D$ contains an approximate unit, $\varphi\in\StateSpace(\D|\M_\D)$.

Now, by definition, for all $a\in\Loc{\D}{\M_\D}{K}$:
\begin{equation*}
|\varphi(a)| = |\varphi(p a p)| \leq \sqrt{\varphi(p)}\|ap\|_\D = \sqrt{p(x)}\|ap\|_\A = 0 \text{.}
\end{equation*}
Thus:
\begin{equation}\label{liploc-ineq0}
\LipLoc{\D}{\M_\D}{K} \subseteq \left\{ p a p \in \sa{\D} : \Lip_\D(a)\leq 1\text{ and }\varphi(a) = 0 \right\}
\end{equation}
and since $(\D,\Lip_\D,\M_\D)$ is a {\lcqms}, the set on the right hand-side of Inequality (\ref{liploc-ineq0}) is norm compact by Theorem (\ref{lcqms-thm}).

On the other hand, if $(a_n)_{n\in\N}$ is a sequence in $\LipLoc{\D}{\M_\D}{K}$ converging in norm to some $a\in\sa{\D}$, then since $\Loc{\D}{\M_\D}{K}$ is closed, we conclude that $a\in\Loc{\D}{\M_\D}{K}$, and since $\Lip_\D$ is lower semi-continuous, we conclude that $\Lip_\D(a) \leq 1$. Thus $a\in\LipLoc{\D}{\M_\D}{K}$. As a closed subset of a compact set, $\LipLoc{\D}{\M_\D}{K}$ is thus compact. 
\end{proof}

We refer to \cite{Latremoliere13,Latremoliere13b,Latremoliere13c,Latremoliere14} for further examples and discussions of this important class of {\qcms s} and the analogues of the  Gromov-Hausdorff distances for these spaces.

\section{Tunnels}

We now begin the construction of our hypertopology on the class of all {\pqpms s}. Our construction will actually allow for some flexibility, where we could restrict our attention to sub-classes of {\pqpms s} for which the construction we propose may lead to a stronger metric. 

The basic ingredient of our construction are tunnels. A tunnel is a mean to relate two {\pqpms s} in such a manner as allowing for the computation of a nonnegative number which we then may use to define our distance. This number ought to capture information akin to Definition (\ref{delta-r-def}). 

The challenge which we encountered during this research is the following. In all the constructions of some noncommutative analogues of the Gromov-Hausdorff in the compact setting, a notion of isometric embedding between {\qcms s} is employed. In our more general setup, however, this notion alone appears to raise difficulties. Indeed, the notion of {\lcqms s} provides us only with information on locally supported elements, as shown with Theorem (\ref{lcqms-thm}). Thus, our entire work must involve such elements. 

We could yet add a constraint to the notion of admissibility used, for instance, in \cite{Latremoliere13b} when defining tunnels, to ensure that we work only with locally supported elements. However, when exploring this matter, we soon realized that such a constraint is actually the core ingredient, and the usual notion of admissibility can be considerably weakened as long as it meets our new demands. As we shall note, when working with a compact {\pqpms}, our new notion recovers the usual notion of tunnel in \cite{Latremoliere13b}. Thus, this new approach was hidden from us, and required some effort to unearth. It is also rather surprising that we do not explicitly require the notion of isometry as part of our tunnels.

\subsection{Passages and Extent}

A tunnel will be a special king of passage:

\begin{definition}\label{passage-def}
Let $\mathds{A}_1=(\A_1,\Lip_1,\M_1,\mu_1)$ and $\mathds{A}_2=(\A_2,\Lip_2,\M_2,\mu_2)$ be two {\pqpms s}. A \emph{passage} $(\D,\Lip_\D,\M_\D,\pi_1,\mathds{A}_1,\pi_2,\mathds{A}_2)$ from $\mathds{A}_1$ to $\mathds{A}_2$ is a {\lcqms} $(\D,\M_\D,\Lip_\D)$ such that:
\begin{enumerate}
\item the Lip-norm $\Lip_\D$ is lower semi-continuous with respect to the norm $\|\cdot\|_\D$ of $\D$,
\item $\pi_\A$ and $\pi_\B$ are topographic morphisms.
\end{enumerate}
\end{definition}

\begin{notation}
Let $\mathds{A}$ and $\mathds{B}$ be two {\pqpms s}. If $\tau$ is a passage from $\mathds{A}$ to $\mathds{B}$, then the \emph{domain} $\dom{\tau}$ of $\tau$ is $\mathds{A}$ while the \emph{co-domain} $\codom{\tau}$ of $\tau$ is $\mathds{B}$.
\end{notation}

A tunnel is a passage which, for some $r > 0$, satisfies a weak form of ``local'' admissibility. By local, we shall mean that given a {\lcqms} $(\A,\Lip,\M)$, we will work with elements in $\Loc{\A}{\M_\A}{\star}$. We note that, for any $K\in\compacts{\M^\sigma}$, we have two natural objects we could consider. On the one hand, $\corner{K}{\sa{\A}}$ is an order unit space, and its state space is $\StateSpace[\A|K]$, so it seems to capture a local aspect of our metric space. However, it is not well-behaved with respect to the multiplication. Moreover, it is not a sub-object of $\A$, but rather a quotient, and we must beware to employ quotients when dealing with Leibniz Lip-norms. We also note that with our current definition of a {\pqpms}, $\corner{K}{\sa{\A}}$ is not necessarily a {\qcms} for the quotient of $\Lip$ (as $\indicator{K}$ is the order unit and may not have Lip-norm zero). This matter, actually, could be resolved if we assumed the existence of an approximate unit as described in our section of {\pqms s}, but we leave these considerations for our later work.

The other natural object is $\Loc{\A}{\M_\A}{K}$. This is a sub-object of $\sa{\A}$, and it is even a Jordan-Lie subalgebra. On the other hand, it is usually not an order unit space (it has usually no unit), and thus we can not hope to apply to it our work on compact quantum metric spaces directly. Besides, being a sub-object means that it technically correspond to a ``quantum quotient space'' rather than a subspace, as one would expect when working with an analogue of the Gromov-Hausdorff distance.

Our work, it turns out, will employ $\Loc{\A}{\M_\A}{K}$ for its algebraic properties, and $\StateSpace[\A|K]$ for its metric properties. The difficulty is to balance the use of these objects.

To understand the notion of weak admissibility for tunnels, we first introduce the notion of a lift set and a tunnel set for a passage. 

\begin{definition}\label{lift-set-def}
Let $\mathds{A} = (\A,\Lip_\A,\M_\A,\mu_\A)$ and $\mathds{B} = (\B,\Lip_\B,\M_\B,\mu_\B)$ be two {\pqpms s}. Let $\tau = (\D,\Lip_\D,\M_\D,\pi_\A,\mathds{A},\pi_\B,\mathds{B})$ be a passage from $\mathds{A}$ to $\mathds{B}$. If $K\in\compacts{\M_\D^\sigma}$, $l > 0$, $\varepsilon > 0$, $r > 0$ and $a\in\sa{\A}$ with $\Lip_\A(a)\leq l$, then the \emph{lift set} of $a$ for $\tau$ associated with $(l,r,\varepsilon,K)$ is:
\begin{equation*}
\liftsettunnel{\tau}{a}{l,r,\varepsilon,K} = \Set{d\in\Loc{\D}{\M_\D}{K}}{&\pi_\A(d) = a,\\ &\pi_\B(d) \in \Loc{\B}{\M_\B}{\mu_\B,r+4\varepsilon},\\ &\Lip_\D(d)\leq l}\text{,}
\end{equation*}
and the \emph{target set} of $a$ for $\tau$ associated with $(l,r,K,\varepsilon)$ is:
\begin{equation*}
\targetsettunnel{\tau}{a}{l,r,\varepsilon,K} = \pi_\B\left(\liftsettunnel{\tau}{a}{l,r,\varepsilon,K}\right)\text{.}
\end{equation*}
\end{definition}

Of course, as defined, a lift set for some passage may be empty. The key to our notion of tunnel is, indeed, related to when lift sets are not empty. In essence, our notion of admissibility for a passage, which depends on the choice of some radius $r > 0$, relies on whether one can lift elements locally supported on a ball centered at the base point and of radius $r$ for that passage. Formally:

\begin{definition}\label{left-admissible-def}
Let $\mathds{A} = (\A,\Lip_\A,\M_\A,\mu_\A)$ and $\mathds{B} = (\B,\Lip_\B,\M_\B,\mu_\B)$ be two {\pqpms s}. Let $\tau = (\D,\Lip_\D,\M_\D,\pi_\A,\mathds{A},\pi_\B,\mathds{B})$ be a passage from $\mathds{A}$ to $\mathds{B}$. Let $r > 0$. A pair $(\varepsilon,K)$, where $\varepsilon > 0$ and $K\in\compacts{\M_\D^\sigma}$, is \emph{$r$-left admissible} when:
\begin{enumerate}
\item $\pi_\A^\ast\left(\StateSpace[\A|\mu_\A,r]\right) \subseteq \StateSpace[\D|K]$,
\item we have:
\begin{equation*}
\StateSpace[\D|K]\almostsubseteq{(\StateSpace(\D),\Kantorovich{\Lip_\B})}{\varepsilon} \pi_\A^\ast\left(\StateSpace[\A|\mu_\A,r+4\varepsilon]\right)\text{,}
\end{equation*}
\item for all $a\in\Loc{\A}{\M_\A}{\mu_\A,r}$, we have:
\begin{equation*}
\targetsettunnel{\tau}{a}{\Lip_\A(a),r,\varepsilon,K} \not= \emptyset\text{,}
\end{equation*}
\item for all $a\in\Loc{\A}{\M_\A}{\mu_\A,r}\cap \M_\A$, we have:
\begin{equation*}
\targetsettunnel{\tau}{a}{\Lip_\A(a),r,\varepsilon,K}\cap\M_\B\not=\emptyset\text{,}
\end{equation*}
\item for all $d\in\Loc{\D}{\M_\D}{K}$, we have $\Lip_\B(\pi_\B(d)) \leq \Lip_\D(d)$.
\end{enumerate}
\end{definition}

We invite the reader to compare this notion with Theorem (\ref{GH-equivalences-thm}). 

\begin{remark}
Let $\tau=(\D,\Lip,\M,\ldots)$ be a passage from $(\A,\Lip_\A,\M_\A,\mu_\A)$ and let $a\in\Loc{\A}{\M_\A}{\mu_\A,r}$ for some $r > 0$. If $r' \geq r$, $\varepsilon' \geq \varepsilon > 0$, $l ' \geq l \geq \Lip_\A(a)$, and $K\subseteq K' \in \compacts{\M^\sigma}$, then:
\begin{equation*}
\liftsettunnel{\tau}{a}{l,r,\varepsilon,K}\subseteq \liftsettunnel{\tau}{a}{l',r',\varepsilon',K'}
\end{equation*}
and therefore:
\begin{equation*}
\targetsettunnel{\tau}{a}{l,r,\varepsilon,K}\subseteq \targetsettunnel{\tau}{a}{l',r',\varepsilon',K'}\text{.}
\end{equation*}
This observation will be used repeatedly. In particular, if $\targetsettunnel{\tau}{a}{\Lip_\A(a),r,\varepsilon,K}$ is not empty, then neither is $\targetsettunnel{\tau}{a}{l,r,\varepsilon,K}$ for $l\geq \Lip_\A(a)$.
\end{remark}

We observe that under condition of left admissibility, lift sets and tunnel sets are topologically well behaved:

\begin{lemma}\label{compact-lift-set-lemma}
Let $\mathds{A} = (\A,\Lip_\A,\M_\A,\mu_\A)$ and $\mathds{B} = (\B,\Lip_\B,\M_\B,\mu_\B)$ be two {\pqpms s}. Let $\tau = (\D,\Lip_\D,\M_\D,\pi_\A,\mathds{A},\pi_\B,\mathds{B})$ be a passage from $\mathds{A}$ to $\mathds{B}$. If $K\in\compacts{\M_\D^\sigma}$, $\varepsilon > 0$, and $r > 0$ are chosen so that $(\varepsilon,K)$ is $r$-left admissible for $\tau$, and if $a\in\Loc{\A}{\M_\A}{\mu_\A,r}$ and $l\geq\Lip_\A(a)$, then the lift set $\liftsettunnel{\tau}{a}{l,r,\varepsilon,K}$ is a nonempty compact subset of $\lsa{\A}{\M_\A}{K}$ and the target set $\targetsettunnel{\tau}{a}{l,r,\varepsilon,K}$ is a nonempty compact subset of $\lsa{\B}{\M_\B}{r+4\varepsilon}$.
\end{lemma}

\begin{proof}
The sets $\liftsettunnel{\tau}{a}{l,r,\varepsilon,K}$ and $\targetsettunnel{\tau}{a}{l,r,\varepsilon,K}$ are not empty by Definition (\ref{left-admissible-def}) and our choice of $(\varepsilon,K)$ and $a\in\Loc{\A}{\M_\A}{\mu_\A,r}$. Moreover $\targetsettunnel{\tau}{a}{l,r,\varepsilon,K}\subseteq\Loc{\B}{\M_\B}{\mu_\B,r+4\varepsilon}$ and $\liftsettunnel{\tau}{a}{l,r,\varepsilon,K}\subseteq\Loc{\D}{\M_\D}{K}$ by Definition (\ref{lift-set-def}). 

We differentiate two cases. 

First, we work with the assumption that $\LipLoc{\D}{\M_\D}{K}$ is compact, i.e. $K\subsetneq \M^\sigma_\D$. Let $(d_n)_{n\in\N}$ be a sequence in $\liftsettunnel{\tau}{a}{l,r,\varepsilon,K}$. Then, $(d_n)_{n\in\N}$, a sequence in the compact set $l\cdot\LipLoc{\D}{\M_\D}{K}$, admits a convergent subsequence $(d_{f(n)})_{n\in\N}$. Let $d\in\Loc{\D}{\M_\D}{K}$ be its limit.

Now, we have:
\begin{itemize}
\item by continuity of $\pi_\A$, we have $\pi_\A(d) = a$,
\item by lower semi-continuity of $\Lip_\D$, we have $\Lip_\D(d)\leq l$.
\item since $\lsa{\B}{\M_\B}{\mu_\B,r+4\varepsilon}$ is closed and $\pi_\B$ is continuous, we have $\pi_\B(d) \in \lsa{\B}{\M_\B}{\mu_\B,r+4\varepsilon}$. Now, since $\Lip_\B(\pi_\B(d))\leq \Lip_\D(d) \leq l$ by Definition (\ref{left-admissible-def}), we conclude that:
\begin{equation*}
\pi_\B(d) \in \Loc{\B}{\M_\B}{\mu_\B,r+4\varepsilon}\text{.}
\end{equation*}
\end{itemize}

Consequently, $d\in\liftsettunnel{\tau}{a}{l,r,\varepsilon,K}$. Thus $\liftsettunnel{\tau}{a}{l,r,\varepsilon,K}$, and by continuity of $\pi_\B$, the set $\targetsettunnel{\tau}{a}{l,r,\varepsilon,K}$ are norm compact. This completes our proof for this case.

Second, assume $\LipLoc{\D}{\M_\D}{K}$ is not compact. By Proposition (\ref{compact-liploc-prop}), we conclude that $\D$ is unital and $\M_\D^\sigma = K$, with $\unit_\D \in \LipLoc{\D}{\M_\D}{K}$. Moreover, $\A$ is unital as well since $\pi_\A$ is a topographic morphism.

In particular, $(\D,\Lip_\D)$ is a compact quantum metric space, and so is $(\A,\Lip_\A)$. Now, let $\varphi \in \StateSpace(\A)$. We observe that:
\begin{multline*}
\left\{ d \in  \sa{\D} :  \Lip_\D(d) \leq l, \pi_\A(d) = a \right\} \\
\subseteq \set{ d\in \sa{\D} }{ \Lip_\D(d) \leq l, \varphi\circ\pi_\A(d - \varphi(a)\unit_\D) = 0 }\\
\subseteq \set{ d\in \sa{\D} }{\Lip_\D(d) \leq l, \varphi\circ\pi_\A(d) = 0 } + \varphi(a)\unit_\D\text{,}
\end{multline*}
and the latter set is norm compact. The target set $\targetsettunnel{\tau}{a}{l,r,\varepsilon,K}$ is, again, a closed subset of $\set{ d \in \sa{\D} }{ \Lip_\D(d) \leq l, \pi_\A(d) = a }$, hence our lemma holds.
\end{proof}

Right admissibility will be based upon the following natural notion:
\begin{notation}
Let $\tau = (\D,\Lip,\M,\pi,\dom{\tau},\rho,\codom{\tau})$ be a passage. The \emph{inverse passage} $\tau^{-1}$ of $\tau$ is the passage $(\D,\Lip,\M,\rho,\codom{\tau},\pi,\dom{\tau})$.
\end{notation}

Admissibility is more than just left and right admissibility. We begin with the notion of admissible pair, and note that we add some algebraic conditions, and involve the base points as well.

\begin{definition}\label{admissible-pair-def}
Let $\tau = (\D,\Lip_\D,\M_\D,\pi,\dom{\tau},\rho,\codom{\tau})$ be a passage and let $r > 0$. 
\begin{enumerate}
\item A pair $(\varepsilon,K)$ is \emph{$r$-right admissible} when $(\varepsilon,K)$ is $r$-left admissible for $\tau^{-1}$,
\item A pair $(\varepsilon,K)$ is \emph{$r$-admissible} when:
\begin{itemize}
\item $(\varepsilon,K)$ is both $r$-left and $r$-right admissible for $\tau$,
\item if $\mu$ and $\nu$ are the respective base points of $\dom{\tau}$ and $\codom{\tau}$, then $\sigmaKantorovich{\Lip_\D}(\mu\circ\pi,\nu\circ\rho)\leq\varepsilon$,
\item for all $d,d'\in\Loc{\D}{\M_\D}{K}$, we have:
\begin{equation}\label{admissible-Jordan-eq}
\Lip_\D\left(\Jordan{d}{d'}\right)\leq \Lip_\D(d)\|d'\|_\D+\Lip_\D(d')\|d\|_\D\text{,}
\end{equation} 
\item for all $d,d'\in\Loc{\D}{\M_\D}{K}$, we have:
\begin{equation}\label{admissible-Lie-eq}
\Lip_\D\left(\Lie{d}{d'}\right)\leq \Lip_\D(d)\|d'\|_\D+\Lip_\D(d')\|d\|_\D\text{,}
\end{equation}
\end{itemize}
\end{enumerate}
\end{definition}

\begin{remark}
Let $(\varepsilon,K)$ be an $r$-admissible pair for a passage:
\begin{equation*}
\tau = (\D,\Lip,\M,\pi,\dom{\tau},\rho,\codom{\tau})\text{.}
\end{equation*}
Since $\lsa{\D}{\M_\D}{K}$ is a Jordan-Lie subalgebra of $\sa{\D}$, we conclude that Conditions (\ref{admissible-Jordan-eq}) and (\ref{admissible-Lie-eq}) together imply that $\Loc{\D}{\M_\D}{K}$ is a Jordan-Lie subalgebra of $\lsa{\D}{\M_\D}{K}$.
\end{remark}

Now, a difficulty we shall try and avoid is that, for different radii, a passage may be admissible, yet the compacts involved in Definition (\ref{admissible-pair-def}) may not be related to each other. We wish to avoid this matter for technical reasons, but informally, we could simply state that we want the noncommutative analogues of the $\delta_r$ from Definition (\ref{delta-r-def}) be monotone in $r$, as it will prove very useful and is rather natural. Thus, we can define our notion of admissibility in our context as follows:

\begin{definition}\label{admissible-def}
A number $\varepsilon > 0$ is $r$-admissible when there exists a family $(K_t)_{t \in (0,r]}$ of compacts of $\M_\D^\sigma$ such that:
\begin{enumerate}
\item for all $t\leq t' \in (0,r]$, we have $K_t\subseteq K_{t'}$,
\item for all $t \in (0,r]$, the pair $(\varepsilon,K_t)$ is $t$-admissible.
\end{enumerate}
\end{definition}

\begin{notation}\label{admissible-notation}
The set of $r$-admissible numbers for some passage $\tau$ is denoted by $\Adm(\tau|r)$.
\end{notation}

\begin{remark}\label{inclusion-adm-rmk}
By Definition (\ref{admissible-def}), if $R \geq r$ then:
\begin{equation*}
\Adm(\tau|R) \subseteq \Adm(\tau|r)\text{,}
\end{equation*}
and
\begin{equation*}
\Adm(\tau^{-1}|r) = \Adm(\tau|r)\text{.}
\end{equation*}
\end{remark}

An $r$-tunnel is a passage for which some number is $r$-admissible:

\begin{definition}\label{tunnel-def}
Let $\mathds{A}$ and $\mathds{B}$ be two {\pqpms s} and $r > 0$. An \emph{$r$-tunnel} $\tau$ from $\mathds{A}$ to $\mathds{B}$ is a passage from $\mathds{A}$ to $\mathds{B}$ such that the set $\Adm(\tau|r)$ is not empty.
\end{definition}

\begin{remark}
If $\tau$ is an $r$-tunnel, then $\tau$ is a $t$-tunnel for all $t\in (0,r]$ and $\tau^{-1}$ is an $r$-tunnel as well by Remark (\ref{inclusion-adm-rmk}).
\end{remark}

The procedure followed in \cite{Latremoliere13b} to define tunnels started with a notion of tunnel, to which a number called the length was then associated. Later \cite{Latremoliere14}, we modified slightly this number and defined the extent of a tunnel. Computing the extent is the very reason to introduce tunnels.

In our current approach, we actually defined the value of interest, essentially, as a mean to define tunnels --- an interesting reversal in our method. We are now ready to define the extent of a tunnel properly:

\begin{definition}\label{extent-def}
Let $\mathds{A}$ and $\mathds{B}$ be two {\pqpms s}, let $r > 0$, and let $\tau$ be an \emph{$r$-tunnel} from $\mathds{A}$ to $\mathds{B}$. The \emph{$r$-extent} $\tunnelextent{\tau}{r}$ is the nonnegative real number:
\begin{equation*}
\tunnelextent{\tau}{r} = \inf \Adm(\tau|r) \text{.}
\end{equation*}
When $\tau$ is an $r$-tunnel, we call $r$ a \emph{radius of admissibility}.
\end{definition}

\begin{remark}\label{extent-monotone-rmk}
For any $r$-tunnel $\tau$, we have $\tunnelextent{\tau^{-1}}{r} = \tunnelextent{\tau}{r}$, and if $t \in (0,r]$ then:
\begin{equation*}
\tunnelextent{\tau}{t} \leq \tunnelextent{\tau}{r}\text{,}
\end{equation*}
by Remark (\ref{inclusion-adm-rmk}).
\end{remark}

We believe that it is important to pause and connect our notion of $r$-tunnel with the notion of tunnels in \cite{Latremoliere13b,Latremoliere14}. Indeed, our notion of admissibility may appear surprising, if not suspicious, since it does not involve explicitly the notion of isometries. The following proposition demonstrates that in fact, $r$-tunnels are just tunnels when taken between {\Lqcms s}, and thus our notion of admissibility may possess enough strength for our purpose. The eventual proof of this lies with Theorem (\ref{coincidence-thm}).

\begin{proposition}\label{compact-tunnel-prop}
If $\mathds{A}_1 = (\A_1,\Lip_1,\M_1,\mu_1)$ and $\mathds{A}_2 = (\A_2,\Lip_2,\M_2,\mu_2)$ are two {\pqpms s}, if $r > 0$, and if:
\begin{equation*}
\tau = (\D,\Lip_\D,\M_\D,\pi_1,\mathds{A}_1,\pi_2,\mathds{A}_2)
\end{equation*}
 is an \emph{$r$-tunnel} from $\mathds{A}_1$ to $\mathds{A}_2$ such that, for some $j\in\{1,2\}$, we have $r \geq \diam{\M_j^\sigma}{\sigmaKantorovich{\Lip_j}}$, then:
\begin{enumerate}
\item $(\A_1,\Lip_1)$, $(\A_2,\Lip_2)$ and $(\D,\Lip_\D)$ are {\Lqcms s},
\item if $(\varepsilon,K)$ is $r$-admissible, then $K = \M_\D^\sigma$,
\item for all $a\in\sa{\A_j}$, we have:
\begin{equation*}
\Lip_j(a) = \inf\Set{\Lip_\D(d)}{\pi_j(d) = a}\text{,}
\end{equation*}
\item $\pi_j$ is a *-epimorphism.
\end{enumerate}
\end{proposition}

\begin{proof}
To ease notation, assume that $r \geq \diam{\M_1^\sigma}{\sigmaKantorovich{\Lip_1}}$. Since $(\M_1^\sigma,\sigmaKantorovich{\Lip_1})$ is proper, we conclude that $\M_1^\sigma$ is compact, so $\M_1$ and hence $\A_1$ is unital; thus $(\A_1,\Lip_1)$ is a compact quantum metric space. Let $\unit_1 \in \A_1$ be the unit of $\A_1$: we note that $\Lip_1(\unit_1) = 0$.

Let $(\varepsilon,K)$ be $r$-admissible for $\tau$. By Definition (\ref{left-admissible-def}), we conclude that there exists $d\in\Loc{\D}{\M_\D}{K}$ such that $\pi_\A(d) = \unit_1$ and $\Lip_\D(d) = 0$. Since $(\D,\Lip_\D)$ is a Lipschitz pair, $d \in \R\unit_\D$; since $\pi_\A(d) = \unit_1$ we conclude $d = \unit_\D$. Consequently, $K = \M_\D^\sigma$ since, if $p$ is the indicator function of $K$ then $p\unit_\D p = \unit_\D$ as $\unit_\D \in \Loc{\D}{\M_\D}{K}$, so $p = \unit_\D$.

In particular, $(\D,\Lip_\D)$ is therefore a Leibniz Lipschitz pair. Moreover, $\Lip_\D$ is lower semi-continuous by assumption on passages. Last, by definition of {\lcqms}, we conclude that $(\D,\Lip_\D)$ is a {\Lqcms}. 

Now, as $\pi_\B$ is topographic, and $\D$ is unital, we conclude that $\B$ is unital as well. Thus $(\B,\Lip_\B)$ is a {\Lqcms}. This proves (1).

Now, for all $a\in\sa{\A_1}$ with $\Lip_1(a)<\infty$,  i.e. for all $a\in\Loc{\A}{\M_1}{\mu_1,r}$, there exists $d\in \sa{\D}$ such that $\pi_1(d) = a$ and $\Lip_\A(a)\leq\Lip_\D(d)\leq\Lip_\A(a)$ by Definition (\ref{left-admissible-def}). Moreover $\Lip_1(\pi_1(d))\leq\Lip_\D(d)$ for all $d\in\sa{\D}=\lsa{\D}{\M_\D}{K}$. Thus (2) holds.

Last, (2) implies that the dense subset $\{a\in\sa{\A}:\Lip_\A(a)<\infty\}$ of $\sa{\A}$ is contained in the range of $\pi_\D$ restricted to $\sa{\D}$. Since the range of $\pi_\D$ is closed as $\pi_\D$ is *-morphism, we conclude that the range of $\pi_\D$ contains $\sa{\A}$, and thus $\pi_\D$ is surjective.
\end{proof}

In the following, we identify a {\Lqcms} $(\A,\Lip)$ with the {\pqpms} $(\A,\Lip,\C\unit_\A,\epsilon)$ where $\epsilon : t\unit_\A \in \C\unit_\A\mapsto t$.

We refer to \cite{Latremoliere13b,Latremoliere14} for the definition of tunnels between {\Lqcms s}. Our new notion of a tunnel is a natural generalization for tunnels in the compact setting, and moreover our notion of extent is compatible with the extent of a tunnel between {\Lqcms s} introduced in \cite{Latremoliere14}:

\begin{corollary}\label{compact-tunnel-corollary}
If $\mathds{A} = (\A,\Lip_\A)$ and $\mathds{B} = (\B,\Lip_\B)$ are {\Lqcms s} and $\tau=(\D,\Lip,\M,\pi_\A,\mathds{A},\pi_\B,\mathds{B})$ is a passage from $\mathds{A}$ to $\mathds{B}$, then the following are equivalent:
\begin{enumerate}
\item $\tau$ is an $r$-tunnel for some $r > 0$,
\item $(\D,\Lip,\pi_\A,\pi_\B)$ is a tunnel with finite length (or finite extent),
\item $\tau$ is an $r$-tunnel for all $r > 0$.
\end{enumerate}
Moreover, if any of (1), (2) or (3) holds, the the extent of $(\D,\Lip,\M,\pi_\A,\pi_\B)$ equals to the $r$-extent of $\tau$ for all $r > 0$.
\end{corollary}

\begin{proof}
If $\tau = (\D,\Lip,\M,\pi_\A,\mathds{A},\pi_\B,\mathds{B})$ in an $r$-tunnel for some $r > 0$, then Proposition (\ref{compact-tunnel-prop}) proves that $(\D,\Lip,\M,\pi_\A,\pi_\B)$ is a tunnel. Moreover, let $(\varepsilon,K)$ be some $r$-admissible pair for $\tau$. By Proposition (\ref{compact-tunnel-prop}), we have $K = \M_\D^\sigma$. From this, we conclude easily that $\tau$ is $t$-admissible for any $t > 0$, and moreover:
\begin{equation*}
\StateSpace(\D) = \StateSpace[\D|K] \subseteq_\varepsilon \pi_\A^\ast(\StateSpace(\A))
\end{equation*}
and, since $\pi_\A^\ast(\StateSpace(\A)) \subseteq \StateSpace(\D)$, we conclude:
\begin{equation*}
\Haus{\Kantorovich{\Lip_\D}}(\StateSpace(\D),\pi_\A^\ast(\StateSpace(\A))) \leq \varepsilon\text{.}
\end{equation*}
The same holds for $\B$ in lieu of $\A$, so $\chi(\tau)\leq \tunnelextent{\tau}{r}$. Conversely, we note that $(\chi(\tau),M_\D^\sigma)$ is $t$-admissible for any $t > 0$ so $\tunnelextent{\tau}{t} \leq \chi(\tau)$, as required.

If, on the other hand, $(\D,\Lip,\M,\pi_\A,\pi_\B)$ is a tunnel, then by Definition (\ref{admissible-def}), it is straightforward that for all $\varepsilon \geq \chi(\tau)$, the pair $(\varepsilon,\{\epsilon_\D\})$ where $\epsilon_D : t\unit_\D \mapsto t$, is $r$-admissible for any $r > 0$ (note that $\lsa{\D}{\M_\D}{\{\epsilon_\D\},r} = \sa{\D}$ for any $r > 0$).

The last implication is trivial.
\end{proof}

Following on the work in \cite{Latremoliere13b}, we shall define an inframetric (unlike \cite{Latremoliere13b} were we got a metric) from tunnels and their extents. The reason why we obtain an inframetrics seem to stem from the fact that we may compose tunnels, in a manner inspired by \cite{Latremoliere14}, yet the composition adds a constraint on the radius of admissibility. We now turn to this issue.

\subsection{Tunnel Composition}

We can form new tunnels from two tunnels with the codomain of one being equal to the domain of the other, in such a manner that the extent of the resulting tunnel is controlled by the extent of the two tunnels we merge, under a natural condition on the admissibility radii of the involved tunnels. This is the topic of the following theorem, from which the relaxed triangle inequality for our eventual new inframetric derives.

\begin{theorem}\label{tunnel-composition-thm}
Let:
\begin{equation*}
\mathds{A}=(\A,\Lip_\A,\M_\A,\mu_\A)\text{,} \mathds{B} = (\B,\Lip_\B,\M_\B,\mu_\B)\text{ and }\mathds{E} = (\alg{E},\Lip_{\alg{E}},\M_{\alg{E}},\mu_{\alg{E}})
\end{equation*}
be three {\pqpms s}. Let $r > 0$.

Let $\tau_1$ be an $r$-tunnel from $\mathds{A}$ to $\mathds{B}$ and $\tau_2$ be an $r$-tunnel from $\mathds{B}$ to $\mathds{E}$. If $t > 0$ is such that:
\begin{equation*}
t + 4\max\{ \tunnelextent{\tau_1}{r}, \tunnelextent{\tau_2}{r} \} < r
\end{equation*}
then for all $\alpha > 0$ there exists a $t$-tunnel $\tau$ from $\mathds{A}$ to $\mathds{E}$ such that:
\begin{equation*}
\tunnelextent{\tau}{t} \leq \alpha + \tunnelextent{\tau_1}{r} + \tunnelextent{\tau_2}{r} \text{.}
\end{equation*}
\end{theorem}

\begin{proof}
Let $t > 0$ such that:
\begin{equation*}
t + 4\max\{ \tunnelextent{\tau_1}{r}, \tunnelextent{\tau_2}{r} \} < r\text{.}
\end{equation*}

Let $\alpha > 0$. For all $(d_1,d_2) \in \sa{\D_1\oplus\D_2}$, we set:
\begin{equation*}
\mathsf{N}(d_1,d_2) = \frac{1}{\alpha}\norm{\pi_2(d_1)-\rho_1(d_2)}{\B}\text{,}
\end{equation*}
and
\begin{equation*}
\Lip(d_1,d_2) = \max\{\Lip_1(d_1),\Lip_2(d_2),\mathsf{N}(d_1,d_2) \}\text{.}
\end{equation*}
We note that $\Lip$ is densely defined and lower semi-continuous on $\sa{\D_1\oplus\D_2}$.

For both $j\in\{1,2\}$, let us denote by $\iota_j : \D_1\oplus\D_2 \twoheadrightarrow \D_j$ the canonical surjection. Our purpose is to prove that:
\begin{equation*}
\tau = (\D_1\oplus\D_2,\Lip,\M_1\oplus\M_2,\pi_1\circ\iota_1,\mathds{A},\rho_2\circ\iota_2,\mathds{E})
\end{equation*}
is a $t$-tunnel.

For $j\in \{1,2\}$, let $\varepsilon_j > 0$ be an $r$-admissible number for $\tau_j$ such that:
\begin{equation*}
t + 4\max\{\varepsilon_1,\varepsilon_2\} \leq r\text{.}
\end{equation*}
Note that such choices of $r$-admissible numbers are possible by assumption on $t$. For $j\in\{1,2\}$, let $(K_x^j)_{x \in (0,r]}$ be an increasing family of compact subsets of $\M_j^\sigma$ such that $(\varepsilon,K_x^j)$ is $x$-admissible for $\tau_j$ and for all $x\in(0,r]$.

Let $x\in (0,t]$. 

Since $\Lip\geq\max\{\Lip_1,\Lip_2\}$, we conclude that for all $j\in\{1,2\}$ and for all $\varphi,\psi \in \StateSpace(\D_j)$:
\begin{equation*}
\Kantorovich{\Lip}(\varphi\circ\iota_j,\psi\circ\iota_j)\leq\Kantorovich{\Lip_j}(\varphi,\psi)\text{.}
\end{equation*}

Let $\varphi\in \StateSpace[\D_1|K_x^1]$. There exists $\psi\in \StateSpace[\B|\mu_\B,x+4\varepsilon_1]$ such that:
\begin{equation*}
\Kantorovich{\Lip_1}(\varphi,\psi\circ\pi_2)\leq\varepsilon_1\text{.}
\end{equation*}

Moreover, by assumption on $t$, we have $x + 4\varepsilon_1 \leq r$. Hence, by Definitions (\ref{left-admissible-def}) and (\ref{admissible-def}), we have:
\begin{equation*}
\rho_1^\ast\left(\StateSpace[\B|\mu_\B,x+4\varepsilon_1]\right) \subseteq \StateSpace[\D_2|K_{x+4\varepsilon_1}^2]\subseteq_{\varepsilon_2} \rho_2^\ast \left(\StateSpace[\alg{E}|\mu_{\alg{E}}, (x+4\varepsilon_1) + 4\varepsilon_2]\right)\text{.}
\end{equation*}

Therefore, there exists:
\begin{equation*}
\theta\in\StateSpace[\alg{E}|x + 4\varepsilon_1 + 4\varepsilon_2]
\end{equation*}
such that $\Kantorovich{\Lip_2}(\psi\circ\rho_1,\theta\circ\rho_2)\leq\varepsilon_2$.

Now, if $\Lip(d_1,d_2)\leq 1$ for some $(d_1,d_2)\in\sa{\D}$, then $\|\pi_2(d_1)-\rho_1(d_2)\|_\B \leq \alpha$, and therefore $\Kantorovich{\Lip}(\psi\circ\pi_2\circ\iota_1,\psi\circ\rho_1\circ\iota_2)\leq \alpha$.

Then:
\begin{equation*}
\begin{split}
\Kantorovich{\Lip}(\varphi\circ\iota_1,\theta\circ\rho_2\circ\iota_2) &\leq \Kantorovich{\Lip}(\varphi\circ\iota_1,\psi\circ\pi_2\circ\iota_1) + \Kantorovich{\Lip}(\psi\circ\pi_2\circ\iota_1,\psi\circ\rho_1\circ\iota_2)\\
&\quad + \Kantorovich{\Lip}(\psi\circ\rho_1\circ\iota_2,\theta\circ\rho_2\circ\iota_2)\\
&\leq \Kantorovich{\Lip_1}(\varphi,\psi\circ\pi_2) + \alpha + \Kantorovich{\Lip_2}(\psi\circ\rho_1,\theta\circ\rho_2)\\
&\leq \varepsilon_1 + \alpha + \varepsilon_2\text{.}
\end{split}
\end{equation*}

Now, if $\varphi \in \StateSpace[\D_1\oplus\D_2|K_x^1\coprod K_x^2]$, then there exists $y\in[0,1]$, $\varphi_1\in\StateSpace[\D_1|K_x^1]$ and $\varphi_2\in\StateSpace[\D_2|K_x^2]$ such that $\varphi = y\varphi_1\circ\iota_1 + (1-y)\varphi_2\circ\iota_2$. Thus there exists $\theta_1, \theta_2 \in\StateSpace[\alg{E}|t + 4\varepsilon_1 + 4\varepsilon_2]$ such that:
\begin{equation*}
\Kantorovich{\Lip}(\varphi_1,\theta_1\circ\rho_2) \leq \varepsilon_1 + \varepsilon_2 + \alpha
\end{equation*}
and
\begin{equation*}
\Kantorovich{\Lip_2}(\varphi_2,\theta_1\circ\rho_2) \leq \varepsilon_2\text{.}
\end{equation*}

Therefore, since $\Kantorovich{\Lip}$ is a convex function over $\StateSpace(\D_1\oplus\D_2)\times\StateSpace(\D_1\oplus\D_2)$:
\begin{equation*}
\begin{split}
\Kantorovich{\Lip}&(\varphi,y\theta_1\circ\rho_2\circ\iota_2+(1-y)\theta_2\circ\rho_2\circ\iota_2) \\
&\leq y\Kantorovich{\Lip}(\varphi_1,\theta_1\circ\rho_2\circ\iota_2) + (1-y)\Kantorovich{\Lip}(\varphi_2,\theta_2\circ\rho_2\circ\iota_2)\\
&\leq y\Kantorovich{\Lip_1}(\varphi_1,\theta_1\circ\rho_2) + (1-y)\Kantorovich{\Lip_2}(\varphi_2,\theta_2\circ\rho_2)\\
&\leq y(\varepsilon_1+\varepsilon_2+\alpha) + (1-y)\varepsilon_2\\
&\leq \varepsilon_1+\varepsilon_2+\alpha\text{.}
\end{split}
\end{equation*}

Thus:
\begin{equation*}
\StateSpace[\D_1\oplus\D_2|K_x^1\coprod K_x^2] \almostsubseteq{(\StateSpace(\D_1\oplus\D_2),\Kantorovich{\Lip})}{\varepsilon_1+\varepsilon_2+\alpha}\StateSpace[\alg{E}|t+4\varepsilon_1+4\varepsilon_2]\text{.}
\end{equation*}

Let $a\in\Loc{\alg{E}}{\M_{\alg{E}}}{\mu_{\alg{E}},x}$. There exists $d_2\in\Loc{\D_2}{\M_2}{K_x^2}$ such that $\rho_2(d_2)=a$, $\Lip_2(d_2) = \Lip_{\alg{E}}(a)$ and:
\begin{equation*}
b = \rho_1(d_1) \in\Loc{\B}{\M_\B}{\mu_\B,x+4\varepsilon_2} \subseteq \Loc{\B}{\M_\B}{\mu_\B,r}\text{.}
\end{equation*}
Hence there exists $d_1\in\Loc{\D_1}{\M_1}{K_x^1}$ such that $\Lip_1(d_1) = \Lip_\B(b)$, $\pi_2(d_1) = b$ and $\pi_1(d_1)\in\Loc{\A}{\M_\A}{\mu_\A,x+4\varepsilon_1+4\varepsilon_2}$.

Consequently, $\Lip(d_1,d_2) = \Lip_{\alg{E}}(a)$ while $(d_1,d_2)\in\Loc{\D_1\oplus\D_2}{\M_1\oplus\M_2}{K_x^1\coprod K_x^2}$ and $\pi_1\circ\iota_1(d_1,d_2)\in\Loc{\A}{\M_{\A}}{x+4\varepsilon_1+4\varepsilon_2}$ while $\rho_2\circ\iota_2(d_1,d_2)=a$.

This proves that $(\varepsilon_1 + \varepsilon_2 + \alpha, K_x^1\coprod K_x^2)$ is $x$-right admissible for $\tau$. The proof that $(\varepsilon_1 + \varepsilon_2 + \alpha, K_x^1\coprod K_x^2)$ is $x$-left admissible is done symmetrically.  Now, we note that $(K_x^1\cup K_x^2)_{x\in(0,t]}$ is an increasing family of compacts subsets of $(\M_1\oplus\M_2)^\sigma$.

Last, $\mathrm{N}$ satisfies, for all $d_1,d_1' \in \D_1$, $d_2,d_2' \in \D_2$:
\begin{equation*}
\begin{split}
\mathrm{N}(d_1d_1',d_2d_2') &= \frac{1}{D}\norm{\pi_2(d_1d_1') - \rho_1(d_2d_2')}{\B}\\
&\leq\frac{1}{D}\norm{\pi_2(d_1)\pi_2(d_1')-\pi_2(d_1)\rho_1(d_2')}{\B} \\
&\quad + \frac{1}{D}\norm{\pi_2(d_1)\rho_1(d_2')-\rho_1(d_2)\rho_1(d_2')}{\B}\\
&\leq \norm{d_1}{\D_1}\mathrm{N}(d_1',d_2') + \norm{d_2'}{\D_2}\mathrm{N}(d_1,d_2)\\
&\leq \norm{(d_1,d_2)}{\D_1\oplus\D_2}\mathrm{N}(d_1',d_2') + \norm{(d_1',d_2')}{\D_1\oplus\D_2}\mathrm{N}(d_1,d_2) \text{.}
\end{split}
\end{equation*}
It then follows that $\Lip$ satisfies Conditions (\ref{admissible-Jordan-eq}) and (\ref{admissible-Lie-eq}) in Definition (\ref{admissible-def}). This concludes our proof as $x\in (0,t]$ is arbitrary.
\end{proof}

We postponed to this point the proof of the existence of tunnels between appropriate {\pqpms s}, since the following proof relies on some of the techniques involved in Theorem (\ref{tunnel-composition-thm}).

\begin{proposition}\label{existence-prop}
Let $(\A_1,\Lip_1,\M_1)$ and $(\A_2,\Lip_2,\M_2)$ be two {\pqpms s}. If $r > 0$, and if either one of the following condition holds:
\begin{enumerate}
\item $r \geq \max\left\{\diam{\M_1^\sigma}{\sigmaKantorovich{\Lip_1}},\diam{\M_1^\sigma}{\sigmaKantorovich{\Lip_1}} \right\}$,
\item $r < \min\left\{\diam{\M_1^\sigma}{\sigmaKantorovich{\Lip_1}},\diam{\M_1^\sigma}{\sigmaKantorovich{\Lip_1}} \right\}$,
\end{enumerate}
then there exists a $r$-tunnel from $\mathds{A}_1$ to $\mathds{A}_2$.
\end{proposition}

\begin{proof}
Condition (1) implies, by Proposition (\ref{compact-tunnel-prop}), that $\mathds{A}_1$ and $\mathds{A}_2$ are compact. By \cite[Proposition 4.6]{Latremoliere13}, there exists a tunnel $(\D,\Lip_\D,\pi_1,\pi_2)$ of finite length from $(\A_1,\Lip_1)$ to $(\A_2,\Lip_2)$. By \cite{Latremoliere14}, the extent of $\tau$ is finite as well. By Corollary (\ref{compact-tunnel-corollary}), the passage $(\D,\Lip_\D,\pi_1,\pi_2)$ is an $r$-tunnel for any $r>0$.

Assume now Condition (2). By Proposition (\ref{compact-liploc-prop}), the set $\LipLoc{\A_j}{\M_j}{\mu_j,r}$ is compact, hence norm bounded, for all $j\in\{1,2\}$. Let:
\begin{equation*}
D = \max\left\{ \|a\|_{\A_j} : j \in \{1,2\}, a\in\Loc{\A_j}{\M_j}{\mu_j,r} \right\} < \infty \text{.}
\end{equation*}

There exist two representations $\rho_1$ and $\rho_2$ of, respectively, $\A_1$ and $\A_2$ on some Hilbert space $\Hilbert$. We define, for all $(a_1,a_2)\in\sa{\A_1\oplus\A_2}$:
\begin{equation*}
\Lip(a_1,a_2) = \max\left\{\Lip_1(a_1),\Lip_2(a_2),\frac{1}{D}\vertiii{\rho_1(a_1)-\rho_2(a_2)}\right\}
\end{equation*}
where $\vertiii{\cdot}$ is the operator norm for bounded linear operators on $\Hilbert$. We denote the canonical surjections from $\A_1\oplus\A_2$ onto $\A_1$ and $\A_2$ by $\pi_1$ and $\pi_2$ respectively. We propose to prove that the passage:
\begin{equation*}
\tau = (\A_1\oplus\A_2,\Lip,\M_1\oplus\M_2,\pi_1,\mathds{A}_1,\pi_2,\mathds{A}_2)
\end{equation*}
is an $r$-tunnel.

Let $t\in (0,r]$. Now, let $a\in\Loc{\A_1}{\M_\A}{\mu_\A,t}$. Then $\Lip(a,0) = \Lip_1(a)$ and of course $\pi_1(a,0) = a$. Moreover, $\pi_2(a,0) = 0 \in \Loc{\B}{\M_\B}{\mu_\B,t}$.

Let $K = \StateSpace[\A_1|\mu_1,t]\coprod\StateSpace[\A_2|\mu_2,t]$. If $\varphi\in\StateSpace[\A_2|\mu_2,t]$, then:
\begin{equation*}
\Kantorovich{\Lip}(\varphi,\mu_\A)\leq D\text{.}
\end{equation*}
Consequently, we have: $\StateSpace[\A_1\oplus\A_2|K] \subseteq_D \StateSpace[\A_1|\mu_1,t]$. In conclusion, $(D,K)$ is $t$-left admissible for $\tau$.

Our reasoning is symmetric in $\A_1$ and $\A_2$, and the Leibniz property for $\Lip$ is easily verified, so we conclude that $\tau$ is a $r$-tunnel of extent at most $D$. 
\end{proof}

We conclude this section with an important remark.
\begin{remark}\label{inversion-rmk}
Let:
\begin{equation*}
\tau = (\D,\Lip_\D,\M_\D,\pi_\A,\mathds{A},\pi_\B,\mathds{B})
\end{equation*}
be a $r$-tunnel from $\mathds{A} = (\A,\Lip_\A,\M_\A,\mu_\A)$ to $\mathds{B} = (\B,\Lip_\B,\M_\B,\mu_\B)$. Let $a\in\Loc{\A}{\M_\A}{\mu_\A,r}$ with $\Lip_\A(a)<\infty$ and let $(\varepsilon,K)$ be an $r$-admissible pair for $\tau$. Let $l\geq\Lip_\A(a)$. By Definition (\ref{left-admissible-def}), there exists $d\in\Loc{\D}{\M_\D}{K}$ such that $\pi_\A(d) = a$ and $\Lip_\D(d)\leq l$. Moreover $b = \pi_\B(d) \in \Loc{\B}{\M_\B}{\mu_\B,r+4\varepsilon}$ and $\Lip_\B(b)\leq l$.

Now, we observe that we can read the above assertions about $d$ in a symmetric manner, namely: for the given $b$ above, there exists $d\in\Loc{\D}{\M_\D}{K}$ with $\Lip_\D(d)\leq l$; thus $\pi_\A(a) \in\targetsettunnel{\tau^{-1}}{b}{l,r,\varepsilon,K}$ --- even though we have no reason to expect $\tau^{-1}$ to be $(r+4\varepsilon)$-admissible, of course. However, there is enough symmetry in our Definition (\ref{admissible-def}) to ensure that the following holds:
\begin{equation*}
b\in\targetsettunnel{\tau}{a}{l,r,\varepsilon,K} \iff a\in\targetsettunnel{\tau^{-1}}{b}{l,r+4\varepsilon,\varepsilon,K}
\end{equation*}
for all $a\in\Loc{\A}{\M_\A}{\mu_\A,r}$.

In this sense, we see that a composition of $\tau$ and $\tau^{-1}$ may be a good approximation for the identity map of $\A$, for small $r$-extents. This point will play an important role in the proof of Theorem (\ref{coincidence-thm}).
\end{remark}

\subsection{Target Sets of Tunnels}

In this subsection, we prove Theorem (\ref{fundamental-thm}), the fundamental result on which Theorem (\ref{coincidence-thm}) relies to prove that our hypertopology, defined in the next section, will possess the desired separation property. The main point of Theorem (\ref{fundamental-thm}) is that target sets posses a form of morphism property; tunnels can thus be seen as a generalized form of morphisms, and their extent can be interpreted as a measure of their distortion or Lipschitz number.

We begin with:

\begin{lemma}\label{norm-lemma}
Let $\A$ be a C*-algebra, $p\in\A^{\ast\ast}$ be a projection, and $a\in \sa{p\A p \cap \A}$. Then:
\begin{equation*}
\|a\|_\A = \sup\Set{|\varphi(a)|}{\varphi\in\StateSpace(\A) \text{ and }\varphi(p) = 1}\text{.}
\end{equation*}
\end{lemma}

\begin{proof}
Since $a\in\sa{\A}$, we have:
\begin{multline}\label{norm-lemma-eq0}
\sup\Set{|\varphi(a)|}{\varphi\in\StateSpace(\A),\varphi(p)=1}\leq\|a\|_\A\\
=\sup\Set{|\varphi(a)|}{\varphi\in\StateSpace(\A)}\text{.}
\end{multline}

Let $\varepsilon > 0$ and let $\varphi\in\StateSpace(\A)$ such that $\|a\|_\A\leq|\varphi(a)|+\varepsilon$, which exists by Inequality (\ref{norm-lemma-eq0}).

Set $\psi:b\in\A\mapsto \frac{1}{\varphi(p)}\varphi(p b p)$. On the one hand, $\psi$ is a positive linear map, since $b\geq 0\implies cbc^\ast \geq 0$ for all $b,c \in \A^{\ast\ast}$. Moreover, $\psi(\unit_\A) = \psi(p) = 1$ by construction as well, so $\psi\in\StateSpace(\A)$.

On the other hand, since $a\in \A\cap p\A p$, there exists $b\in \A$ such that $a = p b p$; then $a = p b p = p p b p p = p a p$. Thus:
\begin{equation*}
|\varphi(a)| = |\varphi(p a p)| = \varphi(p)|\psi(a)| \leq |\psi(a)|\text{,}
\end{equation*}
since $\varphi$ is a state and $0\leq p\leq 1$. Thus $\|a\|_\A\leq\varepsilon + \psi(a)$ with $\psi\in\StateSpace(\A)$, $\psi(p)=1$. This concludes our lemma since $\varepsilon > 0$ is arbitrary.
\end{proof}

We now prove the key theorem for our construction:

\begin{theorem}\label{fundamental-thm}
Let $\mathds{A}=(\A,\Lip_\A,\M_\A,\mu_\A)$ and $\mathds{B}=(\B,\Lip_\B,\M_\B,\mu_\B)$ be two pointed {\pqms s}. Let:
\begin{equation*}
\tau = (\D,\Lip_\D,\M_\D,\pi_\A,\mathds{A},\pi_\B,\mathds{B})
\end{equation*}
be an $r$-tunnel from $\mathds{A}$ to $\mathds{B}$ for some $r > 0$. Let $l \geq 0$ and let $(\varepsilon,K)$ be an $r$-admissible pair for $\tau$.

\begin{enumerate}

\item If $a\in\Loc{\A}{\M_\A}{\mu_\A,r}$ with $\Lip_\A(a)\leq l$, and if $d\in\liftsettunnel{\tau}{a}{l,r,\varepsilon,K}$, then:
\begin{equation}\label{lift-norm-bound-eq}
\|d\|_\D \leq \|a\|_\A + l \varepsilon\text{,}
\end{equation}
and therefore for all $b\in\targetsettunnel{\tau}{a}{l,r,\varepsilon,K}$:
\begin{equation}\label{target-norm-bound-eq}
\|b\|_\D \leq \|a\|_\A + l \varepsilon\text{.}
\end{equation}

\item If $a,a' \in \Loc{\A}{\M_\A}{\mu_\A,r}$ with $\Lip_\A(a) \leq l$, $\Lip_\A(a')\leq l$, then for any $b\in \targetsettunnel{\tau}{a}{l,r,\varepsilon,K}$ and $b'\in \targetsettunnel{\tau}{a}{l,r,\varepsilon,K}$, and for any $t \in \R$, we have:
\begin{equation}\label{linearity-eq}
b+tb' \in \targetsettunnel{\tau}{a+ta'}{(1+|t|)l,r,\varepsilon,K}\text{.}
\end{equation}

\item If $a,a' \in \Loc{\A}{\M_\A}{\mu_\A,r}$ with $\Lip_\A(a) \leq l$, $\Lip_\A(a')\leq l$, then:
\begin{multline}\label{distance-eq}
\sup\Set{\|b-b'\|_\B}{b\in\targetsettunnel{\tau}{a}{r,l,\varepsilon,K},b'\in\targetsettunnel{\tau}{a'}{r,l,\varepsilon,K}}\\\leq \|a-a'\|_\A + 2l \varepsilon\text{.}
\end{multline}
In particular, we have:
\begin{equation}\label{diameter-eq}
\diam{\targetsettunnel{\tau}{a}{l,r,\varepsilon,K}}{\|\cdot\|_\B}\leq 2l\varepsilon\text{.}
\end{equation}

\item if $a,a' \in \Loc{\A}{\M_\A}{\mu_\A,r}$ with $\Lip_\A(a) \leq l$, $\Lip_\A(a')\leq l$, then:
\begin{equation}\label{Jordan-eq}
\Jordan{b}{b'} \in \targetsettunnel{\tau}{\Jordan{a}{a'}}{l(\|a\|_\A+\|a'\|_\A+2l\varepsilon),r,\varepsilon,K}
\end{equation}
and
\begin{equation}\label{Lie-eq}
\Lie{b}{b'} \in \targetsettunnel{\tau}{\Lie{a}{a'}}{l(\|a\|_\A+\|a'\|_\A+2l\varepsilon),r,\varepsilon,K}\text{.}
\end{equation}
\end{enumerate}

\end{theorem}

\begin{proof}
Let $d\in \liftsettunnel{\tau}{a}{l,r,\varepsilon,K}$. Let $\varphi \in \StateSpace[\D|K]$. By Definition (\ref{left-admissible-def}), there exists $\psi \in \StateSpace[\A|\mu_\A,r+4\varepsilon]$ such that:
\begin{equation*}
\Kantorovich{\Lip_\D}(\varphi,\psi\circ\pi_\A) \leq \varepsilon\text{.}
\end{equation*}
Thus:
\begin{equation*}
\begin{split}
|\varphi(d)| &\leq |\varphi(d) - \psi\circ\pi_\A(d)| + |\psi\circ\pi_\A(d)|\\
&\leq l \varepsilon + \|a\|_\A\text{.}
\end{split}
\end{equation*}
Consequently, by Lemma (\ref{norm-lemma}), since $d\in\Loc{\D}{\M_\D}{K}$, we conclude:
\begin{equation*}
\norm{d}{\D} \leq l\varepsilon + \|a\|_\A\text{.}
\end{equation*}

This proves Inequality (\ref{lift-norm-bound-eq}).

Since $\pi_\B$ has norm $1$, we conclude that Inequality (\ref{target-norm-bound-eq}) holds as well.

For the remainder of this proof, we fix $b\in\targetsettunnel{\tau}{a}{l,r,\varepsilon,K}$ and $b'\in\targetsettunnel{\tau}{a'}{l,r,\varepsilon,K}$ and we let $d\in\liftsettunnel{\tau}{a}{l,r,\varepsilon,K}$ and $d'\in\liftsettunnel{\tau}{a}{l,r,\varepsilon,K}$ such that $\pi_\A(d) = b$ and $\pi_\A(d')=b'$. 

For all $t\in \R$, we then have:
\begin{enumerate}
\item $\Lip_\D(d+td') \leq \Lip_\D(d) + |t|\Lip_\D(d') \leq (1+|t|)l$,
\item $\pi_\A(d+td') = a + ta'$,
\item $d+td' \in \Loc{\D}{\M_\D}{K}$,
\item $\pi_\B(d+tb') \in \Loc{\B}{\M_\D}{\mu_\B,r+4\varepsilon}$.
\end{enumerate}
Hence $d+td' \in \liftsettunnel{\tau}{a}{l,r,\varepsilon,K}$ and thus $b+tb' = \pi_\B(d+td') \in \targetsettunnel{\tau}{a+ta'}{l,r,\varepsilon,K}$.

We conclude from Inequality (\ref{linearity-eq}):
\begin{equation*}
b-b'\in\targetsettunnel{\tau}{a-a'}{2l,r,\varepsilon,K}\text{,}
\end{equation*}
hence Inequality (\ref{target-norm-bound-eq}) gives us:
\begin{equation*}
\|b-b'\|_\B \leq \|a-a'\|+2 l \varepsilon \text{.}
\end{equation*}
Inequality (\ref{diameter-eq}) follows from Inequality (\ref{distance-eq}) by setting $a=a'$.

Since $\Lip_\D$ satisfies Condition (\ref{admissible-Jordan-eq}), and by Inequality (\ref{lift-norm-bound-eq}), we compute:
\begin{equation*}
\begin{split}
\Lip_\D(\Jordan{d}{d'}) &\leq \Lip_\D(d)\|d'\|_\D + \Lip_\D(d')\|d\|_\D \\
&\leq l(\|a\|_\A + l\varepsilon) + l(\|a'\|_\A + l\varepsilon)\\
&= l(\|a\|_\A + \|a'\|_\A + 2l\varepsilon)\text{.}
\end{split}
\end{equation*}
The proof of Expression (\ref{Lie-eq}) is similar. This concludes our proof.
\end{proof}

\section{The Topographic Gromov-Hausdorff Quantum Hypertopology}

The topographic Gromov-Hausdorff quantum hypertopology is a topology on the class of {\pqpms s} which extend the topology of the Gromov-Hausdorff distance on classical proper metric spaces and the topology of the dual Gromov-Hausdorff propinquity for {\Lqcms s} at once. It is the main object of this paper, and a key feature of our topology is that it is Hausdorff modulo isometric isomorphism. This section begins by defining the noncommutative analogues of $\Delta_r$ of Definition (\ref{Delta-r-def}). It then constructs our analogue of the Gromov-Hausdorff distance for {\pqpms s}, which in our setting is an inframetric, i.e. satisfies a modified version of the triangle inequality. Nonetheless, we prove that we can define a natural topology from our inframetric. The last part of this section is to prove the separation axiom for our topology, and constitute the main theorem of this paper.

\subsection{The Local Propinquities}

Our construction allows for the same form of flexibility as offered in \cite{Latremoliere13,Latremoliere13b,Latremoliere14}, where we can actually work on proper subclasses of {\pqpms s}  and put additional constraints on tunnels, in the goal to build stronger metrics which preserve more structure than the generic construction. To this end, we first isolate the properties we need a class of tunnels to satisfy for our construction to satisfy all the desired properties.

\begin{definition}\label{w-appropriate-def}
Let $\mathcal{C}$ be a nonempty class of {\pqpms s}. A class $\mathcal{T}$ of passages is \emph{weakly appropriate} for $\mathcal{C}$ when:
\begin{enumerate}
\item if $\tau \in \mathcal{T}$, then there exists $r > 0$, $\mathds{A}\in\mathcal{C}$ and $\mathds{B} \in \mathcal{C}$ such that $\tau$ is an $r$-tunnel from $\mathds{A}$ to $\mathds{B}$,
\item if $\tau\in\mathcal{T}$ then $\tau^{-1}\in\mathcal{T}$,
\item if $\tau_1,\tau_2\in\mathcal{T}$ are $r$-tunnels for some $r >0$, if $\varepsilon_1 > \tunnelextent{\tau_1}{r}$ and $\varepsilon_2 > \tunnelextent{\tau_2}{r}$, then for any:
\begin{equation*}
0 < t \leq r - 4\max\{ \varepsilon_1,\varepsilon_2 \}
\end{equation*}
and for any $\alpha > 0$, then there exists a $t$-tunnel $\tau \in \mathcal{T}$ from $\dom{\tau_1}$ to $\codom{\tau_2}$ such that:
\begin{equation*}
\tunnelextent{\tau}{t} \leq \varepsilon_1 + \varepsilon_2 + \alpha\text{,}
\end{equation*}
\item if $h : \mathds{A} \rightarrow \mathds{B}$ is some isometric isomorphism from $\mathds{A} \in \mathcal{C}$ to $\mathds{B} \in \mathcal{C}$, then, writing $\mathds{A} = (\A,\Lip_\A,\M_\A,\mu_\A)$ and $\mathds{B} = (\B,\Lip_\B,\M_\B,\mu_\B)$, both of the following passages lie in $\mathcal{T}$:
\begin{equation*}
\left(\A,\Lip_\A,\M_\A,\mathrm{id}_\A,\mathds{A},h,\mathds{B}\right)
\end{equation*}
and
\begin{equation*}
\left(\B,\Lip_\B,\M_\B,\mathrm{id}_\B,\mathds{B},h^{-1},\mathds{A}\right)\text{,}
\end{equation*}
where we used the notation $\mathrm{id}_\alg{E}$ for the identity automorphism of any C*-algebra $\alg{E}$.
\end{enumerate}
\end{definition}

We do not require that a class of tunnels $\mathcal{T}$ weakly appropriate for $\mathcal{C}$ be connected, i.e. that any two elements of $\mathcal{C}$ be connected via an $r$-tunnel in $\mathcal{T}$ for any $r > 0$. This justifies the qualifier weakly in contrast with \cite{Latremoliere14}. 

\begin{theorem}
Let $\mathcal{PQMS}$ be the class of all {\pqpms s}. The class $\mathcal{\PQMST}$ of all passages $\tau$ for which there exists $r > 0$ such that $\tau$ is an $r$-tunnel from its domain to its codomain, is weakly appropriate for $\mathcal{\PQMS}$.
\end{theorem}

\begin{proof}
Theorem (\ref{tunnel-composition-thm}) provides the proof that Assertion (3) of Definition (\ref{w-appropriate-def}) holds for $\mathcal{PQMST}$. The other assertions of Definition (\ref{w-appropriate-def}) are trivial.
\end{proof}

To establish that a class of passages is weakly appropriate, one may attempt to check whether the construction of composed tunnels given in the proof of Theorem (\ref{tunnel-composition-thm}) applies. Examples where this construction can be applied include, for instance:
\begin{enumerate}
\item the class of all passages over all {\pqpms s} of the form $\tau=(\D,\Lip,\M,\pi,\dom{\tau},\rho,\codom{\tau})$ where $(\D,\Lip)$ is a Leibniz Lipschitz pair,
\item the class of all passages over all {\pqpms s} of the form $\tau=(\D,\Lip,\M,\pi,\dom{\tau},\rho,\codom{\tau})$ where $(\D,\Lip,\M)$ is a {\pqms},
\item the class of all passages over all {\pqpms s} of the form $\tau=(\D,\Lip,\M,\pi,\dom{\tau},\rho,\codom{\tau})$ where $(\D,\Lip)$ is a Leibniz Lipschitz pair and the maps $\pi$ and $\rho$ are *-epimorphisms such that the quotients of $\Lip$ from $\pi$ and $\rho$ are, respectively, the Lip-norms on $\dom{\tau}$ and $\codom{\tau}$,
\item the class of all passages over all {\pqpms s} of the form $\tau=(\D,\Lip,\M,\pi,\dom{\tau},\rho,\codom{\tau})$ where $(\D,\Lip,\M)$ is a {\pqms} and the maps $\pi$ and $\rho$ are *-epimorphisms such that the quotients of $\Lip$ from $\pi$ and $\rho$ are, respectively, the Lip-norms on $\dom{\tau}$ and $\codom{\tau}$,
\end{enumerate}
to quote a few interesting special classes for which one can check that Theorem (\ref{tunnel-composition-thm}) holds. In particular, the classes described in Items (3) and (4) are weakly appropriate by using the argument of \cite[Theorem 3.1]{Latremoliere14}. Other more specialized classes, weakly appropriate to possibly proper subclasses of $\PQMS$, may play an important role in future development of this theory. The main result of this paper holds for any weakly appropriate class over any nonempty class of {\pqpms s}.

\begin{notation}
Let $\mathcal{C}$ be a nonempty class of {\pqpms s} and $\mathcal{T}$ be a weakly appropriate class of tunnels for $\mathcal{C}$. Let $r > 0$ and $\mathds{A},\mathds{B} \in \mathcal{C}$. The set of all $r$-tunnels from $\mathds{A}$ to $\mathds{B}$ in $\mathcal{T}$ is denoted by:
\begin{equation*}
\tunnelset{\mathcal{T}}{r}{\mathds{A}}{\mathds{B}}\text{.}
\end{equation*}
This set may be empty.
\end{notation}

We now define the basic ingredient in the construction of our hypertopology, in analogy with our discussion in Section 2.

\begin{definition}\label{local-propinquity-def}
Let $\mathcal{C}$ be a nonempty class of {\pqpms s} and let $\mathcal{T}$ be a weakly appropriate class of tunnels over $\mathcal{C}$. Let $\mathds{A}$ and $\mathds{B}$ be two {\pqpms s} and $r > 0$. The \emph{$r$-local $\mathcal{T}$-propinquity} between $\mathds{A}$ and $\mathds{B}$ is:
\begin{equation*}
\propinquity{r,\mathcal{T}}(\mathds{A},\mathds{B}) = \inf\set{ \tunnelextent{\tau}{r} }{\tau \in \tunnelset{\mathcal{T}}{r}{\mathds{A}}{\mathds{B}} }
\end{equation*}
with the usual convention that the infimum of the empty set is $\infty$.
\end{definition}

By construction, we observe the useful property:

\begin{lemma}\label{monotone-local-propinquity-lemma}
Let $\mathcal{C}$ be a nonempty class of {\pqpms s} and let $\mathcal{T}$ be a weakly appropriate class of tunnels over $\mathcal{C}$. Let $\mathds{A}$ and $\mathds{B}$ be two {\pqpms s}. If $r' \geq r > 0$ then:
\begin{equation*}
\propinquity{r,\mathcal{T}}(\mathds{A},\mathds{B}) \leq \propinquity{r',\mathcal{T}}(\mathds{A},\mathds{B})\text{.}
\end{equation*}
\end{lemma}
\begin{proof}
This follows immediately from Definition (\ref{local-propinquity-def}) and Remark (\ref{extent-monotone-rmk}). 
\end{proof}

The local propinquity satisfies a form of triangle inequality, owing to the composition of tunnels:

\begin{proposition}\label{triangle-inequality-local-propinquity-prop}
Let $\mathcal{C}$ be a nonempty class of {\pqpms s} and let $\mathcal{T}$ be a class of tunnels weakly appropriate for $\mathcal{C}$. Let $\mathds{A}$, $\mathds{B}$ and $\mathds{E}$ be three {\pqpms s} in $\mathcal{C}$ and $r > 0$. If:
\begin{equation*}
R > r + 4\max\{ \propinquity{r,\mathcal{T}}(\mathds{A},\mathds{B}), \propinquity{r,\mathcal{T}}(\mathds{B},\mathds{E}) \}
\end{equation*}
then:
\begin{equation*}
\propinquity{r,\mathcal{T}}(\mathds{A},\mathds{E}) \leq \propinquity{R,\mathcal{T}}(\mathds{A},\mathds{B}) + \propinquity{R,\mathcal{T}}(\mathds{B},\mathds{E})\text{.}
\end{equation*}
\end{proposition}

\begin{proof}
Let:
\begin{equation*}
R > r + 4\max\{ \propinquity{r,\mathcal{T}}(\mathds{A},\mathds{B}), \propinquity{r,\mathcal{T}}(\mathds{B},\mathds{E}) \}\text{,}
\end{equation*}
and let $\varepsilon > 0$.

We distinguish two cases. First, if there is no $R$-tunnel in $\mathcal{T}$ between $\mathds{A}$ and $\mathds{B}$ or no $R$-tunnel in $\mathcal{T}$ between $\mathds{B}$ and $\mathds{E}$, then we observe that our inequality is trivial, since the right hand-side is infinite.

Second, we assume that both:
\begin{equation*}
\tunnelset{\mathcal{T}}{R}{\mathds{A}}{\mathds{B}} \text{ and }\tunnelset{\mathcal{T}}{R}{\mathds{B}}{\mathds{E}}
\end{equation*}
are not empty. By Definition (\ref{local-propinquity-def}), there exists:
\begin{equation*}
\tau_1 \in \tunnelset{\mathcal{T}}{R}{\mathds{A}}{\mathds{B}}
\end{equation*}
such that
\begin{equation*}
\tunnelextent{\tau_1}{R} < \propinquity{R,\mathcal{T}}(\mathds{A},\mathds{B}) + \frac{\varepsilon}{3}
\end{equation*}
and there exists:
\begin{equation*}
\tau_2 \in \tunnelset{\mathcal{T}}{R}{\mathds{B}}{\mathds{E}}
\end{equation*}
such that
\begin{equation*}
\tunnelextent{\tau_2}{R} < \propinquity{R,\mathcal{T}}(\mathds{B},\mathds{E}) + \frac{\varepsilon}{3}\text{.}
\end{equation*}

By Definition (\ref{admissible-def}), there exists a $R$-admissible number $\varepsilon_1$ for $\tau_1$ and a $R$-admissible number $\varepsilon_2$ for $\tau_2$ such that:
\begin{equation*}
\varepsilon_1 \leq \propinquity{R,\mathcal{T}}(\mathds{A},\mathds{B}) + \frac{\varepsilon}{3}\text{ and }\varepsilon_2 \leq \propinquity{R,\mathcal{T}}(\mathds{B},\mathds{E}) + \frac{\varepsilon}{3}\text{.}
\end{equation*}

Since $\mathcal{T}$ is weakly appropriate, there exists:
\begin{equation*}
\tau \in \tunnelset{\mathcal{T}}{r}{\mathds{A}}{\mathds{E}}
\end{equation*}
such that $\tunnelextent{\tau}{r}\leq \varepsilon_1 + \varepsilon_2 + \frac{\varepsilon}{3}$. Hence, by Definition:

\begin{equation*}
\begin{split}
\propinquity{r,\mathcal{T}}(\mathds{A},\mathds{E}) &\leq \varepsilon_1 + \varepsilon_2 + \frac{\varepsilon}{3} \\
&\leq \propinquity{R,\mathcal{T}}(\mathds{A},\mathds{B}) + \frac{\varepsilon}{3} + \propinquity{R,\mathcal{T}}(\mathds{B},\mathds{E}) + \frac{\varepsilon}{3} + \frac{\varepsilon}{3}\\
&\leq \propinquity{R,\mathcal{T}}(\mathds{A},\mathds{B}) +  \propinquity{R,\mathcal{T}}(\mathds{B},\mathds{E}) + \varepsilon \text{.}
\end{split}
\end{equation*}

Our result is now proven since $\varepsilon > 0$ is arbitrary.
\end{proof}

We record the following easy result:
\begin{proposition}\label{symmetric-local-propinquity-prop}
Let $\mathcal{C}$ be a nonempty class of {\pqpms s} and let $\mathcal{T}$ be a class of tunnels weakly appropriate for $\mathcal{C}$. Let $\mathds{A}$ and $\mathds{B}$ be two {\pqpms s} in $\mathcal{C}$ and let $r > 0$.
\begin{enumerate}
\item $\propinquity{r,\mathcal{T}}(\mathds{A},\mathds{B}) = \propinquity{r,\mathcal{T}}(\mathds{B},\mathds{A})$,
\item if there exists an isometric isomorphism between $\mathds{A}$ and $\mathds{B}$ then $\propinquity{r,\mathcal{T}}(\mathds{A},\mathds{B}) = 0$.
\end{enumerate}
\end{proposition}

\begin{proof}
Assertion (1) follows from the observation that if $\tau\in\mathcal{T}$ then $\tau^{-1} \in\mathcal{T}$ since $\mathcal{T}$ is appropriate, and $\tunnelextent{\tau}{r} = \tunnelextent{\tau^{-1}}{r}$.

Assertion (2) follows from the observation that for any isometric isomorphism $h$ from $\mathds{A}=(\A,\Lip,\M,\mu)$ to $\mathds{B}$, and if $\mathrm{id}_\A$ is the identity of $\A$, then for any $r > 0$, the passage $(\A,\Lip,\M,\mathrm{id}_\A,\mathds{A},h,\mathds{B})$ is a $r$-tunnel in $\mathcal{T}$ of $r$-extent $0$.
\end{proof}

\subsection{The Topographic Gromov-Hausdorff Propinquity}

The noncommutative analogue for the Gromov-Hausdorff distance which results from our research presented in this paper is:

\begin{definition}
Let $\mathcal{C}$ be a nonempty class of {\pqpms s} and let $\mathcal{T}$ be a weakly appropriate class of tunnels for $\mathcal{C}$. The \emph{$\mathcal{T}$-topographic Gromov-Hausdorff Propinquity} $\propinquity{\mathcal{T}}(\mathds{A},\mathds{B})$ is the nonnegative real number:
\begin{equation*}
\propinquity{}(\mathds{A},\mathds{B}) = \max\left\{\inf\set{\varepsilon > 0}{\propinquity{\frac{1}{\varepsilon},\mathcal{T}}(\mathds{A},\mathds{B}) < \varepsilon},\frac{\sqrt{2}}{4}\right\}\text{.}
\end{equation*}
\end{definition}

We begin with an easy observation:

\begin{lemma}
Let $\mathcal{C}$ be a nonempty class of {\pqpms s} and $\mathcal{T}$ a weakly appropriate class of tunnels for $\mathcal{C}$. Let $\mathds{A}$ and $\mathds{B}$ be two {\pqpms s} in $\mathcal{C}$. If for some $r>0$ we have:
\begin{equation*}
\propinquity{\mathcal{T}}(\mathds{A},\mathds{B}) < r
\end{equation*}
then:
\begin{equation*}
\propinquity{\frac{1}{r},\mathcal{T}}(\mathds{A},\mathds{B}) < r \text{.}
\end{equation*}
\end{lemma}

\begin{proof}
Let us be given $r>0$ such that
\begin{equation*}
\propinquity{r^{-1},\mathcal{T}}(\mathds{A},\mathds{B}) \geq r\text{.}
\end{equation*}
Then for all $R \leq r$ we have:
\begin{equation*}
\propinquity{R^{-1},\mathcal{T}}(\mathds{A},\mathds{B}) \geq \propinquity{r^{-1},\mathcal{T}}(\mathds{A},\mathds{B}) \geq r \geq R
\end{equation*}
and thus $\propinquity{\mathcal{T}}(\mathds{A},\mathds{B}) \geq r$.
\end{proof}

We note that we could use the local propinquity to define our topology as well:

\begin{theorem}
Let $\mathcal{C}$ be a nonempty class of {\pqpms s} and let $\mathcal{T}$ be a weakly appropriate class of tunnels for $\mathcal{C}$. Let $(\mathds{A}_j)_{j\in J}$ be a net of {\pqpms s} in $\mathcal{C}$ and $\mathds{L}$ a {\pqpms} in $\mathcal{C}$. The following assertions are equivalent:
\begin{enumerate}
\item for all $r > 0$ we have:
\begin{equation*}
\lim_{j\in J} \propinquity{r,\mathcal{T}}(\mathds{A}_j,\mathds{L}) = 0 \text{.}
\end{equation*}
\item we have:
\begin{equation*}
\lim_{j\in J} \propinquity{\mathcal{T}}(\mathds{A}_j,\mathds{L}) = 0 \text{.}
\end{equation*}
\end{enumerate}
\end{theorem}

\begin{proof}
Assume that for all $r > 0$ we have:
\begin{equation*}
\lim_{j\in J}\propinquity{r,\mathcal{T}}(\mathds{A}_j,\mathds{L}) = 0\text{.}
\end{equation*}
Let $\varepsilon > 0$. There exists $j_0 \in J$ such that for all $j\succ j_0$ we have $\propinquity{\varepsilon^{-1},\mathcal{T}}(\mathds{A}_j,\mathds{L}) < \varepsilon$. Hence $\propinquity{\mathcal{T}}(\mathds{A}_j,\mathds{L}) \leq \varepsilon$ and this completes the proof that (1) implies (2).

Conversely, suppose that:
\begin{equation*}
\lim_{j\in J} \propinquity{\mathcal{T}}(\mathds{A}_j,\mathds{L}) = 0 \text{.}
\end{equation*}

Let $\varepsilon > 0$ and $r>0$. Let $\epsilon = \min\{\varepsilon, r^{-1}\}$. There exists $j_0\in J$ such that for all $j\succ j_0$ we have $\propinquity{\mathcal{T}}(\mathds{A}_j,\mathds{L})<\epsilon$. Therefore:
\begin{equation*}
\propinquity{r,\mathcal{T}}(\mathds{A}_j,\mathds{L})\leq\propinquity{\epsilon^{-1},\mathcal{T}}(\mathds{A}_j,\mathds{L})\leq \epsilon\leq\varepsilon \text{,}
\end{equation*}
thus:
\begin{equation*}
\lim_{j\in J} \propinquity{r,\mathcal{T}}(\mathds{A}_j,\mathds{L}) = 0
\end{equation*}
as desired.
\end{proof}

The following theorem shows what form of the triangle inequality the topographic propinquity satisfies:

\begin{theorem}\label{propinquity-inframetric-thm}
Let $\mathcal{C}$ be a nonempty class of {\pqpms s} and let $\mathcal{T}$ be a weakly appropriate class of tunnels for $\mathcal{C}$. The propinquity $\propinquity{\mathcal{T}}$ is a pseudo infra-metric: for any three {\pqpms s} $\mathds{\A}$, $\mathds{B}$ and $\mathds{D}$ in $\mathcal{C}$, we have:
\begin{equation*}
\propinquity{\mathcal{T}}(\mathds{A},\mathds{B})\leq 2 \left(\propinquity{\mathcal{T}}(\mathds{A},\mathds{D})+\propinquity{\mathcal{T}}(\mathds{D},\mathds{B}) \right)\text{,}
\end{equation*}
while
\begin{equation*}
\propinquity{\mathcal{T}}(\mathds{A},\mathds{B}) = \propinquity{\mathcal{T}}(\mathds{B},\mathds{A})\text{,}
\end{equation*}
and $\propinquity{\mathcal{T}}(\mathds{A},\mathds{B}) = 0$ if $\mathds{A}$ and $\mathds{B}$ are isometrically isomorphic.
\end{theorem}

\begin{proof}
The symmetry of $\propinquity{\mathcal{T}}$ and the last property of our theorem follow trivially from the symmetry for local propinquities in Proposition (\ref{symmetric-local-propinquity-prop}). We now prove the triangle inequality for the propinquity.

The triangle inequality is trivial if $\max\{\propinquity{\mathcal{T}}(\mathds{A},\mathds{D}),\propinquity{\mathcal{T}}(\mathds{D},\mathds{B})\} \geq \frac{\sqrt{2}}{4}$, so we shall assume henceforth that:
\begin{equation*}
\max\{\propinquity{\mathcal{T}}(\mathds{A},\mathds{D}),\propinquity{\mathcal{T}}(\mathds{D},\mathds{B})\} < \frac{\sqrt{2}}{4}\text{.}
\end{equation*}
Let $\varepsilon > 0$, $r_1,r_2 \in \left(0,\frac{\sqrt{2}}{2}\right)$ such that:
\begin{equation*}
\propinquity{\mathcal{T}}(\mathds{A},\mathds{D}) < r_1 \leq \propinquity{\mathcal{T}}(\mathds{A},\mathds{D}) + \frac{1}{4}\varepsilon
\end{equation*}
and
\begin{equation*}
\propinquity{\mathcal{T}}(\mathds{D},\mathds{B}) < r_2 \leq \propinquity{\mathcal{T}}(\mathds{D},\mathds{B}) + \frac{1}{4}\varepsilon \text{.}
\end{equation*}

Thus $\propinquity{\frac{1}{r_1},\mathcal{T}}(\mathds{A},\mathds{D}) < r_1$ and $\propinquity{\frac{1}{r_2},\mathcal{T}}(\mathds{D},\mathds{B}) < r_2$. Without loss of generality, since $\propinquity{\mathcal{T}}$ is symmetric, we assume $r_1\leq r_2$. 

We then have:
\begin{equation*}
\begin{split}
\frac{1}{2(r_1+r_2)} + 4\max\{\{\propinquity{r_1,\mathcal{T}}(\mathds{A},\mathds{D}),\propinquity{r_2,\mathcal{T}}(\mathds{D},\mathds{B})\}&< \frac{1}{2(r_1+r_2)} + 4r_2\\
&\leq \frac{1}{2r_2} + 4r_2 = \frac{1+8r_2^2}{2r_2}\text{.}
\end{split}
\end{equation*}
Hence, if $r_2 \leq \frac{\sqrt{2}}{4}$ then $1+8r_2^2 \leq 2$ and therefore:
\begin{equation*}
\frac{1}{2(r_1+r_2)} + 4\max\{\{\propinquity{r_1,\mathcal{T}}(\mathds{A},\mathds{D}),\propinquity{r_2,\mathcal{T}}(\mathds{D},\mathds{B})\} < \frac{1}{r_2}\text{.}
\end{equation*}

Consequently, Proposition (\ref{triangle-inequality-local-propinquity-prop}) and Lemma (\ref{monotone-local-propinquity-lemma}) gives us:
\begin{equation*}
\begin{split}
\propinquity{\frac{1}{2(r_1+r_2)},\mathcal{T}}(\mathds{A},\mathds{D}) &\leq \propinquity{\frac{1}{r_2},\mathcal{T}}(\mathds{A},\mathds{B}) + \propinquity{\frac{1}{r_2},\mathcal{T}}(\mathds{B},\mathds{D})\\
&\leq \propinquity{\frac{1}{r_1},\mathcal{T}}(\mathds{A},\mathds{B}) + \propinquity{\frac{1}{r_2},\mathcal{T}}(\mathds{B},\mathds{D})\\
&< r_1 + r_2 \leq 2(r_1+r_2)\text{.}
\end{split}
\end{equation*}

Hence:
\begin{equation*}
\propinquity{}(\mathds{A},\mathds{B}) \leq 2(r_1 + r_2) \leq 2\propinquity{}(\mathds{A},\mathds{D}) + 2\propinquity{}(\mathds{D},\mathds{B}) + \varepsilon \text{.}
\end{equation*}
As $\varepsilon>0$ is arbitrary, our proof is completed.
\end{proof}

We are now ready to define our topology. The following theorem justifies our definition, which we give afterward:

\begin{theorem}
Let $\mathcal{C}$ be a nonempty class of {\pqpms s}, and $\mathcal{T}$ be a weakly appropriate class of tunnels for $\mathcal{C}$. For any subclass $\mathcal{A}$ of $\mathcal{C}$, we define the class $\closure{\mathcal{A}}$ as the class of all $\mathds{A}\in\mathcal{C}$ such that $\mathcal{A}$ is the limit of some sequence of elements in $\mathcal{A}$ for $\propinquity{\mathcal{T}}$. The operator $\closure{\cdot}$ satisfies the Kuratowsky closure operator axioms.
\end{theorem}

\begin{proof}
It is straightforward that $\closure{\emptyset} = \emptyset$ and $\closure{\mathcal{C}} = \mathcal{C}$; just as easily we check that:
\begin{enumerate}
\item for all $\mathcal{A}\subseteq\mathcal{C}$ we have $\mathcal{A}\subseteq \closure{\mathcal{A}}$,
\item for all $\mathcal{A}\subseteq\mathcal{B}\subseteq\mathcal{C}$ we have $\closure{\mathcal{A}}\subseteq\closure{\mathcal{B}}$,
\item $\closure{\mathcal{A}\cup\mathcal{B}} = \closure{\mathcal{A}}\cup\closure{\mathcal{B}}$.
\end{enumerate}

It remains to show that $\closure{\cdot}$ is idempotent. Let $\mathcal{A}$ be a subclass of $\mathcal{C}$, and $\mathds{L} \in \closure{\closure{\mathcal{A}}}$. Thus, there exists a sequence $(\mathcal{L})_{n\in\N}$ in $\closure{\mathcal{A}}$ which converges to $\mathcal{L}$ for $\propinquity{\mathcal{T}}$. Now, for each $n\in\N$ there exists a sequence $(\mathds{A}_{n,k})_{k\in\N}$ converging to $\mathcal{L}$ and contained in $\mathcal{A}$.

Let $m \in \N$ with $m > 0$. There exists $n(k)\in\N$ such that $\propinquity{\mathcal{T}}(\mathds{L}_{n(k)},\mathds{L})\leq\frac{1}{4k}$. Now, there exists $j(k) \in \N$ such that $\propinquity{\mathcal{T}}(\mathds{A}_{n(k),j(k)},\mathds{L}_{n(k)})\leq \frac{1}{4k}$. Consequently, if $\mathds{B}_k = \mathds{A}_{n(k),j(k)}$, we have:
\begin{equation*}
\propinquity{\mathcal{T}}(\mathds{B}_k,\mathds{L}) \leq 2\left(\frac{1}{4k}+\frac{1}{4k}\right) = \frac{1}{k}\text{.}
\end{equation*}
Thus $(\mathds{B}_k)_{k\in\N}$ (with $\mathds{B}_0$ arbitrary in $\mathcal{A}$) is a sequence in $\mathcal{A}$ converging to $\mathds{L}$ for $\propinquity{\mathcal{T}}$. Thus $\closure{\closure{\mathcal{A}}}\subseteq\closure{\mathcal{A}}$. The other inclusion follows from $\closure{\cdot}$ being an increasing operator. This completes our proof.
\end{proof}

\begin{definition}
Let $\mathcal{C}$ be a nonempty class of {\pqpms s} and let $\mathcal{T}$ be a weakly appropriate class of tunnels for $\mathcal{C}$. The \emph{topographic Gromov-Hausdorff quantum hypertopology} defined by $\mathcal{T}$ over $\mathcal{C}$ is the topology induced on $\mathcal{C}$ by the closure operator which associates, to any subclass $\mathcal{A}$ of $\mathcal{C}$, the class of elements in $\mathcal{C}$ which are limits of sequences of elements in $\mathcal{A}$ for $\propinquity{\mathcal{T}}$. 
\end{definition}

\subsection{Separation Property}

We now prove that the hypertopology induced by a weakly appropriate class of tunnels over a class of {\pqpms s} is separated in the sense that the propinquity between two {\pqpms s} is null if and only if they are isometrically isomorphic. Thus our hypertopology is ``Hausdorff modulo isometric isomorphism''. 

\begin{hypothesis}\label{zero-hyp}
Let $\mathcal{C}$ be a nonempty class of {\pqpms s}, and let $\mathcal{T}$ be a weakly appropriate class of tunnels over $\mathcal{C}$. Let $\mathds{A}$ and $\mathds{B}$ be {\pqpms s} in $\mathcal{C}$ such that, for each $n\in\N$ with $n > 0$, there exists a $n$-tunnel:
\begin{equation*}
\tau_n = (\D_n,\Lip_n,\M_n,\pi_n,\mathds{A},\rho_n,\mathds{B})
\end{equation*}
from $\mathds{A}=(\A,\Lip_\A,\M_\A,\mu_\A)$ to $\mathds{B}=(\B,\Lip_\B,\M_\B\mu_\B)$ with an $n$-admissible number $\varepsilon_n$ where $\varepsilon_n \leq \frac{1}{n}$. Moreover, for each $n\in\N$ with $n > 0$, let $(K^n_x)_{x\in (0,n]}$ be an increasing sequence of compact subsets of $\M_n^\sigma$ such that for all $x\in (0,n]$, the pair $(\varepsilon_n,K_x^n)$ is $x$-admissible.
\end{hypothesis}

\begin{notation}
For any $R>0$, we denote the set $\{ n\in \N : n \geq R\}$ by $\N_R$.
\end{notation}

\begin{lemma}\label{key-lemma}
Assume Hypothesis (\ref{zero-hyp}) and let $r > 0$. Let $f : \N\rightarrow\N\setminus\{0\}$ be a strictly increasing function. If $a\in\Loc{\A}{\M_\A}{\mu_\A,r}$ with $\Lip_\A(a)<\infty$, then:
\begin{enumerate}
\item there exists a strictly increasing function $\varphi : \N\rightarrow\N_r$ and $b \in \Loc{\B}{\M_\B}{\mu_\B,r}$ with $\Lip_\B(b)\leq \Lip_\A(a)$ such that for all $l \geq \Lip_\A(a)$ and $R \geq r$, the sequence:
\begin{equation*}
\left(\targetsettunnel{\tau_{f\circ\varphi(n)}}{a}{l,R,\varepsilon_{f\circ\varphi(n)},K^{f\circ\varphi(n)}_R}\right)_{n\in \N_R}
\end{equation*}
converges to $\{b\}$ in the Hausdorff distance $\Haus{\|\cdot\|_\B}$ induced by the norm of $\B$ on compact subsets of $\sa{\B}$,
\item if, for any strictly increasing $\varphi: \N\rightarrow\N$, the sequence:
\begin{equation*}
\left(\targetsettunnel{\tau_{f\circ\varphi(n)}}{a}{l,R,\varepsilon_{f\circ\varphi(n)},K^{f\circ\varphi(n)}_R}\right)_{n\in \N_R}
\end{equation*}
converges for $\Haus{\|\cdot\|_\B}$, then its limit is the singleton $\{b\}$. Moreover if $(b_n)_{n\in\N}$ is a sequence in $\sa{\B}$ such that, for some $R\geq r$ and for all $n\in\N$, we have $b_n \in \targetsettunnel{\tau_{f\circ\varphi(n)}}{a}{l,R,\varepsilon_{f\circ\varphi(n)},K^{f\circ\varphi(n)}_R}$, then $(b_n)_{n\in\N}$ converges in norm to $b$.
\item if $a\in\M_\A$ then $b\in\M_\B$.
\end{enumerate}
\end{lemma}

\begin{proof}
Let $a\in\Loc{\A}{\M_\A}{\mu_\A,r}$ for some $r > 0$ with $\Lip_\A(a)<\infty$ and $l \geq \Lip_\A(a)$. By Theorem (\ref{fundamental-thm}), since $\varepsilon_n\leq 1$ for all $n>0$, and by Lemma (\ref{compact-lift-set-lemma}), the sequence:
\begin{equation*}
\left(\targetsettunnel{\tau_{f(n)}}{a}{l,r,\varphi_{f(n)},K^{f(n)}_r}\right)_{n\in\N_r}
\end{equation*}
is a sequence of compact subsets of the set:
\begin{equation*}
\mathfrak{b} = \set{ b \in l\cdot \LipLoc{\B}{\M_\B}{r + 4\varepsilon_n}}{ \|b\|_\B \leq \|a\|_\A + 4l }
\end{equation*}
and the sets of this sequence are non-empty for $n\geq r$ by Definition (\ref{left-admissible-def}). Now, $\mathfrak{b}$ is norm compact in $\sa{\B}$. Indeed, either $\LipLoc{\B}{\M_\B}{r+4\varepsilon_n}$ is compact, or by Proposition (\ref{compact-liploc-prop}), $\B$ is unital, and thus $(\B,\Lip_\B)$ is a {\Lqcms}. In this case, the set $\{ b \in \sa{\B} : \Lip_\B(b) \leq L, \|b\|\leq M\}$ is totally bounded by \cite{Rieffel98a} for any $L, M > 0$, and is closed since $\Lip_\B$ is lower semi-continuous. Thus it is a compact set in $\sa{\B}$ which contains the closed set $\mathfrak{b}$. 

Hence, by Blaschke selection's theorem \cite{burago01}, there exists a strictly increasing function $\varphi:\N\rightarrow\N_r$ such that:
\begin{equation}\label{k-eq0}
\left(\targetsettunnel{\tau_{f\circ\varphi(n)}}{a}{l,r,\varepsilon_{f\circ\varphi(n)},K^{f\circ\varphi(n)}_r}\right)_{n\in\N}
\end{equation}
converges in the Hausdorff distance induced on the closed subsets of the compact $\mathfrak{b}$ by the norm $\|\cdot\|_\B$ of $\B$.

It is now enough to show that, for any sequence of the form given by Expression (\ref{k-eq0}) and convergent to some subset $\mathfrak{t}(a|l,r)$ of $\sa{\B}$ for the Hausdorff distance $\Haus{\|\cdot\|_\B}$ induced by the norm of $\B$ on the compact subsets of $\sa{\B}$, the limit satisfies the conclusions given in Assertion (2) and (3) of our lemma.

Now, for all $n\in\N$, Theorem (\ref{fundamental-thm}) implies that:
\begin{equation}\label{t-eq1}
\diam{\targetsettunnel{\tau_{f\circ\varphi(n)}}{a}{l,r,\varepsilon_{f\circ\varphi(n)},K^{f\circ\varphi(n)}_r}}{\|\cdot\|_\B} \leq 2l\varepsilon_{f\circ\varphi(n)}\text{.}
\end{equation}
This has two consequences. First, we conclude that:
\begin{equation*}
\diam{\mathfrak{t}(a|l,r)}{\|\cdot\|_\B} = 0\text{.}
\end{equation*}
Second, let $b_n \in \targetsettunnel{\tau_{f\circ\varphi(n)}}{a}{l,r,\varepsilon_{f\circ\varphi(n)},K^{f\circ\varphi(n)}_r}$ for all $n\in\N$. Let $\varepsilon > 0$. There exists $N\in\N$ such that, for all $n\geq N$, we have $2l\varepsilon_n \leq \frac{\varepsilon}{2}$. 
Moreover, there exists $M\in\N$ such that, for all $p,q\geq N$, we have:
\begin{equation*}
\Haus{\|\cdot\|_\B}\left(\targetsettunnel{\tau_{f\circ\varphi(p)}}{a}{l,r,\varepsilon_{f\circ\varphi(p)},K^{f\circ\varphi(p)}_r},\targetsettunnel{\tau_{f\circ\varphi(q)}}{a}{l,r,\varepsilon_{f\circ\varphi(q)},K^{f\circ\varphi(q)}_r}\right)\leq \frac{\varepsilon}{2}\text{.}
\end{equation*}
Thus, for all $p,q \geq M$, there exists $c_q\in\targetsettunnel{\tau_{f\circ\varphi(q)}}{a}{l,r,\varepsilon_{f\circ\varphi(q)},K^{f\circ\varphi(q)}_r}$ such that $\|b_p - c_q\|_\B\leq\frac{\varepsilon}{2}$, and thus:
\begin{equation*}
\begin{split}
\|b_p - b_q\|_\B &\leq \|b_p - c_q\|_\B + \|b_q - c_q\|_\B\\
&\leq \frac{\varepsilon}{2} + \frac{\varepsilon}{2} \text{ using Inequality (\ref{t-eq1}),}\\
&= \varepsilon\text{.}
\end{split}
\end{equation*}
Hence $(b_n)_{n\in\N}$ is a Cauchy sequence in the complete metric space $\mathfrak{b}$: therefore $(b_n)_{n\in\N}$ converges in norm to some $b\in \mathfrak{b}$.

By Definition (\ref{admissible-pair-def}) of admissibility and Definition (\ref{lift-set-def}) of lift sets, we conclude that $\Lip_\B(b_n)\leq l$ for all $n\in\N$. Since $\Lip_\B$ is lower semi-continuous, we conclude that $\Lip_\B(b)\leq l$.

We then check easily that $\{ b \} = \mathfrak{t}(a|l,r)$. Indeed, for all $\varepsilon>0$, there exists $N\in\N$ such that for all $n\geq N$, we have $\|b_n-b\|\leq\frac{\varepsilon}{2}$ so:
\begin{equation*}
\{b\} \subseteq_\varepsilon \targetsettunnel{\tau_{f\circ\varphi(n)}}{a}{l,r,\varepsilon_{f\circ\varphi(n)},K^{f\circ\varphi(n)}_r}\text{.}
\end{equation*}
On the other hand, there exists $M\in\N$ such that, for all $n\geq M$, we have:
\begin{equation*}
\diam{\targetsettunnel{\tau_{f\circ\varphi(n)}}{a}{l,r,\varepsilon_{f\circ\varphi(n)},K^{f\circ\varphi(n)}_r}}{\|\cdot\|_\B}\leq\frac{\varepsilon}{2}\text{.}
\end{equation*}

If $n\geq\max\{N,M\}$, then for any $c\in\targetsettunnel{\tau_{f\circ\varphi(n)}}{a}{l,r,\varepsilon_{f\circ\varphi(n)},K^{f\circ\varphi(n)}_r}$, we have:
\begin{equation*}
\|b-c\|_\B \leq \|b-b_n\|_\B + \|b_n-c\|_\B \leq \varepsilon\text{,}
\end{equation*}
hence:
\begin{equation*}
\{b\} = \Haus{\|\cdot\|_\B}\text{--}\lim_{n\rightarrow\infty} \targetsettunnel{\tau_{f\circ\varphi(n)}}{a}{l,r,\varepsilon_{f\circ\varphi(n)},K^{f\circ\varphi(n)}_r}
\end{equation*}
so $\{b\} = \mathfrak{t}(a|l,r)$ by uniqueness of limits for the Hausdorff distance on compact sets.

Now, let $L\geq l$ and $R \geq r$. we let $\varepsilon > 0$ and let $N\in\N$ such that for all $n\geq N$, we have:
\begin{equation}\label{k-eq2}
\diam{\targetsettunnel{\tau_n}{a}{L,R,\varepsilon_n,K^{n}_R}}{\|\cdot\|_\B} \leq \frac{\varepsilon}{2}\text{.}
\end{equation}
Moreover, we let $M\in\N$ such that, for all $n\geq M$, we have:
\begin{equation}\label{k-eq3}
\Haus{\|\cdot\|_\B}\left(\targetsettunnel{\tau_{f\circ\varphi(n)}}{a}{l,r,\varepsilon_{f\circ\varphi(n)},K^{f\circ\varphi(n)}_r},\{b\}\right) \leq \frac{\varepsilon}{2}\text{.}
\end{equation}
Let $n\geq \max\{N,M\}$. Let $c\in\targetsettunnel{\tau_{f\circ\varphi(n)}}{a}{L,R,\varepsilon_{f\circ\varphi(n)},K^{f\circ\varphi(n)}_R}$. Since, by Definition (\ref{lift-set-def}):
\begin{equation*}
\targetsettunnel{\tau_{f\circ\varphi(n)}}{a}{l,r,\varepsilon_{f\circ\varphi(n)},K^{f\circ\varphi(n)}_r} \subseteq \targetsettunnel{\tau_{f\circ\varphi(n)}}{a}{L,R,\varepsilon_{f\circ\varphi(n)},K^{f\circ\varphi(n)}_R}\text{,}
\end{equation*}
as $K^{f\circ\varphi(n)}_r\subseteq K_R^{f\circ\varphi(n)}$, we conclude that there exists:
\begin{equation*}
b_n \in \targetsettunnel{\tau_{f\circ\varphi(n)}}{a}{l,r,\varepsilon_{f\circ\varphi(n)},K^{f\circ\varphi(n)}_r}
\end{equation*}
such that $\|b_n-c\|_\B\leq \frac{\varepsilon}{2}$ by Inequality (\ref{k-eq2}). By Inequality (\ref{k-eq3}), in turn, we conclude that $\|b_n - b\|_\B\leq\frac{\varepsilon}{2}$. Consequently, $\|c-b\|_\B\leq\varepsilon$. From this, we conclude:
\begin{equation*}
\{b\} = \Haus{\|\cdot\|_\B}\text{--}\lim_{n\rightarrow\infty} \targetsettunnel{\tau_{f\circ\varphi(n)}}{a}{L,R,\varepsilon_{f\circ\varphi(n)},K^{f\circ\varphi(n)}_R}\text{.}
\end{equation*}
We note that the above argument also establishes that if $(b_n)_{n\in\N}$ is a sequence in $\sa{\B}$ with $b_n \in \targetsettunnel{\tau_{f\circ\varphi(n)}}{a}{L,R,\varepsilon_{f\circ\varphi(n)},K^{f\circ\varphi(n)}_R}$ for all $n\in\N$, then $(b_n)_{n\in\N}$ converges to $b$ in norm. Since we may now choose $l= \Lip_\A(a)$, we conclude that $\Lip_\B(b)\leq \Lip_\A(a)$ as desired.

Last, if $a\in\M_\A\cap\Loc{\A}{\M_\A}{\mu_\A,r}$, then by Definition (\ref{left-admissible-def}), for all $n\in\N$ there exists:
\begin{equation*}
b_n \in \targetsettunnel{\tau_{f\circ\varphi(n)}}{a}{l,r,\varepsilon_{f\circ\varphi(n)},K^{f\circ\varphi(n)}_r}\cap \M_\B
\end{equation*}
and by Lemma (\ref{key-lemma}), we conclude that $b = \lim_{n\rightarrow\infty} b_n \in \M_\B$ since $\M_\B$ is closed.

This completes our key lemma.
\end{proof}

\begin{lemma}\label{diagonal-lemma}
Assume Hypothesis (\ref{zero-hyp}). If $\mathfrak{a}$ is a countable subset of $\Loc{\A}{\M_\A}{\star}$ with a subset dense in $\Loc{\A}{\M_\A}{\star}\cap\M_\A$, then there exists a strictly increasing function $\psi : \N \rightarrow \N\setminus\{0\}$ and a function $\pi : \mathfrak{a} \rightarrow \sa{\B}$ such that, for all $a \in \mathfrak{a}$, for all $l \geq \Lip_\A(a)$ and for all $r > 0$ with $a\in\Loc{\A}{\M_\A}{\mu_\A,r}$:
\begin{enumerate}
\item we have:
\begin{equation*}
\Haus{\|\cdot\|_\B}\text{--}\lim_{n\rightarrow\infty}\targetsettunnel{\tau_{\psi(n)}}{a}{l,r,\varepsilon_{\psi(n)},K^{\psi(n)}_r} = \{\pi(a)\}\text{,}
\end{equation*}
\item we have $\Lip_\B(h(a))\leq \Lip_\A(a)$,
\item we have $h(a) \in \Loc{\B}{\M_\B}{\mu_\B,r}$,
\item if $a\in\M_\A$ then $\pi_\A(a)\in\M_\B$,
\item the function $\pi$ is uniformly continuous in norm.
\end{enumerate}
\end{lemma}

\begin{proof}
We employ a diagonal procedure to prove this lemma. We write $\mathfrak{a} = \{a_n : n\in\N\}$ in such a way that we may choose a nondecreasing sequence $(r_n)_{n\in\N}$ such that $a_n \in \Loc{\A}{\M_\A}{\mu_\A,r_n}$.

Let $\varphi_0 :\N\rightarrow\N$ be a function be a strictly increasing function given Lemma (\ref{key-lemma}) for $a_0$ (note that this function does not depend on our choice of $r_0$). Assume we have constructed $\varphi_0,\varphi_1,\ldots,\varphi_m$ for some $m\in\N$, all strictly increasing functions from $\N$ to $\N\setminus\{0\}$, such that:
\begin{equation}\label{diag-eq0}
\left(\targetsettunnel{\tau_{\psi_j(n)}}{a_j}{\Lip_\A(a_j),r_j,\varepsilon_{\psi_j(n)},K^{\psi_j(n)}_{r}}\right)_{n\in\N}
\end{equation}
converges to some $\{\pi(a_j)\}$, where $\psi_j = \varphi_0\circ\ldots\circ\varphi_j$, for all $j\in\{0,\ldots,m\}$. 

Since $\psi_m$ is a strictly increasing function from $\N$ to $\N\setminus\{0\}$, Lemma (\ref{key-lemma}) applied to $a_{m+1}$ provides us with a function $\varphi_{m+1}$ such that:
\begin{equation*}
\left(\targetsettunnel{\tau_{\psi_m\circ\varphi_{m+1}(n)}}{a_{m+1}}{\Lip_\A(a_{m+1}),r_{m+1},\varepsilon_{\psi_m\circ\varphi_{m+1}(n)},K^{\psi_m\circ\varphi_{m+1}(n)}_{r_{m+1}}}\right)_{n\in\N}
\end{equation*}
converges to some singleton, denoted by $\{\pi(a_{m+1})\}$. The induction hypothesis is thus satisfied, since:
\begin{equation*}
\left(\targetsettunnel{\tau_{\psi_j\circ\varphi_{m+1}(n)}}{a_j}{\Lip_\A(a_j),r_j,\varepsilon_{\psi_j\circ\varphi_{m+1}(n)},K^{\psi_j\circ\varphi_{m+1}(n)}_{r_{m+1}}}\right)_{n\in\N}
\end{equation*}
is a subsequence of the sequence in Expression (\ref{diag-eq0}) for all $j\in\{0,\ldots,m\}$. 

We then define $\psi : \N \rightarrow \N$ by setting for all $n\in\N$:
\begin{equation*}
\psi(n) = \varphi_0\circ\varphi_1\circ\ldots\circ\varphi_n(n) = \circ_{j=0}^n\varphi_j(n)\text{.}
\end{equation*}
Now, $\psi$ is a strictly increasing function from $\N$ to $\N\setminus\{0\}$, and moreover, for any fix $m\in\N$, the sequence:
\begin{equation*}
\left(\targetsettunnel{\tau_{\psi(n)}}{a_m}{\Lip_\A(a_m),r_m,\varepsilon_{\psi(n)},K^{\psi(n)}_{r_m}}\right)_{n\in\N_m}
\end{equation*}
is a subsequence of $\left(\targetsettunnel{\tau_{\psi_m(n)}}{a_m}{\Lip_\A(a),r_m,\frac{1}{\psi_m(n)},K^{\psi_m(n)}_{r}}\right)_{n\in\N_m}$, and thus converges to some singleton $\{\pi(a)\}$ for the Hausdorff distance associated with the norm of $\B$.

Now, by Lemma (\ref{key-lemma}), for any $n \in \N$, $r > 0$ and $l > 0$ such that $a_n \in \Loc{\A}{\M_\A}{\mu_\A,r}$ and $\Lip_\A(a_n)\leq l$, we conclude that:
\begin{equation*}
\left(\targetsettunnel{\tau_{\psi(n)}}{a_m}{l,r,\varepsilon_{\psi(n)},K^{\psi(n)}_r}\right)_{n\in\N_m}
\end{equation*}
converges in $\Haus{\|\cdot\|_\B}$ to $\{\pi(a)\}$ as desired. 

Let now $\varepsilon > 0$ and $a,a'\in\mathfrak{a}$ with $\|a-a'\|_\A\leq\frac{\varepsilon}{4}$. Let $l > \max\{\Lip_\A(a),\Lip_\A(a')\}$ and $r > 0$ such that $a,a'\in\Loc{\A}{\M_\A}{\mu_\A,r}$. By construction of $\pi$, there exists $N\in\N$ such that for all $n\geq N$ we have:
\begin{equation*}
\Haus{\|\cdot\|_\B}\left(\{\pi(a)\},\targetsettunnel{\tau_{\psi(n)}}{a}{l,r,\varepsilon_{\psi(n)},K^{\psi(n)}_r}\right) \leq \frac{\varepsilon}{4}
\end{equation*}
and similarly, there exists $N'\in\N$ such that for all $n\geq N'$ we have:
\begin{equation*}
\Haus{\|\cdot\|_\B}\left(\{\pi(a')\},\targetsettunnel{\tau_{\psi(n)}}{a'}{l,r,\varepsilon_{\psi(n)},K^{\psi(n)}_r}\right) \leq \frac{\varepsilon}{4}\text{.}
\end{equation*}
Let $b \in \targetsettunnel{\tau_{\psi(n)}}{a}{l,r,\varepsilon_{\psi(n)},K^{\psi(n)}_r}$ and $b'\in \targetsettunnel{\tau_{\psi(n)}}{a'}{l,r,\varepsilon_{\psi(n)},K^{\psi(n)}_r}$ such that:
\begin{equation*}
\|b-\pi(a)\|_\B\leq\frac{\varepsilon}{4}\text{ and }\|b'-\pi(a')\|_\A\leq\frac{\varepsilon}{4}\text{.}
\end{equation*}
We note that the inequality are weak since we are working with compact sets, so the Hausdorff distance is reached.

Let $P \in \N$ such that for all $n\geq P$, we have $\varepsilon_n\leq\frac{\varepsilon}{8 l}$. If $n\geq\max\{N,M,P\}$ then, by Inequality (\ref{distance-eq}) of Theorem (\ref{fundamental-thm}), we have:
\begin{equation*}
\|b-b'\|_\A\leq \|a-a'\|_\A + 2l\varepsilon_{\psi(n)} \leq \frac{\varepsilon}{2} + \frac{\varepsilon}{4} \leq \frac{\varepsilon}{2}\text{.}
\end{equation*}
Thus:
\begin{equation*}
\|\pi(a)-\pi(a')\|_\B \leq \|\pi(a)-b\|_\B + \|b - b'\|_\B + \|b' - \pi_\A(b')\|_\A \leq \varepsilon\text{.}
\end{equation*}

This concludes our proof, as we also observe that $\pi(\M_\A)\subseteq\M_\B$ since $\M_\B$ is closed.
\end{proof}

\begin{corollary}\label{map-cor}
Assume Hypothesis (\ref{zero-hyp}). There exists $\psi : \N\rightarrow\N\setminus\{0\}$ strictly increasing, and $\pi:\sa{\A}\rightarrow\sa{\B}$ uniformly continuous in norm, such that:
\begin{enumerate}
\item for all $a\in\Loc{\A}{\M_\A}{\star}$, and for any $l \geq \Lip_\A(a)$ and $r > 0$ such that $a\in\Loc{\A}{\M_\A}{\mu_\A,r}$, we have:
\begin{equation*}
\Haus{\|\cdot\|_\B}\text{--}\lim_{n\rightarrow\infty} \targetsettunnel{\tau_{\psi(n)}}{a}{l,r,\varepsilon_{\psi(n)},K^{\psi(n)}_r} = \{ \pi(a) \}\text{,}
\end{equation*}
\item for all $a\in \Loc{\A}{\M_\A}{\star}$ we have $\Lip_\B(\pi(a))\leq \pi(a)$,
\item for all $a\in \sa{\A}$ we have $\|\pi(a)\|_\B \leq \|a\|_\A$,
\item for all $a\in\sa{\M_\A}$ we have $\pi(a) \in \sa{\M_\B}$.
\end{enumerate}
\end{corollary}

\begin{proof}
Let $\mathfrak{a}$ be a countable subset of $\sa{\A}$ such that:
\begin{enumerate}
\item for all $a\in\mathfrak{a}$, we have $a\in\Loc{\A}{\M_\A}{\mu_\A,r}$ for some $r>0$,
\item $\mathfrak{a}$ is dense in $\sa{\A}$,
\item $\mathfrak{a}\cap\M_\A$ is dense in $\sa{\M_\A}$.
\end{enumerate}
Such a set exists since $\A$ by Definition (\ref{pqms-def}).

Let $\psi :\N\rightarrow\N$ and $\pi:\mathfrak{a}\rightarrow\sa{\B}$ given by Lemma (\ref{diagonal-lemma}) applied to $\mathfrak{a}$. Since $\pi$ is uniformly continuous on the dense subset $\mathfrak{a}$ of $\sa{\A}$, it admits a unique uniformly continuous extension to $\sa{\A}$, which we still denote by $\pi$.

More can be said for Lipschitz, locally supported elements. Let $a\in\Loc{\A}{\M_\A}{r}$ for some $r > 0$. Let $\varepsilon > 0$. Since $\pi$ is uniformly continuous on $\sa{\A}$, there exists $\delta > 0$ such that, if $a,a'\in\sa{\A}$ with $\|a-a'\|_\A\leq\delta$, then $\|\pi(a)-\pi(a')\|_\B\leq\frac{\varepsilon}{4}$.

Since $\mathfrak{a}$ is dense in $\sa{\A}$, there exists $a'\in\mathfrak{a}$ with $\|a-a'\|_\A\leq \min\{\frac{\varepsilon}{4},\delta\}$. We may choose $a'\in\M_\A$ if $a\in\M_\A$. Let $l > \max\{\Lip_\A(a),\Lip_\A(a')\}$ and $R = \max\{r,r'\}$ if $a'\in\Loc{\A}{\M_\A}{\mu_\A,r'}$. 

Let $P \in \N$ such that for all $n\geq P$ we have $\varepsilon_n\leq\frac{\varepsilon}{8 l}$. Let $n\geq P$. Let $b_n \in \targetsettunnel{\tau_{\psi(n)}}{a}{\Lip_\A(a),r,\varepsilon_{\psi(n)},K_r^{\psi(n)}}$. 

For all $n\in\N$ let $c_n \in \targetsettunnel{\tau_{\psi(n)}}{a'}{l,r,\varepsilon_{\psi(n)},K^{\psi(n)}_r}$. By Lemma (\ref{key-lemma}), the sequence $(c_n)_{n\in\N}$ converges to $\pi(a')$, so there exists $N\in\N$ such that for all $n\geq N$ we have $\|c_n-\pi(a')\|_\B \leq\frac{\varepsilon}{4}$. 

Then, using again the observation that:
\begin{equation*}
\targetsettunnel{\tau_{\psi(k)}}{a}{\Lip_\A(a),r,\varepsilon_{\psi(k)},K^{\psi(k)}_r}\subseteq\targetsettunnel{\tau_{\psi(k)}}{a}{l,r,\varepsilon_{\psi(k)},K^{\psi(k)}_r}
\end{equation*}
for all $k\in\N$, and Inequality (\ref{distance-eq}) of Theorem (\ref{fundamental-thm}), we obtain that for all $n \geq \max\{P,N\}$:
\begin{equation*}
\begin{split}
\|b_n-\pi(a)\|_\B &\leq \|b_n - c_n\|_\B + \|c_n - \pi(a')\|_\B + \|\pi(a') - \pi(a)\|_\B\\
&\leq \|a-a'\|_\A + 2l\varepsilon_{\psi(n)} + \frac{\varepsilon}{4} + \frac{\varepsilon}{4} \\
&\leq \varepsilon\text{.}
\end{split}
\end{equation*}

This shows that the sequence:
\begin{equation}\label{c-eq1}
\left(\targetsettunnel{\tau_{\psi(k)}}{a}{\Lip_\A(a),r,\varepsilon_{\psi(k)},K_r^{\psi(k)}}\right)_{k\in\N}
\end{equation}
converges to $\{\pi(a)\}$ for the Hausdorff distance $\Haus{\|\cdot\|_\B}$.  Note that if $a\in\M_\A$ then $\pi(a)\in\M_\B$ as $\M_\B$ is closed. Lemma (\ref{key-lemma}) completes our proof; in particular we note that since $\pi(a)$ is the limit of elements with Lip-norm less than $\Lip_\A(a)$ for all $a\in\sa{\A}$, we must have $\Lip_\B(\pi_\A(a))\leq \Lip_\A(a)$.
\end{proof}

\begin{lemma}\label{support-lemma}
Let $(\A,\Lip,\M,\mu)$ be a {\pqpms} and $a\in\Loc{\A}{\M}{\star}$. If there exists a decreasing sequence $(r_n)_{n\in\N}$ converging to some $r>0$ and such that $a\in\Loc{\A}{\M}{\mu_\A,r_n}$, then $a\in \Loc{\A}{\M}{\mu_\A,r}$.
\end{lemma}
\begin{proof}
Let $p_n$ be the indicator function of $\cBall{\M_\A^\sigma}{}{\mu}{r_n}$ and $p$ the indicator function of $\cBall{\M_\A^\sigma}{}{\mu}{r}$. By assumption, $(p_n)_{n\in\N}$ is a decreasing sequence which converges pointwise to $p$. Thus, by the monotone convergence theorem, if $\varphi\in\StateSpace(\A)$ then $(\varphi(p-p_n))_{n\in\N}$ converges to $0$ (note that $\M$ is Abelian so the restriction of $\varphi$ to $\M$ is a Radon probability measure).

By assumption, $p_n a p_n = a$ for all $n\in\N$, so in particular, $p_0 a p_0 = a$ and thus $a$, and $p_n a p_n$ for all $n\in\N$, are in $\Loc{\A}{\M_\A}{\mu,r_0}$. Let $\varphi\in\StateSpace[\A|\mu,r_0]$. Then:
\begin{equation*}
\begin{split}
|\varphi(a-p a p)| &= |\varphi(p_n a p_n - p a p)|\\
&\leq 3 \sqrt{\varphi(p_n - p)} \|a\|_\A \text{ by Cauchy-Schwarz,}\\
&\stackrel{n\rightarrow\infty}{\longrightarrow} 0 \text{.}
\end{split}
\end{equation*}
Thus, $\|a-p a p\|_\A = 0$ and this concludes our lemma.
\end{proof}

\begin{theorem}\label{coincidence-thm}
Let $\mathcal{C}$ be a nonempty class of {\pqpms s} and let $\mathcal{T}$ be a weakly appropriate class of tunnels for $\mathcal{C}$. If:
\begin{equation*}
\propinquity{\mathcal{T}}(\mathds{A},\mathds{B}) = 0
\end{equation*}
for $\mathds{A},\mathds{B} \in \mathcal{C}$, then there exists a pointed isometric isomorphism $\pi : \mathds{A} \longrightarrow \mathds{B}$.
\end{theorem}

\begin{proof}
We denote $\mathds{A} = (\A,\Lip_\A,\M_\A,\mu_\A)$ and $\mathds{B} = (\B,\Lip_\B,\M_\B,\mu_\B)$. By assumption, for all $n\in\N$ with $n > 0$, there exists a $n$-tunnel $\tau_n$ such that:
\begin{equation*}
\tunnelextent{\tau_n}{n} < \varepsilon_n\text{.}
\end{equation*}
For all $n\in\N$ with $n > 0$, let $\varepsilon_n \leq\frac{1}{n}$ be $n$-admissible, so that we now may assume Hypothesis (\ref{zero-hyp}).

Let $\pi$ and $\psi$ be provided by Corollary (\ref{map-cor}). Let $a,a'\in \Loc{\A}{\M_\A}{\star}$. Let $r > 0$ so that $a,a'\in\Loc{\A}{\M_\A}{\mu_\A,r}$ and let $l \geq\max\{\Lip_\A(a),\Lip_\A(a')\}$.

For each $n\in\N$, let:
\begin{equation*}
b_n\in\targetsettunnel{\tau_{\psi(n)}}{a}{r,l,\varepsilon_{\psi(n)},K^{\psi(n)}_r}\text{ and }b'_n\in\targetsettunnel{\tau_{\psi(n)}}{a'}{r,l,\varepsilon_{\psi(n)},K^{\psi(n)}_r}\text{.}
\end{equation*}
By Lemma (\ref{key-lemma}), we conclude that $(b_n)_{n\in\N}$ converges to $\pi(a)$ and $(b'_n)_{n\in\N}$ converges to $\pi(a')$.

We begin with the proof that $\pi$ is linear. Let $t\in\R$. By Assertion (\ref{linearity-eq}) of Theorem (\ref{fundamental-thm}), we thus have:
\begin{equation*}
b_n + tb_n' \in \targetsettunnel{\tau_{\psi(n)}}{a+ta'}{r,(1+|t|)l,\varepsilon_{\psi(n)},K^{\psi(n)}_r}
\end{equation*}
and thus, by Lemma (\ref{key-lemma}), we conclude that $(b_n+tb_n')_{n\in\N}$ converges to $\pi(a+ta')$. On the other hand, by continuity of the vectorial operations, $b_n+tb_n'$ converges to $\pi(a)+t\pi(a')$. Thus $\pi(a+ta')=\pi(a)+t\pi(a')$, i.e. $\pi$ is linear on $\Loc{\A}{\M_\A}{\star}$. As $\pi$ is continuous on $\sa{\A}$ and $\Loc{\A}{\M_\A}{\star}$ is dense in $\sa{\A}$, we conclude that $\pi$ is linear on $\sa{\A}$.

We now prove that $\pi$ is a Jordan-Lie morphism. By Assertion (\ref{Jordan-eq}) of Theorem (\ref{fundamental-thm}), we have:
\begin{equation*}
\Jordan{b_n}{b_n'} \in \targetsettunnel{\tau_{\psi(n)}}{\Jordan{a}{a'}}{r,l\left(\|a\|_\A + \|a'\|_\A + 2\varepsilon_{\psi(n)}l\right),\varepsilon_{\psi(n)},K^{\psi(n)}_r}\text{.}
\end{equation*}
By Lemma (\ref{key-lemma}), we conclude that $(\Jordan{b}{b'})_{n\in\N}$ converges to $\pi(\Jordan{a}{a'})$. On the other hand, by continuity of the Jordan product, $\lim_{n\rightarrow\infty}\Jordan{b_n}{b_n'} = \Jordan{\pi(a)}{\pi(a')}$. Consequently, $\pi(\Jordan{a}{a'})=\Jordan{\pi(a)}{\pi(a')}$.

Similarly, $\pi(\Lie{a}{a'})=\Lie{\pi(a)}{\pi(a')}$ using Assertion (\ref{Lie-eq}) of Theorem (\ref{fundamental-thm}) and Lemma (\ref{key-lemma}), since the Lie product is also continuous.

This proves that $\pi : \sa{\A}\rightarrow\sa{\B}$ is a Jordan-Lie morphism since $\pi$ is continuous on $\sa{\A}$ and $\Loc{\A}{\M_\A}{\star}$ is dense in $\sa{\A}$. We already have the properties that $\|\pi(a)\|_\B\leq\|a\|_\A$ and $\Lip_\B(\pi(a))\leq\Lip_\A(a)$ by Corollary (\ref{map-cor}).

We now proceed to extent $\pi$ by linearity: for any $a\in\A$, we write $\Re(a) = \frac{a+a^\ast}{2}$ and $\Im(a) = \frac{a-a^\ast}{2i}$ so that $\Re(a),\Im(a)\in\sa{\A}$ and $a = \Re(a) + i\Im(a)$. We set, for all $a\in\A$:
\begin{equation*}
\pi(a) = \pi(\Re(a)) + i\pi(\Im(a))\text{.}
\end{equation*}
Our definition does not induce any confusion since $a\in\sa{\A}$ if and only if $\Im(a)=0$ and $\pi(0) = 0$, and thus we keep the same notation for $\pi$ and its extension to $\A$. As $\pi$ is linear on $\sa{\A}$, it is straightforward that $\pi$ is linear on $\A$.

It is also easy to check that $\pi(a^\ast) = \pi(a)^\ast$ for all $a\in\A$. It remains to check that this extension of $\pi$ is multiplicative. We first observe that, for all $a,b \in \sa{\A}$:
\begin{equation}\label{mul-eq1}
\begin{split}
\pi(ab) &= \pi(\Jordan{a}{b} + \Lie{a}{b}) \\
&= \pi(\Jordan{a}{b}) + \pi(\Lie{a}{b})\\
&= \Jordan{\pi(a)}{\pi(b)} + \Lie{\pi(a)}{\pi(b)}\\
&= \pi(a)\pi(b) \text{,}
\end{split}
\end{equation}
since $\pi$ is a Jordan-Lie morphism on $\sa{\A}$ (note that $ab$ itself is not in $\sa{\A}$ in general).

Now, if $a,b \in \A$, then, since $\Re(a),\Im(a) \in\sa{\A}$:
\begin{equation*}
\begin{split}
\pi(ab) &= \pi(\Re(a)\Re(b)) - \pi(\Im(a)\Im(b)) + i\left(\pi(\Re(a)\Im(b)) + \pi(\Im(a)\Re(b)) \right)\\
&= \pi(\Re(a))\pi(Re(b)) - \pi(\Im(a))\pi(\Im(b)) \\
&\quad + i\left(\pi(\Re(a))\pi(\Im(b)) + \pi(\Im(a))\pi(\Re(b)) \right) \text{by Equation (\ref{mul-eq1}),}\\
&= Re(\pi(a))\Re(\pi(b)) - \Im(\pi(a))\Im(\pi(b)) \\
&\quad + i\left(\Re(\pi(a))\Im(\pi(b)) + \Im(\pi(a))\Re(\pi(b)) \right) \text{ as $\pi$ linear, self-adjoint,}\\
&= (\Re(\pi(a))+i\Im(\pi(a)))(\Re(\pi(b))+i\Im(\pi(b)))\\
&= \pi(a)\pi(b)\text{.}
\end{split}
\end{equation*}

This concludes the proof that $\pi:\A\rightarrow\B$ is a *-morphism with $\Lip_\B\circ\pi\leq\pi_\A$ on $\Loc{\A}{\M_\A}{\star}$. In particular, we have $\|\pi(a)\|_\B\leq\|a\|_\A$ for all $a\in\A$.

It remains to prove that $\pi$ is an isomorphism. We first observe that by Lemma (\ref{support-lemma}), we have: 
\begin{equation*}
a\in\Loc{\A}{\M_\A}{\mu_\A,r} \implies \pi(a) \in \Loc{\B}{\M_\B}{\mu_\B,r}\text{,}
\end{equation*}
since $\pi(a) \in \Loc{\B}{\M_\B}{\mu_\B,r+4\varepsilon_n}$ for all $n\in\N$, and we can always extract a monotone subsequence from $(\varepsilon_n)_{n\in\N}$; since this sequence converges to $0$ and consists of positive numbers, it must be decreasing.

Let us now apply the construction of Lemma (\ref{key-lemma}), Lemma (\ref{diagonal-lemma}) and Corollary (\ref{map-cor}) to the sequence of tunnels $(\tau_{\psi(n)}^{-1})_{n\in\N}$, which satisfies Hypothesis (\ref{zero-hyp}). We thus get a map $\rho : \sa{\B} \rightarrow \sa{\A}$ which, using the method we applied to $\pi$ in the current proof, can be extended to a *-morphism $\rho : \B\rightarrow\A$ with $\Lip_\A\circ\rho \leq\Lip_\B$ on $\Loc{\B}{\M_\B}{\star}$, and a strictly increasing function $\theta :\N\rightarrow\N$ such that, for all $r > 0$ and for all $b\in\Loc{\B}{\M_\B}{\mu_\B,r}$ and for all $l\geq \Lip_\B(b)$:
\begin{equation*}
\{ \rho(a) \} = \Haus{\|\cdot\|_\A}\text{--}\lim_{n\rightarrow\infty} \targetsettunnel{\tau^{-1}_{\psi\circ\theta(n)}}{b}{r,l,\varepsilon_{\psi\circ\theta(n)},K^{\psi\circ\theta(n)}_r}\text{.}
\end{equation*}

We check that $\pi\circ\rho$ and $\rho\circ\pi$ are respectively given by the identity map of $\A$ and $\B$. It is sufficient to check this on a norm dense subset of $\sa{\A}$, and by symmetry in $\A$ and $\B$, it is sufficient to write the proof for $\pi\circ\rho$.

Let $a\in\Loc{\A}{\M_\A}{\mu_\A,r}$ for some $r > 0$. By construction:
\begin{equation*}
\{\pi(a)\} = \Haus{\|\cdot\|_\B}\text{--}\lim_{n\rightarrow\infty} \targetsettunnel{\tau_{\psi(n)}}{a}{l,r,\varepsilon_{\psi(n)},K^{\psi(n)}_{r}}\text{.}
\end{equation*}
Similarly, we have:
\begin{equation*}
\{\rho(\pi(a))\} = \Haus{\|\cdot\|_\B}\text{--}\lim_{n\rightarrow\infty} \targetsettunnel{\tau_{\psi\circ\theta(n)}^{-1}}{\pi(a)}{l,r,K^{\psi\circ\theta(n)}_r}\text{.}
\end{equation*}

For all $n\in\N$ we let:
\begin{equation*}
 b_n \in \targetsettunnel{\tau_{\psi\circ\theta(n)}}{a}{l,r,\varepsilon_{\psi\circ\theta(n)},K^{\psi\circ\theta(n)}_r}
\end{equation*}
and
\begin{equation*}
a_n\in\targetsettunnel{\tau_{\psi\circ\theta(n)}^{-1}}{\pi(a)}{l,r,\varepsilon_{\psi\circ\theta(n)},K^{\psi\circ\theta(n)}_{r}}\text{.}
\end{equation*}

By Lemma (\ref{key-lemma}), there exists $N_1\in\N$ such that for all $n\geq N_1$ we have:
\begin{equation*}
\|b_n-\pi(a)\|_\B\leq \frac{\varepsilon}{3}\text{,}
\end{equation*}
and there exists $N_2\in\N$ such that:
\begin{equation*}
\|a_n-\rho(\pi(a))\|_\A \leq\frac{\varepsilon}{3}\text{.}
\end{equation*}

Let $M\in\N$ be chosen in $\N_{r+4\varepsilon}$ so that for all $n\geq M$, we have $2l\varepsilon_n\leq \frac{\varepsilon}{3}$. Let $n = \max\{M,N_1,N_2\}$. 

We note first that since $b_n\in\targetsettunnel{\tau_{\psi\circ\theta(n)}}{a}{l,r,\varepsilon_{\psi\circ\theta(n)},K^{\psi(n)}_r}$, we have:
\begin{equation*}
a\in\targetsettunnel{\tau_{\psi\circ\theta(n)}^{-1}}{b_n}{l,r,\varepsilon_{\psi(n)},K^{\psi(n)}_{r}}
\end{equation*}
by Remark (\ref{inversion-rmk}). 

Now, since $\psi\circ\theta(n)\geq n \geq n+4\varepsilon$, we conclude that $\tau^{-1}_{\psi\circ\theta_n}$ and its inverse are $(r+4\varepsilon)$-tunnels. Moreover, we have by our choice of notations that $K_{r}^{\psi\circ\theta(n)}\subseteq K_{r+4\varepsilon}^{\psi\circ\theta(n)}$. Thus we observe that:
\begin{multline*}
a\in\targetsettunnel{\tau^{-1}_{\psi\circ\theta(n)}}{b_n}{l,r+4\varepsilon,\varepsilon_{\psi\circ\theta(n)},K_{r+4\varepsilon}^{\psi\circ\theta(n)}}\text{ and }\\ a_n \in \targetsettunnel{\tau^{-1}_{\psi\circ\theta(n)}}{\pi(a)}{l,r+4\varepsilon,\varepsilon_{\psi\circ\theta(n)},K_{r+4\varepsilon}^{\psi\circ\theta(n)}}\text{.}
\end{multline*}

Thus, by Theorem (\ref{fundamental-thm}):
\begin{equation*}
\begin{split}
\|a-\rho(\pi(a))\|_\A &\leq  \|a-a_n\|_\A + \|a_n-\rho(\pi(a))\|_\A\\
&\leq 2l\varepsilon_{\psi\circ\theta(n)} + \|b_n - \pi(a)\|_\B + \frac{\varepsilon}{3}\\
&\leq \frac{\varepsilon}{3} + \frac{\varepsilon}{3} + \frac{\varepsilon}{3}\\
&= \varepsilon\text{.}
\end{split}
\end{equation*}

We conclude, as desired, that $\|a-\rho(\pi(a))\|_\A = 0$. By continuity, since $\Loc{\A}{\M_\A}{\star}$ is dense, we conclude that $\rho\circ\pi$ is the identity of $\A$. The same reasoning shows that $\pi\circ\rho$ is the identity of $\B$.

Moreover, we note that for all $a\in\Loc{\A}{\M_\A}{\star}$:
\begin{equation*}
\Lip_\A(a) = \Lip_\A(\rho\circ\pi(a)) \leq \Lip_\B(\pi(a)) \leq \Lip_\A(a)
\end{equation*}
and thus $\pi\circ\Lip_\A = \Lip_\B$; the same results holds for $\rho$.

Last, let $a\in\Loc{\A}{\M_\A}{r}$ for some $r > 0$ and $\varepsilon > 0$. There exists $n\in\N$ such that $\varepsilon_n\leq\varepsilon$ and $n \geq r$ while:
\begin{equation*}
\Haus{\|\cdot\|_\B}\left(\targetsettunnel{\tau_n}{a}{\Lip_\A(a),r,\varepsilon_n,K^n_r},\{\pi(a)\}\right) < \varepsilon\text{.}
\end{equation*}

To fix notations, write:
\begin{equation*}
\tau_{\psi(n)} = (\D_n,Lip_n,\M_n,\pi_n,\mathds{A},\rho_n,\mathds{B})\text{.}
\end{equation*}
Let $d\in\liftsettunnel{\tau_n}{a}{\Lip_\A(a),r,\varepsilon_n,K^n_r}$ and $b = \rho_n(d)$. Then:
\begin{equation*}
\begin{split}
|\mu_\B\circ\pi(a) - \mu_\A(a)| &\leq |\mu_\B\circ\pi(a)-\mu_\B(b)| + |\mu_\B(b) - \mu_\A(a)| \\
&\leq \Lip_\A(a)\varepsilon + |\mu_\A\circ\pi_n(d)-\mu_\B\circ\rho_n(d)|\\
&\leq \Lip_\A(a)\varepsilon + \Lip_\A(a)\varepsilon = 2\Lip_\A(a)\varepsilon\text{.}
\end{split}
\end{equation*}
Since $\varepsilon > 0$, we conclude that $\mu_\B\circ\pi(a) = \mu_\A(a)$. As $\Loc{\A}{\M_\A}{\star}$ is dense in $\sa{\A}$, we conclude that $\mu_\B\circ\pi = \mu_\A$ (and similarly $\mu_\A\circ\rho = \mu_\B$). This concludes our proof.
\end{proof}
We observe that in the proof of Theorem (\ref{coincidence-thm}) and using its notations, we could replace $\psi$ by $\psi\circ f$ for any strictly increasing function $f:\N\rightarrow\N$ prior to the construction of $\rho$. We thus would construct, a priori, different versions of $\rho$ --- one for each choice of $f$ --- but our proof shows that they all are inverse of $\pi$, and thus they all agree. A consequence of this observation is that $\theta$ may in fact be chosen to be the identity.

We also note that of course, once an isometric isomorphism is given between two {\pqpms s}, then the composition on either end by an isometric automorphism would lead to a new isometric isomorphism. This non-uniqueness implies that a choice must be made in the construction of Theorem (\ref{coincidence-thm}). This choice, of course, is given by the construction in Lemma (\ref{diagonal-lemma}), itself relying on the choices made in Lemma (\ref{key-lemma}). Both, in turn, are consequences of compactness of the Hausdorff distance on a well-chosen set.

When dealing with strongly proper quantum metric spaces, we may slightly strengthen the separation result as follows:

\begin{corollary}
Let $\mathcal{C}$ be a nonempty class of {\pqpms s} and let $\mathcal{T}$ be a weakly appropriate class of tunnels for $\mathcal{C}$. If:
\begin{equation*}
\propinquity{\mathcal{T}}(\mathds{A},\mathds{B}) = 0
\end{equation*}
for $\mathds{A}=(\A,\Lip_\A,\M_\A,\mu_A)$,$\mathds{B}=(\B,\Lip_\B,\M_\B,\mu_\B) \in \mathcal{C}$ strongly proper quantum metric spaces, then there exists a pointed isometric isomorphism $\pi : \mathds{A} \longrightarrow \mathds{B}$ such that for all $a\in\sa{\A}$ we have $\Lip_\B\circ\pi(a) = \Lip_\A(a)$ and for all $b\in\sa{\B}$ we have $\Lip_\A\circ\pi^{-1}(b) = \Lip_\B(b)$.
\end{corollary}

\begin{proof}
Theorem (\ref{coincidence-thm}) provides us with an isometric isomorphism $\pi$. Let $a\in\sa{\A}$ with $\Lip_\A(a)<\infty$. Then by Proposition (\ref{strongly-proper-implies-proper-prop}), there exists $(a_n)_{n\in\N}$ such that $a_n\in\Loc{\A}{\M_\A}{\star}$ for all $n\in\N$, converging to $a$ in norm and with $\lim_{n\rightarrow \infty}\Lip_\A(a_n) = \Lip_\A(a)$.

Now, $(\pi(a_n))_{n\in\N}$ converges to $\pi(a)$ by continuity, and $(\Lip_\B\circ\pi(a_n))_{n\in\N} = (\Lip_\A(a_n))$ converges to $\Lip_\A(a)$, as desired. The result is symmetric in $\A$ and $\B$.
\end{proof}

\section{Comparison}

This section proves that our hypertopology on the class of {\pqpms s} restricts to the dual Gromov-Hausdorff propinquity topology for {\Lqcms s} and is weaker than the Gromov-Hausdorff distance on classical proper metric spaces.

\subsection{Classical Gromov-Hausdorff Distance}

Our hypertopology is weaker than the Gromov-Hausdorff topology on pointed proper metric spaces. To establish this result, we recall from \cite{Latremoliere12b} that, to any pointed proper metric space $(X,\mathsf{d}_X,x)$, we associate a canonical {\pqpms} $(C_0(X),\Lip,C_0(X),x)$ where $\Lip$ is the Lipschitz seminorm for $\mathsf{d}_X$. 

\begin{theorem}\label{classical-comparison-thm}
Let $\mathcal{C}$ be a class of {\pqpms s} which contains all classical pointed proper metric spaces. Let $\mathcal{T}$ be a class of tunnels, weakly appropriate for $\mathcal{C}$, and such that if $\mathds{X}$, $\mathds{Y}$ and $\mathds{Z}$ are three proper metric spaces such that there exists isometric embeddings $\iota_X : \mathds{X}\hookrightarrow\mathds{Z}$ and $\iota_Y:\mathds{Y}\hookrightarrow\mathds{Z}$, then $(\mathds{Z},\iota_X,(\mathds{X},x),\iota_Y,\mathds{Y},y) \in \mathcal{T}$ for any $x\in\mathds{X}$ and $y\in\mathds{Y}$.

If a sequence $(X_n,\mathsf{d}_n,x_n)_{n\in\N}$ of pointed proper metric spaces converges to a pointed proper metric space $(Y,\mathsf{d}_Y,y)$ for the Gromov-Hausdorff distance, then the sequence:
\begin{equation*}
\left(\mathds{X}_n\right)_{n\in\N} = \left(C_0(X_n),\Lip_n,C_0(X_n),x_n\right)_{n\in\N}
\end{equation*}
converges to $\mathds{Y} = (C_0(Y),\Lip,C_0(Y),y)$ for the topographic Gromov-Hausdorff propinquity $\propinquity{}$, where $\Lip_n$ and $\Lip$ are the Lipschitz seminorms on, respectively, $C_0(X_n)$ for $\mathsf{d}_n$ for any $n\in\N$, and $C_0(Y)$ for $\mathsf{d}_Y$.
\end{theorem}

\begin{proof}
We refer to \cite{Latremoliere12b} for the proof that the triples $(C_0(X_n),\Lip_n,C_0(X_n))$, for all $n\in\N$, and $(C_0(Y),\Lip,C_0(Y))$ are {\lcqms s}. It is then easy to check that they are in fact {\pqms s}, as required. We also check easily that isometric embeddings give rise not only to passage, but in fact, to tunnels (albeit possibly of very large extent).

Let $\varepsilon \in \left(0,\frac{1}{2}\right)$ and $r > 0$. Define, for all $n\in\N$:
\begin{equation*}
R_n = \inf\set{\mathsf{d}_n(x_n,x)}{x\in X,x\not\in\cBall{X_n}{\mathsf{d}_n}{x_n}{r}}
\end{equation*}
and
\begin{equation*}
R_Y = \inf\set{\mathsf{d}_Y(y,x)}{x\in Y,x\not\in\cBall{Y}{\mathsf{d}_Y}{y}{r}} \text{.}
\end{equation*}
Since $(X_n,\mathsf{d}_n,x_n)_{n\in\N}$ converges to $(Y,\mathsf{d}_Y,y)$, there exists $N \in \N$ such that for all $n\geq N$, there exists two isometric embeddings of $(X_n,\mathsf{d}_n)$ and $(Y,\mathsf{d}_Y)$ (which we will omit in our notations) into a proper metric space $(Z_n,\mathsf{D}_n)$ such that:
\begin{equation*}
\cBall{X_n}{}{x_n}{2r + R_Y + 1} \almostsubseteq{(Z,\mathsf{D}_n)}{\varepsilon} Y \text{ and }\cBall{Y}{}{y}{2r + R_Y + 1} \almostsubseteq{(Z,\mathsf{D}_n)}{\varepsilon} X_n \text{,}
\end{equation*}
while $\mathsf{D}_n(x_n,y)\leq\varepsilon$.

Let $\alpha \in (0,1)$ and $n\geq N$. There exists $x\in Y\setminus\cBall{Y}{}{y}{r}$ such that $\mathsf{d}_Y(y,x) \leq R_Y + \alpha$. Thus there exists $x' \in X_n$ such that $\mathsf{D}_n(x,x')\leq\varepsilon$. Consequently, we have:
\begin{equation*}
\begin{split}
\mathsf{d}_n(x_n,x') &\leq \mathsf{D}_n(x_n,y) + \mathsf{d}_Y(y,x) + \mathsf{D}_n(x,x')\\
&\leq 2\varepsilon + R_Y + \alpha\text{.}
\end{split}
\end{equation*}
Thus $R_n \leq 2\varepsilon + R_Y + \alpha$, and as $\alpha$ is arbitrary, we conclude that $R_n \leq 2\varepsilon + R_Y \leq R_Y + 1$.

Consequently, we have:
\begin{equation*}
\cBall{X_n}{\mathsf{d}_X}{x_n}{2r + R_n} \subseteq_\varepsilon Y\text{ and }\cBall{Y}{\mathsf{d}_Y}{y}{2r + R_Y} \subseteq_\varepsilon X\text{.}
\end{equation*}
Let $\mathsf{w}_n$ be the metric on $X_n\coprod Y$ given by Corollary (\ref{Lipschitz-lift-main-corollary}) (with $r+2\varepsilon$ in lieu of $r$). We now have:
\begin{enumerate}
\item the restriction of $\mathsf{w}_n$ to $X_n$ and $Y$ is given respectively by $\mathsf{d}_n$ and $\mathsf{d}_Y$,
\item moreover:
\begin{equation*}
\cBall{X_n}{\mathsf{d}_n}{x_n}{r} \almostsubseteq{(X_n\coprod Y,\mathsf{w}_n)}{\varepsilon} \cBall{Y}{\mathsf{d}_Y}{y}{r+4\varepsilon}\text{ and }\cBall{Y}{\mathsf{d}_Y}{y}{r} \almostsubseteq{(X_n\coprod Y,\mathsf{w}_n)}{\varepsilon} \cBall{X_n}{\mathsf{d_n}}{x_n}{r+4\varepsilon}\text{,}
\end{equation*}
\item if $t\in (0,r + 2\varepsilon]$ and if $f$ is a $1$-Lipschitz function on $X_n$ supported on $\cBall{X_n}{}{x_n}{t}$, then there exists a $1$-Lipschitz function $g$ on $X_n\coprod Y$ supported on $\cBall{X_n}{}{x_n}{t}\cup\cBall{Y}{}{y}{t+2\varepsilon}$ whose restriction to $X_n$ if $f$, and similarly if we reverse the roles of $X_n$ and $Y$.
\end{enumerate}

Let $K_n = \cBall{X_n}{}{x_n}{r+2\varepsilon}\cup\cBall{Y}{}{y}{r+2\varepsilon}$. We first observe that:
\begin{equation*}
K_n \subseteq_\varepsilon \cBall{X_n}{}{x_n}{r+4\varepsilon}
\end{equation*}
and
\begin{equation*}
K_n \subseteq_\varepsilon \cBall{Y}{}{y}{r+4\varepsilon}\text{.}
\end{equation*}

A standard argument, based on the convexity of the {\mongekant}, then shows that:
\begin{equation*}
\StateSpace\left[C_0\left(X_n\coprod Y_n\right)\middle\vert K_n\right] \almostsubseteq{(\StateSpace(C_0(X)\oplus\C_0(Y)),\Kantorovich{\mathsf{w}})}{\varepsilon} \StateSpace[C_0(X)|\Kantorovich{\mathsf{d}_X}]
\end{equation*}
and
\begin{equation*}
\StateSpace\left[C_0\left(X_n\coprod Y_n\right)\middle\vert K_n\right] \almostsubseteq{(\StateSpace(C_0(X)\oplus\C_0(Y)),\Kantorovich{\mathsf{w}})}{\varepsilon} \StateSpace[C_0(Y)|\Kantorovich{\mathsf{d}_Y}]
\end{equation*}
where $\Kantorovich{\mathsf{d}}$ for any metric $\mathsf{d}$ is the {\mongekant} $\Kantorovich{\Lip_{\mathsf{d}}}$ for the Lipschitz seminorm associated with $\mathsf{d}$.

Let $\pi_X : C_0(X\coprod Y)\twoheadrightarrow C_0(X)$ and $\pi_Y: C_0(X\coprod Y)\twoheadrightarrow C_0(Y)$ be the canonical surjections. We have established that:
\begin{equation*}
\left(C_0\left(X\coprod Y\right),\Lip_{\mathsf{w}},C_0\left(X\coprod Y\right),\pi_X,\mathds{X}_n,\pi_Y,\mathds{Y}\right)
\end{equation*}
is a $r$-tunnel with $(\varepsilon,\cBall{X}{}{x}{r+2\varepsilon}\cup\cBall{Y}{}{y}{r+2\varepsilon})$ as an $r$-admissible pair.

Consequently, $\propinquity{\mathcal{T}}(\mathds{X}_n,\mathds{Y})\leq\varepsilon$. This completes our proof.
\end{proof}

We pause to observe that if we defined the extent of tunnels in terms of pure states, rather than states, then the proof of Theorem (\ref{classical-comparison-thm}) and our work in the second section of this paper show that our modified hypertopology, when restricted to the class of classical pointed proper metric spaces, would agree with the Gromov-Hausdorff topology. We have made the standard choice, in noncommutative metric geometry, to work with more general states, as the pure state spaces may not be topologically well-behaved in general. This is a standard observation in noncommutative Gromov-Hausdorff theory (e.g. \cite{Rieffel00,Kerr02,Latremoliere13,Latremoliere13b}).

\subsection{Dual Gromov-Hausdorff Propinquity}

\begin{theorem}
A sequence $(\A_n,\Lip_n)$ of {\Lqcms s} converges to a {\Lqcms} $(\A,\Lip)$ for the dual propinquity if, and only if it converges for the topographic Gromov-Hausdorff hypertopology.
\end{theorem}

\begin{proof}
Let $(\A_n,\Lip_n)_{n\in\N}$ be a sequence of {\Lqcms s}. For each $n\in\N$, let $\mathds{A}_n$ be the canonically associated {\pqpms} for $(\A_n,\Lip_n)$. Similarly, let $(\A,\Lip)$ be a {\Lqcms} and $\mathds{A}$ the canonically associated {\pqpms}.

If $(\A_n,\Lip_n)_{n\in\N}$ converges to $(\A,\Lip)$ for the dual Gromov-Hausdorff propinquity, then by \cite{Latremoliere13b,Latremoliere14} for all $\varepsilon > 0$, there exists $N\in\N$ such that for all $n\geq N$, there exists a tunnel $\tau_n$ from $(\A_n,\Lip_n)$ to $(\A,\Lip)$ with $\chi(\tau_n)\leq\varepsilon$. For each $n\geq N$, by Proposition (\ref{compact-tunnel-prop}), we can associate to $\tau_n$ a $r$-tunnel, for any $r > 0$, with the same extent as $\tau_n$, from $\mathds{A}_n$ to $\mathds{A}$. Thus $\propinquity{r}(\mathds{A}_n,\mathds{A})\leq\varepsilon$ for all $n\geq N$. Thus $(\mathds{A}_n)_{n\in\N}$ converges to $\mathds{A}$ in the topographic Gro\-mov-Haus\-dorff topology.

Assume now that $(\mathds{A}_n)_{n\in\N}$ converges to $\mathds{A}$ for the topographic Gromov-Hausdorff topology. Let $r > 0$ and let $\varepsilon > 0$. There exists $N\in\N$ such that, for all $n\geq N$, there exists an $r$-tunnel $\tau_n$ from $\mathds{A}_n$ to $\mathds{A}$ with $\tunnelextent{\tau}{r}\leq\varepsilon$. 

Since $\mathds{A}_n$ is compact for all $n\in\N$ and so is $\mathds{A}$, we conclude that if $\tau_n=(\D,\Lip,\M,\pi,\mathds{A}_n,\rho,\mathds{A})$ then $\beta_n = (\D,\Lip,\pi,\rho)$ is a tunnel in the sense of \cite{Latremoliere13b,Latremoliere14}, with $\chi(\beta_n)=\tunnelextent{\tau_n}{r}\leq\varepsilon$. Thus $\Lambda^\ast((\A_n,\Lip_n),(\A,\Lip))\leq\varepsilon$ for all $n\geq N$ and our theorem is proven.
\end{proof}

\bibliographystyle{amsplain}
\bibliography{../thesis}
\vfill

\end{document}